\documentclass[aihp]{imsart}

\usepackage{amsthm,amsmath,amssymb,amsfonts}
\usepackage{enumerate}
\usepackage{graphicx}
\usepackage{float}
\usepackage{color,graphicx}
\RequirePackage[numbers]{natbib}
\RequirePackage[colorlinks,citecolor=blue,urlcolor=blue]{hyperref}

\startlocaldefs
\theoremstyle{plain}
\newtheorem{theorem}{Theorem}[section]
\newtheorem{corollary}[theorem]{Corollary}
\newtheorem{lemma}[theorem]{Lemma}
\newtheorem{proposition}[theorem]{Proposition}

\theoremstyle{remark}
\newtheorem{definition}[theorem]{Definition}
\newtheorem{remark}[theorem]{Remark}

\DeclareMathOperator{\supp}{supp}
\DeclareMathOperator{\law}{Law}
\DeclareMathOperator{\spann}{span}
\DeclareMathOperator{\divv}{div}
\DeclareMathOperator{\pr}{pr}
\DeclareMathOperator{\leb}{Leb}

\newcommand{\Q}{\mathbb{Q}}
\newcommand{\R}{\mathbb{R}}
\newcommand{\I}{\mathbb{I}}

\newcommand{\N}{\mathbb{N}}
\newcommand{\F}{\mathcal{F}}
\newcommand{\p}{\mathbb{P}}
\newcommand{\B}{\mathcal{B}}

\newcommand{\Pp}{\mathcal{P}}

\newcommand{\s}{\mathcal{S}}

\newcommand{\RR}{\mathcal{R}}
\newcommand{\E}{\mathbb{E}}
\newcommand{\h}{\mathcal{H}}

\newcommand{\eps}{\varepsilon}

\newcommand{\St}{\mathcal{S}^{\uparrow}}

\newcommand{\Li}{L_2^{\uparrow}}
\newcommand{\Cfb}{C^{\infty}_b}
\newcommand{\Cfo}{C^{\infty}_0}
\newcommand{\FC}{\mathcal{FC}}
\newcommand{\LL}{L_2(\Li(\xi),\Xi^{\xi})}

\newcommand{\D}{\mathrm{D}}

\newcommand{\e}{\mathcal{E}}
\newcommand{\Dom}{\mathbb{D}}

\newcommand{\K}{\mathcal{K}}
\newcommand{\HS}{\mathcal{L}_2}
\newcommand{\W}{\mathcal{W}}
\newcommand{\LHS}{L_2([0,T],\mathcal{L}_2(L_2))}
\newcommand{\Dr}{D^{\uparrow}}
\newcommand{\id}{\mathrm{id}}
\newcommand{\vect}{\mathbf}

\endlocaldefs

\begin{document}

\begin{frontmatter}

\title{Coalescing-fragmentating Wasserstein dynamics:\\ particle approach}
%\title{A sample article title with some additional note\thanksref{T1}}
\runtitle{CFWD: particle approach}
%\thankstext{T1}{The author is grateful to Max von Renesse for useful discussions and suggestions. The research was partly supported by Alexander von Humboldt Foundation and partly supported by the Deutsche Forschungsgemeinschaft (DFG, German Research Foundation) – SFB 1283/2 2021 – 317210226. The author thanks the anonymous referees for their careful reading of the manuscript and many valuable comments which improved the presentation of the results.}

\begin{aug}
  \author{\inits{V. Konarovskyi}\fnms{Vitalii} \snm{Konarovskyi}\ead[label=e1]{vitalii.konarovskyi@math.uni-bielefeld.de}}
%\author[B]{\inits{???}\fnms{???} \snm{???}\ead[label=e2,mark]{???@???}}
%\and
%\author[B]{\inits{???}\fnms{???} \snm{???}\ead[label=e3,mark]{???@???}}
%%%%%%%%%%%%%%%%%%%%%%%%%%%%%%%%%%%%%%%%%%%%%%
%% Addresses                                %%
%%%%%%%%%%%%%%%%%%%%%%%%%%%%%%%%%%%%%%%%%%%%%%
\address{Fakult\"{a}t f\"{u}r Mathematik, Universit\"{a}t Bielefeld, Germany \printead{e1}}
\address{Fakult\"{a}t f\"{u}r Mathematik und Informatik, Universit\"{a}t Leipzig, Germany}
\address{Institute of Mathematics of NAS of Ukraine, Kiev, Ukraine}
%\address[A]{???. \printead{e1}}

%\address[B]{???. \printead{e2,e3}}
\end{aug}

\begin{abstract}
We construct a family of semimartingales that describes the behavior of a particle system with sticky-reflecting interaction. The model is a physical improvement of the Howitt-Warren flow~\cite{Howitt2009}, an infinite system of diffusion particles on the real line that sticky-reflect from each other. But now particles have masses obeying the conservation law and the diffusion rate of each particle depends on its mass. The equation which describes the evolution of the particle system is a new type of equations in infinite-dimensional space and can be interpreted as an infinite-dimensional analog of the equation for sticky-reflected Brownian motion. The particle model appears as a particular solution to the corrected version of the Dean-Kawasaki equation. 
\end{abstract}

\begin{abstract}[language=french]
Nous construisons une famille de semimartingales décrivant le comportement d'un système de particules avec interactions à effet réflectif et adhésif. Le modèle est un perfectionnement plus physique du flot de Howitt-Warren~\cite{Howitt2009}, un système infini de particules diffusives sur la droite réelle interagissant avec effet réflectif et adhésif. Dans ce papier, les particules ont désormais des masses satisfaisant les lois de conservation et le coefficient de diffusion de chaque particule dépend de sa masse. L'équation décrivant l'évolution du système de particules est un nouveau type d'équation sur un espace de dimension infinie et peut être interprétée comme un analogue infini-dimensionnel de l'équation satisfaite par le mouvement brownien à comportement réflectif et adhésif. Le modèle particulaire apparaît comme une solution particulière d'une version corrigée de l'équation de Dean-Kawasaki.
\end{abstract}

%\begin{keyword}[class=MSC]
%\kwd[Primary ]{???}
%\kwd{???}
%\kwd[; secondary ]{???}
%\end{keyword}

%\begin{keyword}
%\kwd{???}
%\kwd{???}
%\end{keyword}

\end{frontmatter}

%%%%%%%%%%%%%%%%%%%%%%%%%%%%%%%%%%%%%%%%%%%%%%
%%%% Main text entry area:

\section{Introduction}
In~\cite{Konarovskyi_SR:2017}, the author together with von Renesse proposed a class of measure-valued processes, so-called \textit{reversible coalescing-fragmentating Wasserstein dynamics} or shortly \textit{reversible CFWD}, which describes the evolution of mass of particles that interact via some sticky-reflecting mechanism. The construction was aimed at the generalization of a Brownian motion of a single point (atom) to the case of infinite points (measures) on the real line. The main requirement of such a construction was that the process $\mu_t$ had to be reversible in time and its short time asymptotics had to be covered by the Varadhan formula of the form
\[
  \p\left\{ \mu_{t+\eps}=\nu \right\}\sim e^{ -\frac{ d^2_{\mathcal{W}}(\mu_t,\nu) }{ 2 \eps } }, \quad \eps\ll 1,
\]
where $d_{\mathcal{W}}$ denotes the usual Wasserstein distance on the space of probability measures $\Pp_2(\R)$ on the real line with a finite second moment. This led to a new family of measure-valued processes which are naturally connected with the Riemannian structure of the Wasserstein space of probabilities measures and also to a new class of associated invariant measures for those processes.

The reversible CFWD also solves the corrected Dean-Kawasaki equation\footnote{The Dean-Kawasaki equation is a prototype of equations appearing in fluctuating hydrodynamic theory and has a broad application in physics (see, e.g.,~\cite{MR2095422,doi:10.1063/1.4913746,Dean:1996,MR3744636,1742-5468-2014-4-P04004,doi:10.1063/1.4883520,0305-4470-33-15-101,Kawasaki199435,PhysRevE.89.012150,doi:10.1063/1.478705,0953-8984-12-8A-356,MR0462381,doi:10.1143/JPSJ.59.1299}). In~\cite{Konarovskyi:DK:2018,Konarovskyi:DKII:2018}, we showed that the original Dean-Kawasaki equation has either trivial solutions or is ill-posed.} 
\begin{equation} %Dean-Kawasaki equation
  \label{equ_dean-kawasaki_equation}
  d\mu_t=\Delta\mu_t^*dt+\divv(\sqrt{ \mu_t} dW_t )
\end{equation}
on $\Pp_2(\R)$, where $\mu_t^*=\sum_{ x \in \supp\mu_t }   \delta_{ x }$ and $dW$ is a space-time white noise. We remark that the measure $\mu_t$ is purely atomic with a finite number of atoms for almost all $t\geq 0$~\cite{Konarovskyi:TVP:2021}. Therefore, $\mu^*_t$ is well defined for almost all $t$. It is known that the modified massive Arratia flow satisfies the same equation (see~\cite{Konarovskyi_LDP:2015}). This in particular implies the non-uniqueness of solutions to~\eqref{equ_dean-kawasaki_equation}.

The construction in~\cite{Konarovskyi_SR:2017} was based on the Dirichlet form approach. There we proposed a new family of measures on the space $\Pp_2(\R )$ which depend on the interaction potential between particles and then proved an integration by parts formula. This allowed to introduce the naturally associated Dirichlet form $\e$ and construct the corresponding measure-valued process $\mu_t$, $t \in [0,\tau)$, (a family of processes that depend on the interacting potential between particles). In spite of the power of the Dirichlet form method, such a description has many shortcomings which make the model very complicated for further investigation. In particular,   
\begin{itemize}
  \item the process $\mu_t$, $t \in [0,\tau)$, was defined up to the life time $\tau$ and it is unclear in general if the process globally exists, i.e., if $\tau$ is infinite a.s.;

  \item $\mu_t$, $t \in [0,\tau)$, was defined only for initial distributions $\mu_0$ outside an unknown $\e$-exceptional set;

  \item although the process describes the evolution of the mass of interacting particles, one can say nothing about the behavior of individual particles;

  \item the construction does not cover the coalescent interaction between particles that can be considered as a critical case of sticky-reflecting behavior.
\end{itemize}

The present paper is aimed at the elimination of those defects. For this, we choose a completely different construction. We will approximate an infinite particle system by a finite number of particles. This allows us to construct a continuum collection of ordered continuous semimartingales on the real line which satisfy some natural properties. We also note that the obtained system can be considered as a physical improvement of the Howitt-Warren flow~\cite{Howitt2009,MR31557821} which describes the family of Brownian motions with sticky-reflected interaction. The inclusion of the particle mass into the system which influences their motion makes our model more interesting and natural from the physical point of view.

Although the particle model constructed here satisfies the same stochastic differential equation (see equation~\eqref{f_the_main_equation} below) as one constructed in~\cite{Konarovskyi_SR:2017} and has the same intuitive description, discussed in Section~\ref{sub:description_of_the_model_and_the_main_result}, it remains unclear if these two models coincide. The reason is that the uniqueness of solutions to the equation describing the particle systems remains a complicated open problem. However, we conjecture that this equation admits a unique weak solution.

\subsection{Description of the model and formulation of the main results}%
\label{sub:description_of_the_model_and_the_main_result}

We consider a family of diffusion particles on the real line which intuitively can be described as follows. Particles start from a set of points and move keeping their order. When particles collide, they coalesce and form clusters (sets of particles occupying the same positions). We assume that each particle has an ``infinitesimal'' mass and the mass of every cluster equals the total mass of its particles. All clusters fluctuate independently of each other until the moment of collision as Brownian motions with diffusion rates inversely proportional to their masses.  Forming a cluster, particles immediately experience a drift force defined by an interaction potential which makes particles leave the cluster.

Let us assume that the total mass of the system is finite. This assumption is needed to overcome some additional difficulties which can occur considering systems of infinite total muss. Moreover, we will assume that the total mass equals one for simplicity. The case of any finite total muss of the system can be obtained by the rescaling of the considered model. Next, we describe the dynamics more precisely. Let every particle in the system be labeled by a point $u$ from $[0,1]$ and its position at time $t\geq 0$ be denoted by $X(u,t)$. Since particles keep their order, we assume that $X(u,t)\leq X(v,t)$ for all $u<v$ and $t$. Denote the cluster containing particle $u$ by 
$$
\pi(u,t)=\{ v \in (0,1):\ \ X(u,t)=X(v,t) \}.
$$ 
We define the mass $m(u,t)$ of the cluster $\pi(u,t)$ at time $t$ as its length. For convenience, we will also call $m(u,t)$ the mass of particle $u$ at time $t$. According to our requirements, for every $u$ the process $X(u,\cdot )$ has to be a continuous semimartingale with quadratic variation whose derivative equals $ \frac{1}{ m(u,t) }$ at time $t$, that is,
\[
  d\left[ X(u,\cdot ) \right]_t= \frac{dt}{ m(u,t) }.
\]
Since we have assumed that particles move independently up to their collision, it would be reasonable to require that $X(u,t)$ and $X(v,t)$ are independent up to the meeting. The problem is that the processes always depend on each other via the mass of their clusters. Therefore, we replace the condition of independence by zero covariance\footnote{If particles would not change their diffusion rate then this condition would be equivalent to the independent motion of particles at the time when they occupy distinct positions.}
\[
  d\left[ X(u,\cdot ),X(v,\cdot ) \right]_t=0 \quad \mbox{provided} \quad X(u,t)\not= X(v,t).
\]
In order to define the splitting between particles, we prescribe a number $\xi(u)$ to each particle $u$, where $\xi$ is non-decreasing function. This number is called an \textit{interaction potential}  of particle $u$. Then particle $u$, which belongs to the cluster $\pi(u,t)$ at time $t$, has the drift force
$$
\xi(u)-\frac{1}{m(u,t)}\int_{\pi(u,t)}\xi(v)dv
$$
that is the difference between own potential and the average potential over the cluster. Summarizing the assumptions above, the family of process $X(u,\cdot )$, $u \in [0,1]$, formally has to solve the following system of equations
\begin{equation} %system of equation for particles
  \label{equ_system_of_equation_for_particles}
    dX(u,t)= \frac{1}{ m(u,t) }\int_{ \pi(u,t) }    W(dv,dt) +\left( \xi(u)-\frac{1}{m(u,t)}\int_{\pi(u,t)}\xi(v)dv \right)dt,
\end{equation}
$u \in [0,1]$, under the restriction $X(u,t)\leq X(v,t)$, $u<v$, $t\geq 0$, where $W$ is a Brownian sheet. We also provide~\eqref{equ_system_of_equation_for_particles} with the initial condition $X(u,0)=g(u)$.

Let $D([a,b],E)$ denote the Skorohod space of c\`{a}dl\'{a}g functions from $[a,b]$ to a Polish space $E$ with the usual Skorohod topology. We say that a function $f:[0,1]\to \R  $ is piecewise $\gamma$-H\"older continuous if there exists an ordered partition $U=\{u_i,\ i=1,\ldots,l\}$ of $[0,1]$ such that $f$ is $\gamma$-H\"older continuous on each interval $(u_{i-1},u_i)$, $i\in[l]:=\{1,\dots,l\}$. The first main result of the present paper reads as follows.

\begin{theorem}\label{theorem_existence_in_general_case}
  Let $g,\xi \in D([0,1],\R)$ be non-decreasing piecewise $\frac{1}{2}+$-H\"older continuous\footnote{Hereafter we mean that there exists $\eps>0$ such that the function is $ (\frac{1}{ 2 }+\eps)$-H\"older continuous.} functions on $[0,1]$. Then  there exists a random element $X=\{X(u,t),\  t\geq 0,\ u\in[0,1]\}$ in $D([0,1],C[0,\infty))$ such that 
\begin{enumerate}
  \item[$(R1)$] for each $u \in [0,1]$, $X(u,0)=g(u)$;
 
 \item[$(R2)$] for each $u<v$ from $[0,1]$ and $t\geq 0$, $X(u,t)\leq X(v,t)$;
 
 \item[$(R3)$] the process 
   \begin{align*}
     M^X(u,t):=  X(u,t)-g(u) -\int_0^t\left( \xi(u)-\frac{1}{m_{X}(u,s)}\int_{\pi_{X}(u,s)}\xi(v)dv \right)ds, \quad t\geq 0,
   \end{align*}
   is a continuous square-integrable $(\F_t^X)$-martingale for each $u\in(0,1)$ and a continuous local $(\F_t^X)$-martingale for each $u \in \{ 0,1 \}$, where $(\F_t^X)_{t\geq 0}$ is the natural filtration generated\footnote{see Section~\ref{section_preliminaries} and Remark~\ref{remark_filtration} for the precise definition} by $X$, $\pi_{X}(u,t):=\{v:\ X(u,t)=X(v,t)\}$ and $m_{X}(u,t)=\leb\pi_{X}(u,t)$;
 
 \item[$(R4)$] for each $u,v \in [0,1]$ the joint quadratic variation of $M^X(u,\cdot)$ and $M^X(v,\cdot)$ equals
 $$
 [M^X(u,\cdot),M^X(v,\cdot)]_t=\int_0^t\frac{\I_{\{X(u,s)=X(v,s)\}}}{m_{X}(u,s)}ds, \quad t\geq 0.
 $$
 
\end{enumerate}
\end{theorem} 

\begin{remark} %
The processes $M^X(0,\cdot )$ and $M^X(1,\cdot )$ are also square-integrable martingales under the assumptions of Theorem~\ref{theorem_existence_in_general_case}. This directly follows from the fact that they are continuous local martingales and the estimate of their quadratic variation in Corollary~\ref{corollary_estim_of_diff_rate_at_0_and_1}. We do not include this result into condition $(R3)$ to preserve the equivalence between $(R1)-(R4)$ and SDE~\eqref{f_the_main_equation} below (see Theorem~\ref{theorem_existence_in_general_case1} for more details).
\end{remark}

We note that the piecewise H\"older continuity of functions $g$ and $\xi$ is a technical assumption required by our argument (see Remark~\ref{rem_on_holder_continuity_of_initial_condition} below for more details), and we believe it can be removed.

The random element $X$ from Theorem~\ref{theorem_existence_in_general_case} can be interpreted as a weak solution to the system of equations~\eqref{equ_system_of_equation_for_particles}. In particular, for the coalescing particle system (if $\xi=0$), Marx in~\cite{Marx2017} showed that for any family of processes $X$ which satisfies $(R1)-(R4)$ there exists a Brownian sheet $W$ (possibly on an extended probability space) such that $X$ solves system of equations~\eqref{equ_system_of_equation_for_particles}. We believe that the same result can be obtained for any interaction potential $\xi$, using the same argument.

We would like to compare the model with the modified massive Arratia flow described by a system of continuous martingales on the real line which satisfies the same conditions with $\xi=0$~\cite{Konarovskyi:2014:arx,Konarovskyi:2017:EJP,Konarovskyi_LDP:2015}, see also~\cite{Arratia:1979,MR3433579,MR2329772,MR2671379,MR2094432,Le_Jan:2004,Riabov:2017,MR31557821} for the classical Arratia flow and the Brownian web, where particles coalesce and do not change their diffusion rate, and~\cite{MR31557821,Schertzer:2017,Sun:2008} for the Brownian net, which is the massless version of the flow constructed here. The main difference between the constructed particle system and the modified massive Arratia flow is an additional drift potential which leads to the dispersion of particles and makes the model very complicated for construction. Moreover, methods proposed there cannot be applied to the sticky-reflected particle system. In the pictures, computer simulations of both systems are given.

  \begin{figure}[H]%zz
      \centering
      \includegraphics[width=8.2cm,height=4.8cm]{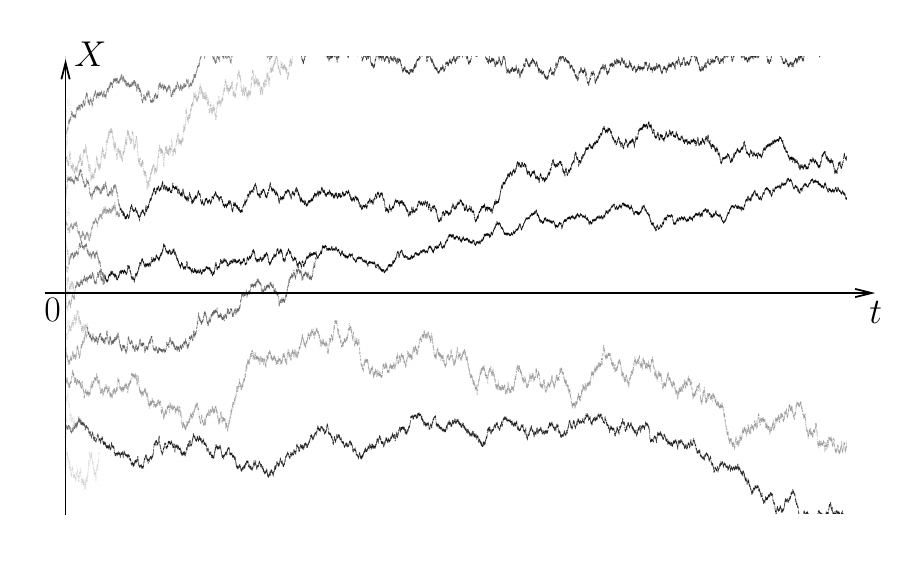}
      \includegraphics[width=8.2cm,height=4.8cm]{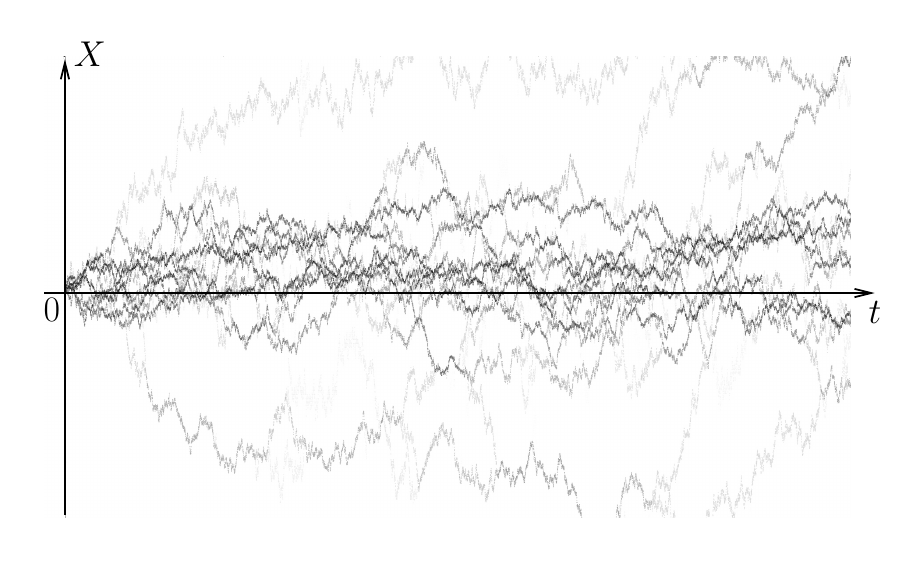}
      \caption {The modified massive Arratia flow (left), where particles started from every point of an interval, and the sticky-reflected particle system (right) with interacting potential $\xi$ which equals the identity function, where all particles started at 0. Grayscale colour coding is illustrating the atom sizes.}
  \end{figure}

  \begin{figure}[H]%zz
      \centering
      \includegraphics[width=8.2cm,height=4.8cm]{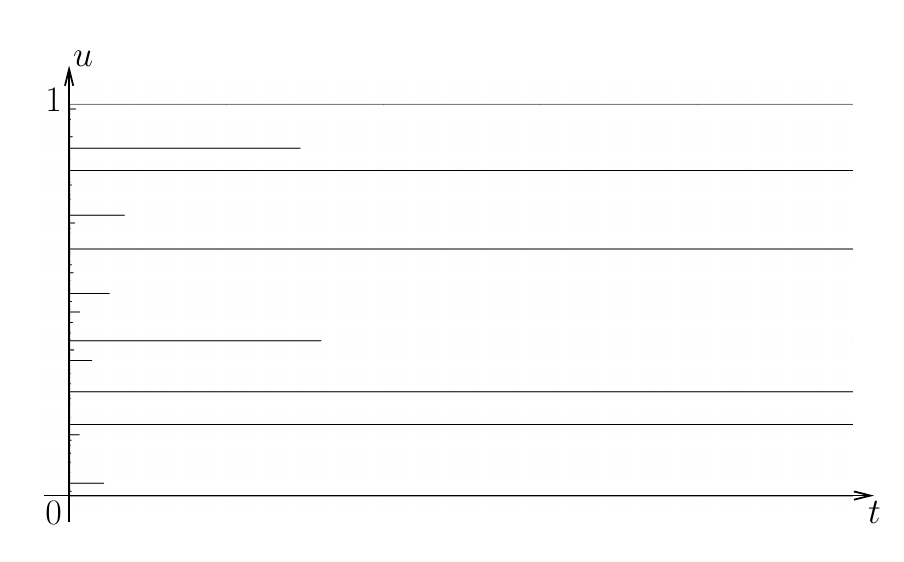}
      \includegraphics[width=8.2cm,height=4.8cm]{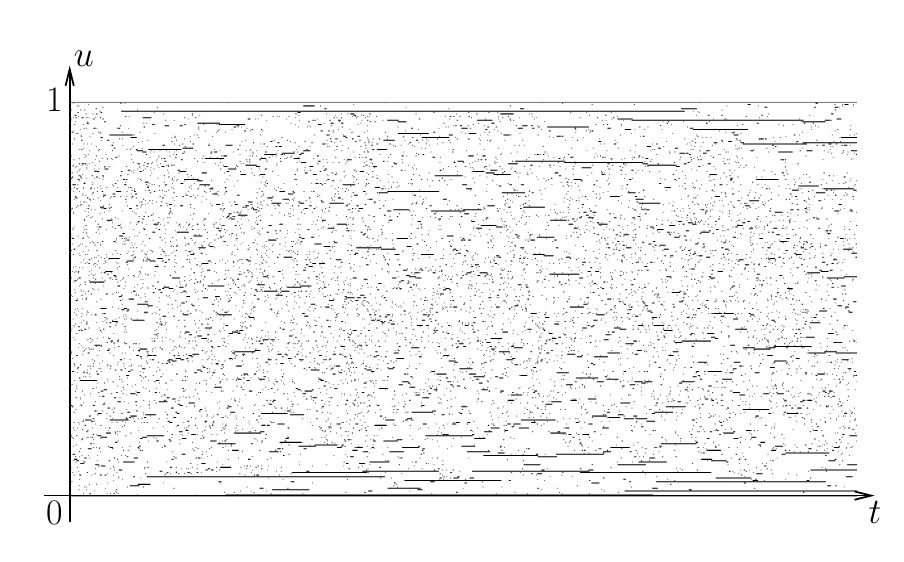}
      \caption {The clusters behaviour of the modified massive Arratia flow (left) and the sticky-reflected particle system (right),  where at every time $t$ dots represent the ends of clusters, i.e., the ends of the intervals $\{ v:\ X(u,t)=X(v,t) \}$, $u \in [0,1]$.}
  \end{figure}

In order to construct the family of processes $X$, we use the approximation of the model by finite particle systems. We first state some estimates for the evolution of particle masses in Section~\ref{sec:a_priory_estimates}. It allows proving the tightness. The main problem is to check that the limiting system of processes satisfies properties $(R1)-(R4)$. To show this, we replace system of equations~\eqref{equ_system_of_equation_for_particles} with an equation in some Hilbert space that has discontinuous coefficients and prove that the new equation has solutions. After that, we show the connection between solutions to the new equation and system~\eqref{equ_system_of_equation_for_particles}.  

For $p \in [1,\infty]$ let $L_p^{\uparrow}$ denote the space of non-decreasing $p$-integrable (with respect to the Lebesgue measure on $[0,1]$ denoted by $\leb$) functions from $[0,1]$ to $\R$, and $\pr_f$ be the projection in $L_2:=L_2([0,1],\leb)$ on the linear subspace $L_2(f)$ of $\sigma(f)$-measurable functions. Let also $W_t$, $t \geq 0$, be a cylindrical Wiener process on $L_2$. System of equations~\eqref{equ_system_of_equation_for_particles} can be rewritten as one SDE
\begin{equation}\label{f_the_main_equation}
dX_t=\pr_{X_t}dW_t+(\xi-\pr_{X_t}\xi)dt,\quad X_0=g
\end{equation} 
in the space $\Li$ due to the form of the projection operator, where $X_t=X(\cdot ,t)\in \Li$. The second contribution of the present paper is the development of new methods for solving equation~\eqref{f_the_main_equation}, and the establishing of a connection between solutions to such an equation and families of semimartingales satisfying $(R1)-(R4)$. We remark that equation~\eqref{f_the_main_equation} can be interpreted as an infinite-dimensional analog of the equation for a sticky-reflected Brownian motion on the half-line
\[
  dx(t)=\I_{\left\{ x(t)>0 \right\}}dw(t)+\lambda\I_{\left\{ x(t)=0 \right\}}dt
\]
for which the question of the existence and uniqueness of solutions is non-trivial (see, e.g.,~\cite{Engelbert:2014}). In our case, the uniqueness of solutions to~\eqref{f_the_main_equation} remains an open problem.

\begin{theorem}\label{theorem_existence_in_general_case1}
\begin{enumerate}
\item[(i)] For each $\delta>0$, $g\in L^{\uparrow}_{2+\delta}$ and $\xi\in  L^{\uparrow}_{\infty}$ there exists a weak solution\footnote{see Definition~\ref{definition_solution}} to SDE \eqref{f_the_main_equation}.

\item[(ii)] Let $Y=\{Y(u,t),\ u\in[0,1],\ t\geq 0\}$ be a random element in the Skorohod space $D([0,1],C[0,\infty))$ and $X_t$, $t \geq 0$, be a continuous process in $\Li$ such that for every $t\geq 0$ almost surely $X_t=Y(\cdot ,t)$ in $L_2$ and $\E \|X_t\|_{L_2}^2<\infty$. Then the random element $Y$ satisfies $(R1)-(R4)$ if and only if the process $X_t$, $t \geq 0$, is a weak solution to~\eqref{f_the_main_equation}.
\end{enumerate}
\end{theorem} 

We call a weak solution to equation~\eqref{f_the_main_equation} a {\it coalescing-fragmentating Wasserstein dynamics} or shortly a {\it CFWD}. According to Theorem~\ref{theorem_existence_in_general_case1}~(ii), a random element in the Skorohod space $D([0,1],C[0,\infty))$ satisfying $(R1)-(R4)$ is also called a CFWD.

Theorem~\ref{theorem_existence_in_general_case} will immediately follow from Theorem~\ref{theorem_existence_in_general_case1} and the existence of a solution to~\eqref{f_the_main_equation} with a modification from the Skorohod space $D([0,1],C[0,\infty))$ (see Section~\ref{sub:cfwd_as_a_family_of_semimartingales_proof_of_theorem_theorem_existence_in_general_case_}).

Next we briefly describe the main idea of proof of Theorem~\ref{theorem_existence_in_general_case1}. The first part of the theorem is proved using a finite particle approximation. We first construct a solution to equation~\eqref{f_the_main_equation} if $\xi$ and $g$ are step functions, using the Dirichlet form approach. This corresponds to the case of a finite particle system. Then we approximate any $\xi$ and $g$ by step functions and show that solutions to~\eqref{f_the_main_equation} are tight and every limiting process solves equation~\eqref{f_the_main_equation}. The tightness argument is based on the control of the particle mass, and is rather standard. We recall that, in the case of the modified massive Arratia flow (if $\xi=0$), the tightness followed from the estimate $$\p\{m(u,t)<r\}\leq \frac{C\sqrt{r}}{\sqrt{t}}(g(u+r)-g(u))$$ \cite[Lemma~4.1]{Konarovskyi:2017:EJP}, which can be proved using the coalescing of particles. Now, particles do not coalesce.  But we can control the integral $\int_0^t\p\{m(u,s)<r\}ds$ (see lemmas~\ref{lemma_estim_of_mass},~\ref{lemma_estim_of_mass_near_0} and~\ref{lemma_estim_of_mass_near_1}). This is enough for the tightness in Section~\ref{section_tightness_results}.

A very complicated problem is to check that a limiting process satisfies SDE~\eqref{f_the_main_equation}. For the modified massive Arratia flow, we showed this, using the fact that the number of distinct particles at each positive time is finite and decreases as time increases because particles coalesce (see Theorem~5.5~\cite{Konarovskyi:2017:EJP}). In the sticky-reflected case of interaction, the model also consists of a finite number of distinct particles (clusters) for almost all times. However, an infinite number of distinct particles can appear.  More precisely, with probability 1 the system admits an infinite number of distinct particles on a dense subset of the time interval if and only if $\xi$ takes an infinite number of values~\cite[Theorem~1.1]{Konarovskyi:TVP:2021}. Therefore, we cannot use the methods which work for the modified massive Arratia flow.

 Let us roughly explain a new approach which we propose in order to show that a limiting process solves~\eqref{f_the_main_equation}. Let for every $n\geq 1$, $X^n$ be a solution to~\eqref{f_the_main_equation} with initial condition $g_n$ and interacting potential $\xi_n$. Then the $L_2$-valued martingale part $M^{X_n}$ of $X_n$ has the quadratic variation process $\langle\langle M^{X_n}\rangle\rangle_t=\int_{ 0 }^{ t } \pr_{X^n_s}ds $, $s\geq 0$, which is a continuous operator-valued process (for more details see Definition~\ref{definition_solution}). We assume that $\{X^n,\ n\geq 1\}$ converges to $X$ and the quadratic variations $\left\{\langle\langle M^{X_n}\rangle\rangle,\ n\geq 1\right\}$ to $\int_0^{\cdot}P_sds$, where $P_s$, $s\geq 0$, is an oparator-valued process. For the identification of the limit, it is needed to prove that $P_s=\pr_{X_s}$ for almost all $s$. Since $X^n$, $n\geq 1$, are continuous semimartingales, $X$ also is a continuous semimartingale with quadratic variation $\int_0^{\cdot}P_sds$. In order to show that $P_s=\pr_{X_s}$, we use the following trick. By the lower semi-continuity of the map $g\mapsto\|\pr_gh\|_{L_2}$ (see Lemma~\ref{lemma_lower_semicontinuity_of_pr}) and the fact that $\pr_{X^n_t}$ is a projection, it is possible to show that $P_t$ is also a projection but maybe on a larger space than $L_2(X_t)$. Next, let $Z_t$, $t\geq 0$, be a continuous $\Li$-valued martingale with quadratic variation $\int_0^{t}L_sL^{*}_sds$, $t\geq 0$, where $L_t$, $t\geq 0$, is an adapted operator-valued process. We prove in Proposition~\ref{proposition_prop_of_quad_var_of_mart} that $L_t\circ\pr_{Z_t}=L_t$ for almost all $t$. This immediately implies $P_t=P_t\circ\pr_{X_t}=\pr_{X_t}$. The proposed method will also work for a wider class of SDEs on $\Li$ with discontinuous coefficients. Recently, in~\cite{Konarovskyi:SPDE:2021}, we also adapted this approach to another SPDE with discontinuous coefficients, which we call a sticky-reflected stochastic heat equation.  Note that Proposition~\ref{proposition_prop_of_quad_var_of_mart} seems to be of independent interest.

\subsection{Preliminaries and notation}\label{section_preliminaries}

We will denote the set of non-decreasing c\`{a}dl\'{a}g functions from $(0,1)$ to $\R$ by $D^{\uparrow}$. The set of all step functions from $D^{\uparrow}$ with a finite number of jumps is denoted by $\St$. If $g\in D^{\uparrow}$ is bounded, then we set
$$
g(0)=\lim_{u\downarrow 0}g(u)\quad\mbox{and}\quad g(1)=\lim_{u\uparrow 1}g(u).
$$

Let $(E,\F,P)$ be a complete probability space and $\h\subset\F$. Then $\sigma^{*}(\h)$ denotes the $P$-completion of $\sigma(\h)$. If $g:E\to\R$ is an $\F$-measurable function, then $\sigma^{*}(g):=\sigma^{*}(\{g^{-1}(A):\ A\in\B(\R)\})$, where $\B(F)$ denotes the Borel $\sigma$-field on a topological space $F$. 

\begin{remark}\label{remark_def_of_sigma_g}
We note that $g_1=g_2$ $P$-a.e. implies $\sigma^{*}(g_1)=\sigma^{*}(g_2)$.
\end{remark}

For $p\in[1,+\infty]$ we denote the space of $p$-integrable (essential bounded, if $p=+\infty$) functions (more precisely equivalence classes) from $[0,1]$ to $\R$ by $L_p$. The usual norm in $L_p$ is denoted by $\|\cdot\|_{L_p}$ and the usual inner product in $L_2$ by $(\cdot,\cdot)_{L_2}$. 

For a Borel measurable function $g:(0,1)\to\R$ the space of all $\sigma^{*}(g)$-measurable functions from $L_2$ is denoted by $L_2(g)$. By Remark~\ref{remark_def_of_sigma_g}, $L_2(g)$ is well-defined for every equivalence class $g$ from $L_p$.

Let $\HS(L_2)$ denote the space of Hilbert-Schmidt operators on $L_2$ with the inner product given by
\begin{equation}\label{f_HS_inner_product}
(A,B)_{HS}=\sum_{i=1}^{\infty}(Ae_i,Be_i)_{L_2},\quad A,B\in\HS(L_2),
\end{equation}
where $\{e_i,\ i\in\N\}$ is an orthonormal basis of $L_2$. We note that the inner product does not depend on the choice of basis $\{e_i,\ i\in\N\}$. The corresponding norm in $\HS(L_2)$ is denoted by $\|\cdot\|_{HS}$.

If $H$ is a Hilbert space with the inner product $(\cdot,\cdot)_H$, then $L_2([0,T],H)$ will denote the Hilbert space of 2-integrable $H$-valued functions on $[0,T]$ endowed with the inner product 
$$
(f,g)_{T,H}=\int_0^T(f_t,g_t)_Hdt,\quad f,g\in L_2([0,T],H).
$$
The corresponding norm is denoted by $\|\cdot\|_{T,H}$. If $H=\HS(L_2)$, then the inner product and the norm will be denoted by $(\cdot,\cdot)_{T,HS}$ and $\|\cdot\|_{T,HS}$, respectively.

Let $C(I,E)$ denote the space of continuous functions from $I\subset\R$ to a Banach space $E$ equipped with the topology of uniform convergence on compacts. For simplicity we also write $C[0,T]$ (resp. $C[0,\infty)$) instead of $C([0,T],\R)$ (resp. $C([0,\infty),\R )$). The uniform norm in $C[0,T]$ is denoted by $\|\cdot\|_{C[0,T]}$.

The set of all infinitely differentiable real-valued functions on $\R^m$ with all partial derivatives bounded is denoted by $\Cfb(\R^m)$ and $\Cfo(\R^m)$ is the set of functions from $\Cfb(\R^m)$ with compact support.

Let $D([a,b],E)$ denote the space of c\`{a}dl\'{a}g functions from $[a,b]$ to a Polish space $E$ with the usual Skorohod distance (see, e.g., Section~3~\cite{Billingsley:1999}\footnote{In contrast to the definition of the Skorohod space in~\cite{Billingsley:1999}, we additionally assume that each function from $D([a,b],E)$ is continuous at $b$.} and Section~\ref{section_compact_sers_in_D}). 

The Lebesgue measure on $\R$ will be denoted by $\leb$.
 
The set of functions from $L_p$ which have a non-decreasing modification is denoted by $L_p^{\uparrow}$. By Proposition~A.1~\cite{Konarovskyi:2017:EJP}, $\Li$ is a closed set in $L_2$ and each $f\in\Li$ has a unique modification from $D^{\uparrow}$. Therefore, considering an element from $\Li$ as a function, we will always take its modification from $D^{\uparrow}$. We also set $L_p^{\uparrow}(g):=L_p^{\uparrow}\cap L_p(g)$ for every Borel measurable function $g:(0,1) \to \R $ and $p \in [1,+\infty]$.
 
For $g\in L_2^{\uparrow}$ we denote the projection operator in $L_2$ on the closed linear subspace $L_2(g)$ by $\pr_g$. Let $\#g$ denote the number of distinct points of the set $\{g(u),\ u\in(0,1)\}$, where the modification $g$ is taken from $D^{\uparrow}$. We will prove in Section~\ref{section_properties_of_pr} (see Lemma~\ref{lemma_connection_HS_with_int} there) that $\# g=\|\pr_g\|_{HS}^2$.

\begin{remark}\label{remark_cadlag_modification_of_pr}
  Since $\pr_g$ maps $L_2^{\uparrow}$ into $L_2^{\uparrow}$ (see, e.g., Lemma~\ref{lemma_monotonisity_of_pr} below), for every $\xi\in L_2^{\uparrow}$ and $u\in(0,1)$ we will understand $\left(\pr_g\xi\right)(u)$ as a value of the function $f\in D^{\uparrow}$ at $u$, where $\pr_g\xi=f$ a.e., and 
$$
\left(\pr_g\xi\right)(0)=\lim_{u\downarrow 0}f(u)\quad\mbox{and}\quad\left(\pr_g\xi\right)(1)=\lim_{u\uparrow 1}f(u),
$$
if the limits exist.
\end{remark}

We denote the filtration generated by a process $X_t$, $t\geq 0$, by $(\F^{\circ,X}_t)_{t\geq 0}$, that is, $\F^{\circ,X}_t=\sigma(X_t,\ s\leq t)$, $t\geq 0$. The smallest right-continuous and complete extension of $(\F^{\circ,X}_t)_{t\geq 0}$ is denoted by $(\F^X_t)_{t\geq 0}$ (see, e.g., Lemma~7.8~\cite{Kallenberg:2002} for existence). The filtration $(\F^X_t)_{t\geq 0}$ is called the {\it natural filtration} generated by $X$. 
%If the process $X$ is defined only on $[0,T]$, then $\F_T^X$ will denote the smallest complete extension of $\sigma(X_t,\ t\in[0,T])$. 

\begin{remark}\label{remark_filtration}
If $X_t$, $t\geq 0$, is an $L_2$-valued process and $\{Y(u,t),\ u\in[0,1],\ t\geq 0\}$ is a random element in $D([0,1],C[0,\infty))$ such that $X_t=Y(\cdot,t)$ in $L_2$ a.s. for all $t\geq 0$, then $(\F_t^X)_{t \geq 0}$ coincides with the smallest right-continuous and complete extension of the filtration
$$
\big(\sigma(Y(u,s),\ u\in[0,1],\ s\leq t)\big)_{t \geq 0}.
$$
This can be proved using, e.g., Lemma~\ref{lemma_molification} below.
\end{remark}

%Let $x(t)$, $t\geq 0$, be an $(\F_t)$-adapted right-continuous process. We remark that $x$ is an $(\F_t)$-martingale if and only if it is a martingale with respect to the smallest right-continuous and complete extension of $(\F_t)$. This easily follows from Lemma~7.8~\cite{Kallenberg:2002} and the right-continuity of $x$. %??
Now we give a definition of weak solution to equation~\eqref{f_the_main_equation}.

\begin{definition}\label{definition_solution}
An $\Li$-valued random process $X_t$, $t\geq 0$, is called a {\it weak solution to SDE}~\eqref{f_the_main_equation}
if
\begin{enumerate}
\item[$(E1)$] $X_0=g$;

\item[$(E2)$] $X\in C([0,\infty),\Li)$;

\item[$(E3)$] $\E\|X_t\|_{L_2}^2<\infty$ for all $t\geq 0$;

\item[$(E4)$] the process 
$$
M^X_t:=X_t-g-\int_0^t(\xi-\pr_{X_s}\xi)ds,\quad t\geq 0,
$$
is a continuous square-integrable $(\F^X)$-martingale\footnote{see Section~2.1.3~\cite{Gawarecki:2011} for the introduction to martingales in a Hilbert space} in $L_2$ with quadratic variation process 
$$
\langle\langle M^X \rangle\rangle_t=\int_0^t\pr_{X_s}ds, \quad t\geq 0.
$$
\end{enumerate}
\end{definition}

\begin{remark}\label{remark_def_of_sol}
\begin{enumerate}
\item[(i)] The process 
$$
A^X_t:=\int_0^t(\xi-\pr_{X_s}\xi)ds,\quad t\geq 0,
$$
is continuous in $L_2$.

\item[(ii)] Condition $(E4)$ is equivalent to
\begin{enumerate}
\item[$(E'4)$] For each $t\geq 0$ $\E\|M^X_t\|_{L_2}^2<\infty$ and for each $h\in L_2$ the process
$$
(M^X_t,h)_{L_2}=(X_t,h)_{L_2}-(g,h)_{L_2}-\int_0^t(\xi-\pr_{X_s}\xi,h)_{L_2}ds,\quad t\geq 0,
$$
is a continuous square-integrable $(\F^X)$-martingale with quadratic variation  
$$
[(M^X_{\cdot},h)_{L_2}]_t=\int_0^t\|\pr_{X_s}h\|_{L_2}^2ds, \quad t\geq 0.
$$
\end{enumerate}
\item[(iii)] For each $t\geq 0$ $\E\|X_t\|_{L_2}^2<\infty$ provided $\E\|M^X_t\|_{L_2}^2<\infty$, since $\|A_t^X\|_{L_2}\leq 2\|\xi\|_{L_2}t$.

\item[(iv)] Using the same argument as in the proof of Lemma~2.1~\cite{Gawarecki:2011}, one can show that the increasing process of $M^X$ is given by 
$$
\langle M^X \rangle_t=\int_0^t\|\pr_{X_s}\|_{HS}^2ds,\quad t\geq 0,
$$
that is, 
$$
\|M^X_t\|_{L_2}^2-\int_0^t\|\pr_{X_s}\|_{HS}^2ds,\quad t\geq 0,
$$
is an $(\F^X)$-martingale. In particular, $\E\|M^X_t\|_{L_2}^2=\E\int_0^t\|\pr_{X_s}\|_{HS}^2ds<\infty$ for all $t\geq 0$.

\item[(v)] If $X$ is a weak solution to SDE~\eqref{f_the_main_equation}, then there exists a cylindrical Wiener process $W_t,\ t\geq 0$, in $L_2$ (maybe on an extended probability space) such that
$$
X_t=g+\int_0^t\pr_{X_s}dW_s+\int_0^t(\xi-\pr_{X_s}\xi)ds,\quad t\geq 0,
$$
by Corollary~2.2~\cite{Gawarecki:2011}.
\end{enumerate} 
\end{remark}

\subsection{Organisation of paper}

In Section~\ref{sec:a_priory_estimates}, a priori estimates for a CFWD as a family of semimartingales satisfying $(R1)-(R4)$ are obtained. We will use them in Section~\ref{section_tightness_results} in order to prove the tightness of a sequence of CFWDs in the space of continuous $L_2$-valued paths (Section~\ref{sub:tightness_of_weak_solutions}) and in the Skorohod space $D([0,1],C[0,\infty))$ (Section~\ref{subsection_tightness_in_Skorohod_space}). Using the tightness results, a CFWD for a general initial condition and a general interaction potential will be constructed as a limit of finite particle systems in Section~\ref{construction_of_cfwd}. We first recall the construction of a reversible CFWD via the Dirichlet form approach in Section~\ref{subsection_Dirichlet_form_approach}, which is needed for the description of a finite sticky-reflect particle system. Then we show the existence of a weak solution to SDE~\eqref{f_the_main_equation} in Section~\ref{sub:existence_of_solutions_to_sde_f_the_main_equation_proof_of_theorem_theorem_existence_in_general_case_i_}. This gives the proof of Theorem~\ref{theorem_existence_in_general_case1}~(i). The equivalence between two descriptions of the CFWD (proof of the second part of Theorem~\ref{theorem_existence_in_general_case1}) is stated in Section~\ref{sub:equivalence_between_two_formulations_of_solutions}. After this, we show that there exists a version of CFWD from the Skorohod space $D([0,1],C[0,\infty))$ that immediately implies the statement of Theorem~\ref{theorem_existence_in_general_case}. Auxiliary statements are given in the appendix. In particular, an estimate of the sitting time at zero of a positive continuous semimartingale is obtained in Section~\ref{sub:the_sitting_time_at_zero}. Properties of projection operators appearing as coefficients of SDE~\eqref{f_the_main_equation} are studied in Section~\ref{section_properties_of_pr}. In Section~\ref{subsection_limit_prop_of_proj_process}, we realize the idea described in the introduction which is used in Section~\ref{sub:identification_of_the_limit} in order to show that a limit of CFWDs is again a CFWD.  The property of the quadratic variation of continuous $\Li$-valued martingales mentioned in the introduction and also needed for the identification of the limit of CFWDs are proved in Section~\ref{subsection_prop_of_semimartingales}.

\section{A priori estimates for CFWD}\label{sec:a_priory_estimates}
The goal of this section is to get some a priori estimates of particle masses in a CFWD which we will use for the proof of the existence of CFWD. During this section we assume that $g,\xi \in D^{\uparrow}$ are fixed and $\{X(u,t),\ u\in[0,1],\ t\geq 0\}$ is a random element in $D([0,1],C[0,\infty))$ defined on a complete probability space $(\Omega,\F,\p)$ and satisfies $(R1)-(R4)$. We recall that $(\F_t^X)_{t\geq 0}$ coincides with the smallest right-continuous and complete extension of the filtration
$$
\big(\sigma(X(u,s),\ u\in[0,1],\ s\leq t)\big)_{t\geq 0},
$$
by Remark~\ref{remark_filtration}.

For notational convenience, we set $\F_t:=\F_t^X$, $m(u,t):=m_{X}(u,t)$ and $M(u,t):=M^X(u,t)$, $u\in[0,1]$, $t\geq 0$, where $m_X$ and $M^X$ were defined in Theorem~\ref{theorem_existence_in_general_case}.

\subsection{Coalescing properties}\label{sub:coalescing_properties}
In this section, we show that particles with the same interaction potential coalesce.

\begin{lemma}\label{lemma_coalescing}
If $\xi(u)=\xi(v)$ for some $u,v\in[0,1]$, then 
\begin{equation} %coalescing
  \label{equ_coalescing}
  \p\left\{X(u,t)=X(v,t)\ \ \mbox{implies}\ \ X(u,t+s)=X(v,t+s)\ \ \forall s\geq 0\right\}=1.
\end{equation}
\end{lemma}

\begin{proof}
  We assume that $u,v$ belong to $(0,1)$ and $u>v$. By $(R2)$, $(R3)$ and Lemma~\ref{lemma_monotonisity_of_pr}, 
\begin{align*}
X(u,t)-X(v,t)&=M(u,t)-M(v,t)+g(u)-g(v)
-\int_0^t\left[\left(\pr_{X_s}\xi\right)(u)-\left(\pr_{X_s}\xi\right)(v)\right]ds,\quad t\geq 0,
\end{align*}
is a continuous positive supermartingale, since $\left(\pr_{X_s}\xi\right)(u)-\left(\pr_{X_s}\xi\right)(v)\geq 0$, $s\geq 0$. Thus, coalescing property~\eqref{equ_coalescing} follows from Proposition~II.3.4~\cite{Revuz:1999}. If $u \in \{0,1\}$ or $v \in \{0,1\}$, then equality~\eqref{equ_coalescing} can be easily obtained from the continuity of the map $u\mapsto X(u,\cdot )$ at $0$ and $1$. Indeed, let for instance $u=1>v>0$. Then for an arbitrary sequence $\{ u_n,\ n\geq 1 \}$ from $(v,u)$ which increases to $u$ one has that $\xi(u_n)=\xi(v)$, by the monotonicity of $\xi$. Hence, almost surely $X(u_n,s)=X(v,s)$ for all $s\geq \tau_{v,u_n}=\inf\left\{ t:\ X(v,t)=X(u_n,t) \right\}$ and $n\geq 1$. Passing to the limit as $n\to\infty$ and using the fact that $\tau_{v,u}\geq \tau_{v,u_n}$ for all $n\geq 1$, we get~\eqref{equ_coalescing}. This completes the proof of the lemma.
\end{proof}

\begin{corollary}\label{corollary_coalescing_from_start}
  If $\xi(u)=\xi(v)$ and $g(u)=g(v)$ for some $u,v\in[0,1]$, then $X(u,\cdot)=X(v,\cdot)$ a.s. Moreover, if $g,\xi \in \St$, then there exists a partition $\{\pi_k,\ k\in[n]\}$ of $[0,1]$ and a system of continuous processes $\{x_k(t),\ t\geq 0,\ k\in[n]\}$ such that almost surely
  $$
    X(u,t)=\sum_{k=1}^nx_k(t)\I_{\pi_k}(u),\quad u\in[0,1],\ \ t\geq 0.
  $$
\end{corollary}

\begin{proof}
The first part of the corollary immediately follows from Lemma~\ref{lemma_coalescing}.  To prove the second part, we first note that the functions $\xi$ and $g$ (from $\St$) can be written as 
$$
\xi=\sum_{k=1}^n\varsigma_k\I_{\pi_k}\quad\mbox{and}\quad g=\sum_{k=1}^ny_k\I_{\pi_k},
$$
for some partition $\{\pi_k,\ k\in[n]\}$ of $[0,1]$, $\varsigma_k\leq\varsigma_{k+1}$ and $y_k\leq y_{k+1}$, $k\in[n-1]$. Hence, taking $x_k(\cdot):=X(u_k,\cdot)$, for some $u_k\in\pi_k$, the needed equality follows from the first part of the corollary and $(R2)$. The corollary is proved.
\end{proof}

\begin{remark} %finite particle system
  \label{rem_finite_particle_system}
  If $g,\xi \in \St$, then the family of semimartingales $X(u,\cdot )$, $u \in [0,1]$, is described by a finite number of processes $x_k(t)$, $t\geq 0$, according to Corollary~\ref{corollary_coalescing_from_start}. In this case, we will talk about a finite number of particles whose evolution is described by $x_k(t)$, $t\geq 0$, $k \in [n]$.
\end{remark}

%zzz\begin{corollary}\label{corollary_estim_of_m_below}
%For each $u\in[0,1]$ there exists a non-random constant $\delta_u>0$ such that $\inf_{t\geq 0}m(u,t)\geq\delta_u$ a.s. 
%\end{corollary}
%
%\begin{proof}
%Taking $\delta_u=\sup\{v_2-v_1:\ g(v_1)=g(u)=g(v_2)\ \mbox{and}\ \xi(v_1)=\xi(u)=\xi(v_2)\}$, the inequality easily follows from Corollary~\ref{corollary_coalescing_from_start}.
%\end{proof}

\subsection{Estimation of the mass of internal particles}%
\label{sub:mass_estimates_of_interior_particles}

In this section, we will obtain some estimates for masses of particles in CFWD. The key idea will be to replace the event $\left\{ m(u,t)<r \right\}$ that the mass of the particle labeled by $u$ is less than $r$ by the event that the particles with labels $u-r$, $u$ and $u+r$ belong to distinct clusters at time $t$. Therefore, the estimates obtained here will be better if $u$ is close to the middle of the interval. The case where $u$ is close to the ends of the interval will be considered in the next section. 

We introduce the following function
\begin{equation}\label{f_function_G}
 \begin{split}
 G(r_1,r_2,u,t)&:=2(g(u+r_2)-g(u))(g(u)-g(u-r_1))\\
 &+2(\xi(u)-\xi(u-r_1))\left[t(g(u+r_2)-g(u))+\frac{t^2}{2}(\xi(u+r_2)-\xi(u))\right]\\
 &+2(\xi(u+r_2)-\xi(u))\left[t(g(u)-g(u-r_1))+\frac{t^2}{2}(\xi(u)-\xi(u-r_1))\right].
 \end{split}
\end{equation}

\begin{lemma}\label{lemma_estim_of_mass}
 For each $u\in(0,1)$, $0<r< u\wedge(1-u)$ and $t\geq 0$ the inequality
 $$
 \int_0^t\p\{m(u,s)<r\}ds\leq rG(r,r,u,t)
 $$
  holds.
\end{lemma}

\begin{proof}
We fix $u$, $r$ as in the assumption of the lemma and denote
\begin{align*}
 Z_+(t):=X(u+r,t)-X(u,t),\quad
 Z_-(t):=X(u,t)-X(u-r,t)
\end{align*}
for all $t\geq 0$.

Then $Z_+$ and $Z_-$ can be written as follows
\begin{align*}
 Z_+(t)&=z_++N_+(t)+\int_0^tb_+(s)ds,\\
 Z_-(t)&=z_-+N_-(t)+\int_0^tb_-(s)ds,
\end{align*}
for all $t\geq 0$, where
\begin{align*}
 z_+&:=g(u+r)-g(u),\quad z_-:=g(u)-g(u-r),\\
 b_+(t)&:=\xi(u+r)-\xi(u)-\left[\left(\pr_{X_t}\xi\right)(u+r)-\left(\pr_{X_t}\xi\right)(u)\right],\\
 b_-(t)&:=\xi(u)-\xi(u-r)-\left[\left(\pr_{X_t}\xi\right)(u)-\left(\pr_{X_t}\xi\right)(u-r)\right]
\end{align*}
and the square-integrable martingales $N_+$, $N_-$ are defined as $Z_+$ and $Z_-$ with $X$ replaced by $M$. 

Since the projection of a non-decreasing function is also non-decreasing \ (see Lemma~\ref{lemma_monotonisity_of_pr}), we have that
\begin{equation}\label{f_estim_of_b}
 b_+(t)\leq\xi(u+r)-\xi(u)\quad \mbox{and}\quad  b_-(t)\leq\xi(u)-\xi(u-r)
\end{equation}
for all $t\geq 0$.

Next, using $(R4)$, we evaluate the joint quadratic variation of $N_+$ and $N_-$. For $t\geq 0$ we have
\begin{align*}
 [N_+,N_-]_t&=[M(u+r,\cdot)-M(u,\cdot),M(u,\cdot)-M(u-r,\cdot)]_t\\
 &=\int_0^t\left[\frac{\I_{\{Z_+(s)=0\}}}{m(u,s)}+\frac{\I_{\{Z_-(s)=0\}}}{m(u,s)}-\frac{1}{m(u,s)}-\frac{\I_{\{X(u+r,s)=X(u-r,s)\}}}{m(u,s)}\right]ds\\
 &= -\int_0^t\frac{\I_{\{Z_+(s)>0,Z_-(s)>0\}}}{m(u,s)}ds.
\end{align*}
Thus, It\^o's formula implies
\begin{align*}
 Z_+(t)Z_-(t)&=z_+z_-+\int_0^tZ_+(s)dN_-(s)+\int_0^tZ_-(s)dN_+(s)\\
 &+\int_0^tZ_+(s)b_-(s)ds+\int_0^tZ_-(s)b_+(s)ds-\int_0^t\frac{\I_{\{Z_+(s)>0,Z_-(s)>0\}}}{m(u,s)}ds.
\end{align*}
Note that the stochastic integrals in the expression above are martingales. This directly follows from the fact that $Z_+,Z_-,N_+$ and $N_-$ are square-integrable martingales, the Burkholder-Davis-Gundy inequality (see, e.g., Theorem~26.12~\cite{Kallenberg:2002}) and the inequality 
\[
  \E \sqrt{ \int_{ 0 }^{ t } Z_{\pm}^2(s)d[ N_{\mp}]_s }\leq \E \sqrt{ Z_{\pm}^*(t)[N_{\mp}]_t  }\leq \sqrt{ \E Z_{\pm}^*(t) } \sqrt{ \E [N_{\mp}]_{t} }<\infty,
\]
where $Z_{\pm}^*(t):=\sup\limits_{ s \in [0,t] }Z_{\pm}^2(s)$.
Taking the expectation, we obtain
\begin{equation}\label{f_estim_of_m_and_product}
\begin{split}
 \E Z_+(t)Z_-(t)&+\E\int_0^t\frac{\I_{\{Z_+(s)>0,Z_-(s)>0\}}}{m(u,s)}ds\\
 &=z_+z_-+\E\int_0^tZ_+(s)b_-(s)ds+\E\int_0^tZ_-(s)b_+(s)ds.
\end{split}
\end{equation}

Next, we estimate the right hand side of the obtained equality, using estimates~\eqref{f_estim_of_b}. We get
\begin{align*}
 \E&\int_0^tZ_+(s)b_-(s)ds\leq (\xi(u)-\xi(u-r))\int_0^t\E Z_+(s)ds\\
 &\qquad\qquad=(\xi(u)-\xi(u-r))\int_0^t\left[z_++\E\int_0^sb_+(s_1)ds_1\right]ds\\
 &\qquad\qquad\leq(\xi(u)-\xi(u-r))\left[z_+t+\frac{t^2}{2}(\xi(u+r)-\xi(u))\right].
\end{align*}
Similarly,
\begin{align*}
 \E\int_0^tZ_-(s)b_+(s)ds\leq(\xi(u+r)-\xi(u))\left[z_-t+\frac{t^2}{2}(\xi(u)-\xi(u-r))\right].
\end{align*}
We also remark that
$$
\frac{1}{m(u,t)}\I_{\{Z_+(t)>0,Z_-(t)>0\}}\geq\frac{1}{2r}\I_{\{Z_+(t)>0,Z_-(t)>0\}},
$$
by the definition of $m(u,t)$.
Consequently, we obtain
\begin{align*}
\frac{1}{2r}\E\int_0^t\I_{\{Z_+(s)>0,Z_-(s)>0\}}ds\leq\E\int_0^t\frac{\I_{\{Z_+(s)>0,Z_-(s)>0\}}}{m(u,s)}ds\leq \frac{1}{2}G(r,r,u,t),
\end{align*}
due to~\eqref{f_estim_of_m_and_product} and the fact that $Z_+(t)Z_-(t)\geq 0$. Thus,
\begin{align*}
 \int_0^t\p\{m(u,s)<r\}ds&\leq\int_0^t\p\{Z_+(s)>0,Z_-(s)>0\}ds =\E\int_0^t\I_{\{Z_+(s)>0,Z_-(s)>0\}}ds\leq rG(r,r,u,t).
\end{align*}
The lemma is proved.
\end{proof}

\begin{corollary}\label{corollary_estim_of_diff_rate_at_u}
 For each $\beta>0$, $u\in(0,1)$ and $t>0$ the following estimate is true
 $$
 \E\int_0^t\frac{1}{m^{\beta}(u,s)}ds\leq \frac{t}{(u\wedge(1-u))^{\beta}}+\beta\int_0^{u\wedge(1-u)}\frac{1}{r^{\beta}}G(r,r,u,t)dr,
 $$
 where $G$ is defined by~\eqref{f_function_G}.
\end{corollary}
 
\begin{proof}
 By Lemma~3.4~\cite{Kallenberg:2002} and Lemma~\ref{lemma_estim_of_mass}, we have
 \begin{align*}
  \E\int_0^t\frac{1}{m^{\beta}(u,s)}ds&=\int_0^t\E\frac{1}{m^{\beta}(u,s)}ds=\beta\int_0^t\left(\int_0^{\infty}r^{\beta-1}\p\left\{\frac{1}{m(u,s)}>r\right\}dr\right)ds\\
  &=\beta\int_0^t\left(\int_0^{\infty}r^{\beta-1}\p\left\{m(u,s)<\frac{1}{r}\right\}dr\right)ds\\
  &\leq\beta\int_0^{\frac{1}{u\wedge(1-u)}}\left(\int_0^t r^{\beta-1}ds\right)dr +\beta\int_{\frac{1}{u\wedge(1-u)}}^{\infty}r^{\beta-1}\left(\int_0^t\p\left\{m(u,s)<\frac{1}{r}\right\}ds\right)dr\\
  &\leq \frac{t}{(u\wedge(1-u))^{\beta}}+\beta\int_{\frac{1}{u\wedge(1-u)}}^{\infty}r^{\beta-1}\frac{1}{r}G\left(\frac{1}{r},\frac{1}{r},u,t\right)dr\\
  &=\frac{t}{(u\wedge(1-u))^{\beta}}+\beta\int_0^{u\wedge(1-u)}\frac{1}{r^{\beta}}G(r,r,u,t)dr.
 \end{align*}
 The lemma is proved.
\end{proof}

We will finish this section with the statement that is needed for the tightness argument in Section~\ref{section_tightness_results} and follows from equality~\eqref{f_estim_of_m_and_product} obtained in the proof of Lemma~\ref{lemma_estim_of_mass}.

\begin{lemma}\label{lemma_three_points}
For each $T>0$, $u\in(0,1)$, $r_1\in(0,u)$, $r_2\in(0,1-u)$ and $\lambda>0$
\begin{align*}
\p\left\{\|X(u+r_2,\cdot)-X(u,\cdot)\|_{C[0,T]}>\lambda,\ \|X(u,\cdot)-X(u-r_1,\cdot)\|_{C[0,T]}>\lambda\right\}\leq\frac{1}{2\lambda^2}G(r_1,r_2,u,T).
\end{align*}
\end{lemma}

\begin{proof}
Let $Z_{+}$ and $b_{+}$ be defined as in the proof of Lemma~\ref{lemma_estim_of_mass} with $r$ replaced by $r_2$, and $Z_{-}$ and $b_{-}$ with $r$ replaced by $r_1$. Let
$$
\sigma^{\pm}:=\inf\{t:\ Z_{\pm}(t)\geq \lambda\}\wedge T
$$
and
$$
Z_{\pm}^{\sigma^{\pm}}(t):=Z_{\pm}(\sigma^{\pm}\wedge t),\quad t\in[0,T].
$$
Then, by Theorem~17.5~\cite{Kallenberg:2002}, Proposition~17.15 ibid. and~\eqref{f_estim_of_m_and_product}, for each $t\geq 0$
\begin{align*}
 \E Z_+^{\sigma^+}(t)Z_-^{\sigma^-}(t)&+\E\int_0^{t\wedge\sigma^+\wedge\sigma^-}\frac{\I_{\{Z_+(s)>0,Z_-(s)>0\}}}{m(u,s)}ds\\
 &=z_+z_-+\E\int_0^{t\wedge\sigma^-}Z_+^{\sigma^+}(s)b_-(s)ds+\E\int_0^{t\wedge\sigma^+}Z_-^{\sigma^-}(s)b_+(s)ds.
\end{align*}
Similarly to the proof of Lemma~\ref{lemma_estim_of_mass}, we get
$$
\E Z_+^{\sigma^+}(T)Z_-^{\sigma^-}(T)\leq\frac{1}{2}G(r_1,r_2,u,T).
$$
Next, we note that $Z_+^{\sigma^+}(T)Z_-^{\sigma^-}(T)\geq \lambda^2\I_{\{\sigma^+\vee\sigma^-<T\}}$. Therefore,
\begin{align*}
&\p\left\{\|X(u+r_2,\cdot)-X(u,\cdot)\|_{C[0,T]}>\lambda,\ \|X(u,\cdot)-X(u-r_1,\cdot)\|_{C[0,T]}>\lambda\right\}\\
&\qquad\qquad\leq\p\{\sigma^+\vee\sigma^-<T\}\leq\frac{1}{\lambda^2}\E Z_+^{\sigma^+}(T)Z_-^{\sigma^-}(T)\leq\frac{1}{2\lambda^2}G(r_1,r_2,u,T).
\end{align*}
The lemma is proved.
\end{proof}

\subsection{Estimation of the mass of external particles}%
\label{sub:estimation_of_the_mass_of_external_particles}

In this section,  we will estimate particle masses whose labels are close to the ends of the interval $[0,1]$. The obtained inequalities will be weaker than ones from the previous section.

\begin{lemma}\label{lemma_estim_of_mass_near_0}
 For each $\alpha\in(0,1)$ and $t\geq 0$ there exists a constant $C=C(\alpha,t)$ such that for all $r\in(0,1)$ and $u\in[0,r)$ satisfying $r+u\leq 1$ we have
 $$
 \int_0^t\p\{m(u,s)<r\}ds\leq C e^{C(\xi(1)-\xi(0))^2}\left(\sqrt{u+r}\right)^{\alpha}G_0^{\alpha}(r,u,t),
 $$
 where 
 \begin{equation}\label{f_def_of_G_0}
  G_0(r,u,t)=(\xi(u+r)-\xi(u))t+g(u+r)-g(u).
 \end{equation}
 
\end{lemma}

\begin{proof}
 Let $r\in(0,1)$ and $u\in[0,r)$ be fixed. We set
 $$
 Z(t):=X(u+r,t)-X(u,t),\quad t\geq 0,
 $$
 and note that $m(u,t)<r$ implies $Z(t)>0$. In order to prove the lemma, we need to estimate the expectation $\E\int_0^t\I_{\{Z(s)>0\}}ds$.
 
 Let us rewrite $Z$ as follows
 \begin{equation}\label{f_equation_for_Z}
 Z(t)=z_0+N(t)+\int_0^tb(s)ds,\quad t\geq 0,
 \end{equation}
 where 
 \begin{align*}
   z_0:&=g(u+r)-g(r),\quad b(t):=\xi(u+r)-\xi(u)-\left[\left(\pr_{X_t}\xi\right)(u+r)-\left(\pr_{X_t}\xi\right)(u)\right]
 \end{align*}
 and $N$ is a continuous local $(\F_t)$-martingale with quadratic variation
 \begin{align*}
 [N]_t&=[M(u+r,\cdot)]_t+[M(u,\cdot)]_t-2[M(u+r,\cdot),M(u,\cdot)]_t\\
 &=\int_0^t\left(\frac{1}{m(u+r,s)}+\frac{1}{m(u,s)}-\frac{2\I_{\{Z(s)=0\}}}{m(u,s)}\right)ds, \quad t\geq 0.
 \end{align*}
 We note that $Z(t)>0$ implies $m(u,t)=\leb\{v:\ X(u,t)=X(v,t)\}<u+r$. Thus, 
 \begin{equation}\label{f_variation_of_N}
 [N]_t=\int_0^ta(s)^2\I_{\{Z(s)>0\}}ds, \quad t\geq 0,
 \end{equation}
 where $a(t):=\left(\frac{1}{m(u+r,t)}+\frac{1}{m(u,t)}\right)^{\frac{1}{2}}\vee\frac{1}{\sqrt{r+u}}\geq \frac{1}{\sqrt{r+u}}$ for any $t\geq 0$. 
 
 Next, we are going to use the Girsanov theorem in order to remove $\int_0^tb(s)\I_{\left\{ Z(s)>0 \right\}}ds$ from the drift $\int_{ 0 }^{ t } b(s)ds $ in~\eqref{f_equation_for_Z}. Since the processes $Z$, $a$, $N$, $b$ are functionals of $X(\cdot ,t)$, $t\geq 0$, which can be considered as a random element of $C([0,\infty),D^{\uparrow})$, without loss of generality, we may assume that $\Omega=C([0,\infty),D^{\uparrow})$, $\p=\law\{X\}$, $X(u ,t,\omega)=\omega(u,t)$, $t\geq 0$, $u \in [0,1]$, $\F$ is the completion of the Borel $\sigma$-field in $C([0,\infty),D^{\uparrow})$ and $(\F_t)_{t\geq 0}$ is the right-continuous and complete induced filtration. By Theorem~2.7.1'~\cite{Watanabe:1981:en} and~\eqref{f_variation_of_N}, there exists a Brownian motion $w(t)$, $t\geq 0$, on an extended probability space $(\widehat{\Omega},\widehat{\F},\widehat{\p})$ with respect to an extended filtration $(\widehat{\F}_t)_{t\geq 0}$  such that
 $$
 N(t)=\int_0^ta(s)\I_{\{Z(s)>0\}}dw(s),\quad t\geq 0.
 $$
 Moreover, we can take $\widehat{\Omega}=C([0,\infty),D^{\uparrow}\times\R)$ and $(\widehat{\F}_t)_{t\geq 0}$ to be the right-continuous and complete induced filtration on $\widehat{\Omega}$.  Let 
 $$
 U_t:=-\int_0^t\frac{b(s)}{a(s)}dw(s),\quad t\geq 0,
 $$
 and 
 $$
 B(t):=w(t)-[w,U]_t=w(t)+\int_0^t\frac{b(s)}{a(s)}ds,\quad t\geq 0.
 $$
 Then, by Novikov's theorem and Lemma~18.18~\cite{Kallenberg:2002}, there exists a probability measure $Q$ on $\widehat{\Omega}$ such that
 $$
 dQ=\exp\left\{U_t-\frac{1}{2}\int_0^t\frac{b(s)^2}{a(s)^2}ds\right\}d\widehat{\p}\quad \mbox{on}\ \ \widehat{\F}_t
 $$
 for all $t\geq 0$. Using the Girsanov theorem, we have that $B(t)$, $t\geq 0$, is a Brownian motion on the probability space $(\widehat{\Omega},\widehat{\F},Q)$ and
 \begin{align*}
  Z(t)&=z_0+\int_0^ta(s)\I_{\{Z(s)>0\}}dw(s)+\int_0^tb(s)ds\\
  &=z_0+\int_0^ta(s)\I_{\{Z(s)>0\}}dB(s)+\int_0^tb(s)\I_{\{Z(s)=0\}}ds\\
  &=z_0+\int_0^ta(s)\I_{\{Z(s)>0\}}dB(s)+(\xi(u+r)-\xi(u))\int_0^t\I_{\{Z(s)=0\}}ds.
 \end{align*}
 
 Next, we will consider the process $Z$ on the probability space $(\widehat{\Omega},\widehat{\F},Q)$ and estimate $\E^{Q}\int_0^t\I_{\{Z(s)>0\}}ds$, where the expectation $\E^{Q}$ is taken with respect to the measure $Q$. Set 
 $$
 Y(t):=\sqrt{u+r}Z(t),\quad t\geq 0.
 $$
 It is easily seen that
 $$
 Y(t)=y_0+\int_0^t\rho(s)\I_{\{Y(s)>0\}}dB(s)+\xi_0\int_0^t\I_{\{Y(s)=0\}}ds, \quad t\geq 0,
 $$
 where $y_0:=\sqrt{u+r}z_0=\sqrt{u+r}(g(u+r)-g(u))$, $\rho(t):=\sqrt{u+r}a(t)\geq 1$, $t\geq 0$, and  $\xi_0:=\sqrt{u+r}(\xi(u+r)-\xi(u))$. Using the inequality
 \begin{align*}
   \rho(t)&=\sqrt{u+r}a(t)\geq 1
 \end{align*}
 for all $t\geq 0$, and Proposition~\ref{proposition_estom_of_sitting_time}, we get for every $t\geq 0$ 
 $$
 \E^QR_t\leq\sqrt{\frac{2t}{\pi}}(\xi_0t+y_0),
 $$
 where
 $$
 R_t:=\int_0^t\I_{\{Y(s)>0\}}ds=\int_0^t\I_{\{Z(s)>0\}}ds,\quad t\geq 0.
 $$

 Now, we can estimate $\E R_t=\E\int_0^t\I_{\{Z(s)>0\}}ds$. Since $R_t$, $t\geq 0$, is an $(\widehat{\F}_t)$-measurable, for each $p,q>1$ satisfying $\frac{1}{p}+\frac{1}{q}=1$ we have 
 \begin{align*}
 \E R_t&=\E^Q\exp\left\{-U_t+\frac{1}{2}\int_0^t\frac{b(s)^2}{a(s)^2}ds\right\}R_t \leq\left(\E^Q\exp\left\{-pU_t+\frac{p}{2}\int_0^t\frac{b(s)^2}{a(s)^2}ds\right\}\right)^{\frac{1}{p}}\left(\E^QR_t^q\right)^{\frac{1}{q}}\\
 &\leq\left(\E\exp\left\{(1-p)U_t-\frac{1-p}{2}\int_0^t\frac{b(s)^2}{a(s)^2}ds\right\}\right)^{\frac{1}{p}}t^{\frac{q-1}{q}}\left(\E^QR_t\right)^{\frac{1}{q}}
 \end{align*}
 for all $t\geq 0$. In the inequality above, we have applied Jensen's inequality to  $R_t^q=\left(\int_0^t\I_{\{Z(s)>0\}}ds\right)^q$. Since $\frac{b(t)^2}{a(t)^2}\leq(\xi(1)-\xi(0))^2(u+r)\leq (\xi(1)-\xi(0))^2$ for all $t\geq 0$ and 
 $$
 \exp\left\{(1-p)U_t-\frac{(1-p)^2}{2}\int_0^t\frac{b(s)^2}{a(s)^2}ds\right\},\quad t\geq 0,
 $$ 
 is a positive martingale with expectation 1, we have
 \begin{align*}
  &\E\exp\left\{(1-p)U_t-\frac{1-p}{2}\int_0^t\frac{b(s)^2}{a(s)^2}ds\right\}\\
  &\qquad\qquad\leq e^{\frac{ tp(p-1) }{ 2 }(\xi(1)-\xi(0))^2}\E\exp\left\{(1-p)U_t-\frac{(1-p)^2}{2}\int_0^t\frac{b(s)^2}{a(s)^2}ds\right\}\\
  &\qquad\qquad=e^{\frac{ tp(p-1) }{ 2 }(\xi(1)-\xi(0))^2}.
 \end{align*}
 Thus,
 $$
 \E R_t\leq t^{\frac{q-1}{q}}e^{\frac{ t(p-1) }{ 2 }(\xi(1)-\xi(0))^2}\left(\sqrt{\frac{2t}{\pi}}(\xi_0t+y_0)\right)^{\frac{1}{q}}.
 $$
 Taking $q=\frac{1}{\alpha}$, we obtain for every $t\geq 0$ 
 $$
 \int_0^t\p\{m(u,s)<r\}\leq\E R_t\leq C e^{C(\xi(1)-\xi(0))^2}\left(\sqrt{u+r}\right)^{\alpha}G_0^{\alpha}(r,u,t),
 $$
 with some constant $C$ depending on $t$ and $\alpha$. The lemma is proved. 
\end{proof}

Applying Lemma~\ref{lemma_estim_of_mass_near_0} to the process $\{-X(1-u,t),\ u\in[0,1],\ t\geq 0\}$, we obtain a similar result for $u$ near 1.

\begin{lemma}\label{lemma_estim_of_mass_near_1}
  For each $\alpha\in(0,1)$ and $t\geq 0$ there exists a constant $C=C(\alpha,t)$ such that for all $r\in(0,1)$ and $u\in(1-r,1]$ satisfying $u-r\geq 0$ we have
 $$
 \int_0^t\p\{m(u,s)<r\}ds\leq C e^{C(\xi(1)-\xi(0))^2} (\sqrt{1-u+r})^{\alpha}G_1^{\alpha}(r,u,t),
 $$
 where 
\begin{equation}\label{f_def_of_G_1} 
 G_1(r,u,t):=(\xi(u)-\xi(u-r))t+g(u)-g(u-r).
\end{equation}
\end{lemma}

\begin{corollary}\label{corollary_estim_of_diff_rate_at_0_and_1}
 For every $\alpha\in(0,1)$, $\beta>0$, $u\in\{0,1\}$ and $t>0$ the estimate
 $$
 \E\int_0^t\frac{1}{m^{\beta}(u,s)}ds\leq t+\beta C e^{C(\xi(1)-\xi(0))^2}\int_0^1\frac{1}{r^{\beta+1-\frac{\alpha}{2}}}G_u(r,u,t)dr,
 $$
 holds, where $G_0$ and $G_1$ are defined in lemmas~\ref{lemma_estim_of_mass_near_0} and~\ref{lemma_estim_of_mass_near_1} and $C=C(\alpha,t)$ is the same constant as in Lemma~\ref{lemma_estim_of_mass_near_0}.
\end{corollary}

\begin{proof}
 Using lemmas~\ref{lemma_estim_of_mass_near_0} and~\ref{lemma_estim_of_mass_near_1}, the corollary can be proved in the same way as Corollary~\ref{corollary_estim_of_diff_rate_at_u}.
\end{proof}

\subsection{\texorpdfstring{$L_p$}{Lp}-norm of CFWD}%
\label{sub:_l_p_norm_of_cfwd}

The aim of this section is to combine estimates from previous two sections in order to get a bound of $\E \sup\limits_{ s \in [0,t] }\|X(\cdot ,s)\|_{L_{2+\delta}}^{2+\delta}$ for some $\delta>0$. We will start from the following proposition.

\begin{proposition}\label{proposition_estim_of_dif_rate}
 For each $p>2$ and $0<\beta<\frac{3}{2}-\frac{1}{p}$ there exists a constant $C=C(p,\beta,t)$ such that
 $$
 \E\int_0^1\int_0^t\frac{duds}{m^{\beta}(u,s)}\leq C e^{C(\xi(1)-\xi(0))^2}(1+\|g\|_{L_p}^3+\|\xi\|_{L_p}).
 $$
\end{proposition}

\begin{proof}
 By Lemma~3.4~\cite{Kallenberg:2002}, we have
 \begin{align*}
  \E\int_0^1\int_0^t\frac{duds}{m^{\beta}(u,s)}&= \beta\int_0^1\int_0^t\int_0^{\infty}r^{\beta-1}\p\left\{\frac{1}{m(u,s)}>r\right\}dudsdr\\
  &=\beta\int_0^2r^{\beta-1}\left(\int_0^1\int_0^t\p\left\{m(u,s)<\frac{1}{r}\right\}duds\right)dr\\
  &+\beta\int_2^{\infty}r^{\beta-1}\left(\int_0^{\frac{1}{r}}\int_0^t\p\left\{m(u,s)<\frac{1}{r}\right\}duds\right)dr\\
  &+\beta\int_2^{\infty}r^{\beta-1}\left(\int_{\frac{1}{r}}^{1-\frac{1}{r}}\int_0^t\p\left\{m(u,s)<\frac{1}{r}\right\}duds\right)dr\\
  &+\beta\int_2^{\infty}r^{\beta-1}\left(\int_{1-\frac{1}{r}}^1\int_0^t\p\left\{m(u,s)<\frac{1}{r}\right\}duds\right)dr=:I_1+I_2+I_3+I_4.
 \end{align*}
 
 The first integral $I_1\leq 2^{\beta}t$, since trivially $\p\left\{m(u,s)<\frac{1}{r}\right\}\leq 1$.
 
 Next, using Lemma~\ref{lemma_estim_of_mass_near_0}, we estimate $I_2$. Let $\alpha\in(0\vee(2\beta-2),1)$ be fixed and satisfy $1+\frac{\alpha}{2}-\frac{\alpha}{p}>\beta$. Then for each $r':=\frac{1}{r}\leq\frac{1}{2}$
 \begin{align*}
  \int_0^{r'}\int_0^t\p\left\{m(u,s)<r'\right\}duds &\leq C_1\int_0^{r'}\left(\sqrt{u+r'}\right)^{\alpha}\left[(\xi(u+r')-\xi(u))t+g(u+r')-g(u)\right]^{\alpha}du\\
  &\leq C_1t^{\alpha}\int_0^{r'}\left(\sqrt{u+r'}\right)^{\alpha}(\xi(u+r')-\xi(u))^{\alpha}du\\
  &+C_1\int_0^{r'}\left(\sqrt{u+r'}\right)^{\alpha}(g(u+r')-g(u))^{\alpha}du,
 \end{align*}
 where $C_1=C_1(\alpha,t):=C(\alpha,t)e^{C(\alpha,t)(\xi(1)-\xi(0))^2}$. We can easily estimate the first term as follows
 $$
 C_1t^{\alpha}\int_0^{r'}\left(\sqrt{u+r'}\right)^{\alpha}(\xi(u+r')-\xi(u))^{\alpha}du\leq 2^{\frac{\alpha}{2}}C_1t^{\alpha}(r')^{1+\frac{\alpha}{2}}(\xi(1)-\xi(0))^{\alpha},
 $$
 using $\xi(u+r')-\xi(u)\leq\xi(1)-\xi(0)$ and $u+r'\leq 2r'$. The second term will be estimated using H\"older's inequality. For $q$ satisfying $\frac{1}{p}+\frac{1}{q}=1$ we have
 \begin{align*}
  &C_1\int_0^{r'}\left(\sqrt{u+r'}\right)^{\alpha}(g(u+r')-g(u))^{\alpha}du\leq C_1(2r')^{\frac{\alpha}{2}}\int_0^{r'}(g(u+r')-g(u))^{\alpha}du\\
  &\qquad\qquad\leq C_1(2r')^{\frac{\alpha}{2}}(r')^{1-\alpha}\left[\int_0^{r'}(g(u+r')-g(u))du\right]^{\alpha}\\
  &\qquad\qquad\leq 2^{\frac{\alpha}{2}}C_1(r')^{1-\frac{\alpha}{2}}\left[\int_0^1\left(\I_{[r',2r']}(u)-\I_{[0,r']}(u)\right)g(u)du\right]^{\alpha}\\
  &\qquad\qquad\leq 2^{\frac{\alpha}{2}}C_1(r')^{1-\frac{\alpha}{2}}\|g\|_{L_p}^{\alpha}\left[\int_0^1\left|\I_{[r',2r']}(u)-\I_{[0,r']}(u)\right|^qdu\right]^{\frac{\alpha}{q}}\\
  &\qquad\qquad= 2^{\frac{\alpha}{2}}C_1(r')^{1-\frac{\alpha}{2}}\|g\|_{L_p}^{\alpha}(2r')^{\frac{\alpha}{q}}=2^{\frac{\alpha}{2}+\frac{\alpha}{q}}C_1(r')^{1+\frac{\alpha}{q}-\frac{\alpha}{2}}\|g\|_{L_p}^{\alpha}.
 \end{align*}
 Thus,
 $$
 I_2\leq 2^{\frac{\alpha}{2}}C_1t^{\alpha}(\xi(1)-\xi(0))^{\alpha}\int_2^{\infty}r^{\beta-1-1-\frac{\alpha}{2}}dr+2^{\frac{\alpha}{2}+\frac{\alpha}{q}}C_1\|g\|_{L_p}^{\alpha}\int_2^{\infty}r^{\beta-1-1-\frac{\alpha}{q}+\frac{\alpha}{2}}dr,
 $$
 where $\int_2^{\infty}r^{\beta-1-1-\frac{\alpha}{2}}dr$ and $\int_2^{\infty}r^{\beta-1-1-\frac{\alpha}{q}+\frac{\alpha}{2}}dr$ are finite because $\beta-2-\frac{\alpha}{2}<-1$ and $\beta-2-\frac{\alpha}{q}+\frac{\alpha}{2}<-1$, by the choice of $\alpha$.
 
 Similarly, we obtain the same estimate for $I_4$, by Lemma~\ref{lemma_estim_of_mass_near_1}. 
 
 In order to estimate $I_3$, we use Lemma~\ref{lemma_estim_of_mass}. For $r'=\frac{1}{r}\leq\frac{1}{2}$ we get
 \begin{align*}
  \int_{r'}^{1-r'}\int_0^t\p\left\{m(u,s)<r'\right\}duds &\leq\int_{r'}^{1-r'}r'G(r',u,t)du =2r'\int_{r'}^{1-r'}(g(u+r')-g(u))(g(u)-g(u-r'))du\\
  &+2tr'\int_{r'}^{1-r'}(\xi(u)-\xi(u-r'))(g(u+r')-g(u))du\\
  &+2t^2r'\int_{r'}^{1-r'}(\xi(u)-\xi(u-r'))(\xi(u+r')-\xi(u))du\\
  &+2tr'\int_{r'}^{1-r'}(\xi(u+r')-\xi(u))(g(u)-g(u-r'))du\\ 
  &=:J_1+J_2+J_3+J_4,
 \end{align*}
where $G$ is defined by~\eqref{f_function_G}. First, we estimate $J_1$. Using the trivial inequality $x^2\leq x+x^2\I_{\{x>1\}}$, $x\geq 0$, and H\"older's inequality with $\frac{1}{l}+\frac{1}{l'}=1$ and $l=\frac{p}{2}$, we obtain 
\begin{align*}
 J_1&\leq 2r'\int_{r'}^{1-r'}(g(u+r')-g(u-r'))^2du \leq 2r'\int_{r'}^{1-r'}(g(u+r')-g(u-r'))du\\
 &+2r'\int_{r'}^{1-r'}(g(u+r')-g(u-r'))^2\I_{\{g(u+r')-g(u-r')> 1\}}du\\
 &\leq 2r'\int_{r'}^{1-r'}(g(u+r')-g(u-r'))du\\
 &+2r'\left[\int_{r'}^{1-r'}(g(u+r')-g(u-r'))^pdu\right]^{\frac{2}{p}}\left[\int_{r'}^{1-r'}\I_{\{g(u+r')-g(u-r')> 1\}}du\right]^{\frac{1}{l'}}\\
  &\leq 2r'\int_{r'}^{1-r'}(g(u+r')-g(u-r'))du +8r'\|g\|_{L_p}^2\left[\int_{r'}^{1-r'}(g(u+r')-g(u-r'))du\right]^{1-\frac{2}{p}}.
\end{align*}
Since
 \begin{equation}\label{f_int_of_diff_g}
 \begin{split}
  \int_{r'}^{1-r'}(g(u+r')-g(u-r'))du&= \int_0^1\left(\I_{[2r',1]}-\I_{[0,1-2r']}\right)g(u)du\\
  &\leq \|g\|_{L_p}\left[\int_0^1\left|\I_{[2r',1]}-\I_{[0,1-2r']}\right|du\right]^{1-\frac{1}{p}}=\|g\|_{L_p}(4r')^{1-\frac{1}{p}},
 \end{split}
 \end{equation}
$J_1$ can be estimated as follows
$$
J_1\leq c_1\|g\|_{L_p}(r')^{2-\frac{1}{p}}+c_2\|g\|_{L_p}^{3-\frac{2}{p}}(r')^{2+\frac{2}{p^2}-\frac{3}{p}},
$$
where $c_1$, $c_2$ are constants. Using~\eqref{f_int_of_diff_g}, we have
\begin{align*}
J_2+J_3+J_4&\leq c_3t(r')^{2-\frac{1}{p}}(\xi(1)-\xi(0))\|g\|_{L_p}\\
&+2t^2r'(\xi(1)-\xi(0))\int_{r'}^{1-r'}(\xi(u+r')-\xi(u-r'))du\\
&\leq c_3t(r')^{2-\frac{1}{p}}(\xi(1)-\xi(0))\|g\|_{L_p}
+c_4t^2(r')^{2-\frac{1}{p}}(\xi(1)-\xi(0))\|\xi\|_{L_p}.
\end{align*}

Thus,
\begin{align*}
 I_3&\leq \beta C_2(\xi(1)-\xi(0))\|g\|_{L_p}\int_2^{\infty}r^{\beta-1-2+\frac{1}{p}}dr +C_3(\xi(1)-\xi(0))\|g\|_{L_p}^{3-\frac{2}{p}}\int_2^{\infty}r^{\beta-1-2-\frac{2}{p^2}+\frac{3}{p}}dr\\
 &+C_4(\xi(1)-\xi(0))\|\xi\|_{L_p}\int_2^{\infty}r^{\beta-1-2+\frac{1}{p}}dr,
\end{align*}
where constants $C_i$, $i\in\{2,3,4\}$, only depend on $p$ and $t$ and the integrals are finite according to the choice of $\beta$. The lemma is proved.
\end{proof}

Since we will consider $X(\cdot ,t)$, $t\geq 0$, as a continuous process with values in $L_{2+\delta}$, for convenience of notation we will denote $X_t:=X(\cdot ,t)$, $t\geq 0$.

\begin{lemma}\label{lemma_estim_of_norm_of_X_g}
 For each $t>0$, $\delta\in[0,1)$ and $\eps>\frac{2\delta}{1-\delta}$ there exists a constant $C=C(t,\delta,\eps)$ such that
 $$
 \E\sup_{s\in[0,t]}\|X_s-g\|_{L_{2+\delta}}^{2+\delta}\leq Ce^{C(\xi(1)-\xi(0))^2}(1+\|g\|_{L_{2+\eps}}^3+\|\xi\|_{L_{2+\eps}}).
 $$
\end{lemma}

\begin{proof}
 By the Burkholder-Davis-Gundy inequality, $(R3)$, $(R4)$ and Proposition~\ref{proposition_estim_of_dif_rate}, we have
 \begin{align*}
  \E\sup_{s\in[0,t]}\|X_s-g\|_{L_{2+\delta}}^{2+\delta}&= \E\sup_{s\in[0,t]}\int_0^1|X(u,s)-g(u)|^{2+\delta}du\\
  &\leq\int_0^1\E\sup_{s\in[0,t]}|X(u,s)-g(u)|^{2+\delta}du\leq 2^{1+\delta}\int_0^1\E\sup_{s\in[0,t]}|M(u,s)|^{2+\delta}du\\
  &+2^{1+\delta}\int_0^1\E\sup_{s\in[0,t]}\left|\int_0^s(\xi(u)-\left(\pr_{X_r}\xi\right)(u))dr\right|^{2+\delta}du\\
  &\leq 2^{1+\delta}C_1\int_0^1\E\left(\int_0^t\frac{ds}{m(u,s)}\right)^{1+\frac{\delta}{2}}du+2^{1+\delta}t^{2+\delta}(\xi(1)-\xi(0))^{2+\delta}\\
  &\leq 2^{1+\delta}C_1t^{\frac{\delta}{2}}\int_0^1\left(\E\int_0^t\frac{ds}{m(u,s)^{1+\frac{\delta}{2}}}\right)du+2^{1+\delta}t^{2+\delta}(\xi(1)-\xi(0))^{2+\delta}\\
  &\leq 2^{1+\delta}C_1t^{\frac{\delta}{2}}Ce^{C(\xi(1)-\xi(0))^2}(1+\|g\|_{L_{2+\eps}}^3+\|\xi\|_{L_{2+\eps}}) +2^{1+\delta}t^{2+\delta}(\xi(1)-\xi(0))^{2+\delta},
 \end{align*}
where $C_1$ depends on $\delta$. We could apply Proposition~\ref{proposition_estim_of_dif_rate} in the last step of the estimate above, since $1+\frac{\delta}{2}<\frac{3}{2}-\frac{1}{2+\eps}$. The lemma is proved.
\end{proof}

\begin{corollary}\label{corollary_estim_of_norm_of_X}
 Under the assumptions of Lemma~\ref{lemma_estim_of_norm_of_X_g},
 $$
 \E\sup_{s\in[0,t]}\|X_s\|_{L_{2+\delta}}^{2+\delta}\leq Ce^{C(\xi(1)-\xi(0))^2}(1+\|g\|_{L_{2+\delta}}^{2+\delta}+\|g\|_{L_{2+\eps}}^3+\|\xi\|_{L_{2+\eps}}),
 $$
 where $C$ depends on $t,\delta$ and $\eps$. 
\end{corollary}

\section{Tightness results}\label{section_tightness_results}

As was discussed in the introduction, we will build a CFWD as a limit of finite sticky-reflected particle systems, where the existence of the limit will be based on the tightness argument.  In this section, we will obtain a sufficient conditions for the tightness of a family of CFWDs. Let $\{\xi_n,\ n\geq 1\}$ and $\{g_n,\ n\geq 1\}$ be arbitrary sequences in $D^{\uparrow}$ and let $\{X^n,\ n\geq 1\}$ be a sequence of random elements in $D([0,1],C[0,\infty))$ satisfying $(R1)-(R4)$ with $\xi_n$, $g_n$ instead of $\xi$, $g$, and $\E \|X^n(\cdot ,t)\|_{L_2}^2<\infty$ for all $t\geq 0$ and $n\geq 1$. As in the previous section, we assume that such random elements $\{X^n,\ n\geq 1\}$ exist. Let $M^n(u,\cdot)$ and $A^n(u,\cdot)$ denote the martingale part and the part of bounded variation of $X^n(u,\cdot)$ for every $u\in[0,1]$, that is,
$$
A^n(u,t)=\int_0^t\left(\xi_n(u)-\left(\pr_{X_s^n}\xi_n\right)(u)\right)ds
$$
and
$$
M^n(u,t)=X^n(u,t)-g_n(u)-A^n(u,t),
$$
for all $t\geq 0$, where $X^n_t:=X^n(\cdot,t)$, $t\geq 0$. 

\subsection{Tightness of weak solutions}
\label{sub:tightness_of_weak_solutions}

In this section, we check the tightness of the family $\{X^n_{\cdot},\ n\geq 1\}$, where we consider $X^n$ as random processes in $\Li$.
Let
$$
M^n_t:=M^n(\cdot,t)\quad\mbox{and}\quad A^n_t:=A^n(\cdot,t),\quad t\geq 0,\ \ n\geq 1.
$$
We will also consider $M^n_{\cdot}$, $n\geq 1$, and $A^n_{\cdot}$, $n\geq 1$, as stochastic processes taking values in $L_2$. Using $(R1)-(R4)$ and Remark~\ref{remark_def_of_sol}, one can show that the processes $X_t^n$, $t\geq 0$, are weak solutions to SDE~\eqref{f_the_main_equation} with $g$ and $\xi$ replaced by $g_n$ and $\xi_n$.

In the next section, we will show that each limit point of $\{X^n_{\cdot},\ n\geq 1\}$ (which will exist by the tightness and Prokhorov's theorem) is a weak solution to SDE~\eqref{f_the_main_equation}. For this, we will need the convergence of the martingale parts, the parts of bounded variation and the quadratic variation processes. Let $\{e_i,\ i\in\N\}$ be a fixed orthonormal basis of $L_2$ and $M^n(e_i):=(M^n_{\cdot},e_i)$, $i\in\N$. We are going to prove that the family 
\begin{equation}\label{f_tilde_X}
\overline{X}^n:=\left(M^n,A^n,([M^n(e_i),M^n(e_j)])_{(i,j)\in\N^2}, \langle M^n \rangle \right),\quad n\geq 1,
\end{equation}
is tight in 
$$
\W:=C([0,\infty),L_2)\times C([0,\infty),L_2)\times C[0,\infty)^{\N^2}\times C[0,\infty)
$$
under the assumptions that $\{g_n,\ \xi_n\ n\geq 1\}$ is bounded in $L_{2+\delta}$ for some $\delta>0$ and $\{\xi_n(1)-\xi_n(0),\ n\geq 1\}$ is bounded in $\R$. In order to prove the tightness, we need the following lemma.

\begin{lemma}\label{lemma_compact_set_in_L_2_uparrow}
For every $C>0$ and $\delta>0$ the set $K_C:=\{g:\in L_2^{\uparrow}:\ \|g\|_{L_{2+\delta}}\leq C\}$ is compact in $L_2^{\uparrow}$ and, consequently, in $L_2$.
\end{lemma}

For the proof of the lemma see, e.g., Lemma~5.1~\cite{Konarovskyi:2017:EJP}.

\begin{proposition}\label{proposition_tightness_in_W}
  If there exists $\delta>0$ such that $\{g_n,\ \xi_n,\ n\geq 1\} \subset D^{\uparrow}$ is bounded in $L_{2+\delta}$ and $\{\xi_n(1)-\xi_n(0),\ n\geq 1\}$ is bounded in $\R$, then $\{\overline{X}^n,\ n\geq 1\}$ is tight in $\W$.
\end{proposition}

\begin{proof}
  In order to prove the proposition, it is enough to show that the coordinate processes of  $\{\overline{X}^n,\ n\geq 1\}$ are tight in the corresponding spaces, by Proposition~3.2.4~\cite{Ethier:1986}. We first prove that $\{A^n,\ n\geq 1\}$ is tight in $C([0,\infty),L_2)$. We remark that $A^n_t=t\xi_n-\int_0^t\pr_{X^n_s}\xi_nds$ and the sequence $\{t\xi_n,\ t\geq 0\}_{n\geq 1}$ is relatively compact in $C([0,\infty),L_2)$, by Lemma~\ref{lemma_compact_set_in_L_2_uparrow}. Thus, to prove the tightness of $\{A^n,\ n\geq 1\}$, it is enough to show that $\left\{\widehat{A}^n_{\cdot}:=\int_0^{\cdot}\pr_{X^n_s}\xi_nds,\ n\geq 1\right\}$ is tight in $C([0,\infty),L_2)$. Note that $\pr_{X^n_t}\xi_n$ belongs to $L_2^{\uparrow}$ for each $t\geq 0$, according to Lemma~\ref{lemma_monotonisity_of_pr}. Therefore,  the process  $\widehat{A}^n$ takes values in $L_2^{\uparrow}$ for all $n\geq 1$.  To prove the tightness, we will use Jakubowski's tightness criterion~\cite[Theorem~3.1]{Jakubowski:1986} (see also Theorem~3.6.4~\cite{Dawson:1993}), which claims that a family of continuous stochastic processes $Z^n_t$, $t\geq 0$, taking values in a Polish space $E$, is tight in $C([0,\infty),E)$\footnote{Originally Jakubowski's tightness criterion is formulated for the Skorohod space $D([0,\infty),E)$. However, the statement remains true for the space $C([0,\infty),E)$ as a closed subspace of $D([0,\infty),E)$.} if and only if the following two conditions are satisfied:
    \begin{enumerate}
      \item [(i)] for each $T>0$ and $\eps>0$ there exists a compact $K_{T,\eps} \subset E$ such that 
	\[
	  \p\left\{ Z^n_t \in K_{ T,\eps}, \ \ t \in [0,T] \right\}> 1-\eps,
	\]
	for every $n\geq 1$;

      \item [(ii)] for a family $F$ of real-valued continuous functions on $E$ that separates points in $E$ and is closed under addition, one has that $\{ f(Z^n_{\cdot }),\ \ n\geq 1 \}$ is tight in $C[0,\infty)$ for every $f \in F$.
    \end{enumerate}
  Using Lemma~\ref{lemma_compact_set_in_L_2_uparrow} and the estimate
\begin{align*}
\E\sup_{s\in[0,t]}\left\|\int_0^s\pr_{X^n_r}\xi_ndr\right\|_{L_{2+\delta}}^{2+\delta}&= \E\sup_{s\in[0,t]}\int_0^1\left|\int_0^s(\pr_{X^n_r}\xi_n)(u)dr\right|^{2+\delta}du\\ 
&\leq t^{1+\delta} \E\int_0^t\left(\int_0^1|\pr_{X^n_r}\xi_n|^{2+\delta}(u)du\right)dr\\
&\leq t^{1+\delta} \E\int_0^t\left(\int_0^1\left(\pr_{X^n_r}|\xi_n|^{\frac{2+\delta}{2}}\right)^2(u)du\right)dr\\
&= t^{1+\delta} \E\int_0^t\left\|\pr_{X^n_r}|\xi_n|^{\frac{2+\delta}{2}}\right\|^2_{L_2}dr\leq t^{2+\delta}\left\|\xi_n\right\|^{2+\delta}_{L_{2+\delta}},
\end{align*}
where the inequality $|\pr_{X^n_r}\xi_n|^{\frac{2+\delta}{2}}\leq\pr_{X^n_r}|\xi_n|^{\frac{2+\delta}{2}}$ follows from Remark~\ref{remark_notes_obout_proj}~(ii) and H\"older's inequality for conditional expectations, we can conclude that for each $\eps>0$ and $T>0$ there exists $C>0$ such that
$$
\p\left\{\exists t\in[0,T],\ \widehat{A}^n_t\not\in K_C\right\}<\eps
$$  
for all $n\geq 1$, where $K_C=\{g\in L_2^{\uparrow}:\ \|g\|_{L_{2+\delta}}\leq C\}$ is compact in $L_2$. Consequently, the family of processes $\{\widehat{A}^n,\ n\geq 1\}$ satisfies property $(i)$ of Jakubowski's tightness criterion. In order to check property $(ii)$, we choose $F$ to be the set of all linear functionals $g\mapsto(g,h)_{L_2}$, $h \in L_2$, on $L_2$. It trivially separates points on $L_2$ and is closed under addition. Thus, we only need to check the tightness of $\left\{(\widehat{A}^n_{\cdot},h)_{L_2},\ n\geq 1\right\}$ in $C[0,\infty)$. We will apply the Aldous tightness criterion (see, e.g., Theorem~16.11~\cite{Kallenberg:2002}). Let $T>0$, $\{ r_n,\ n\geq 1 \} \subset [0,T]$ be any sequence decreasing to 0 and $\{\tau_n,\ n\geq 1\}$ be any sequence of $(\F_t^{X^n})$-stopping times on $[0,T]$. Using Chebyshev's inequality and then the Cauchy–Schwarz inequality, we obtain for every $\eps>0$
\begin{align*}
  &\p\left\{ \left|(\widehat{A}^n_{\tau_n+r_n},h)_{L_2}-(\widehat{A}^n_{\tau_n},h)_{L_2}\right|>\eps \right\}\leq \frac{1}{ \eps^2 } \E\left((\widehat{A}^n_{\tau_n+r_n},h)_{L_2}-(\widehat{A}^n_{\tau_n},h)_{L_2}\right)^2\\
  &\qquad\qquad= \frac{1}{ \eps^2 }\E \left( \int_{ \tau_n }^{ \tau_n+r_n } (\pr_{X^n_s}\xi_n,h)_{L_2}ds  \right)^2\leq \frac{ r_n }{ \eps^2 }\E\int_{ \tau_n }^{ \tau_n+r_n } (\pr_{X_s^n}\xi_n,h)^2_{L_2}ds\\
  &\qquad\qquad \leq \frac{ r_n^2 }{ \eps^2 }\|\xi_n\|_{L_2}^2\|h\|_{L_2}^2 \to 0, \quad r_n \to 0.
\end{align*}
This implies the tightness of $\{\widehat{A}^n,\ n\geq 1\}$ in $C([0,\infty),L_2 )$.

Similarly, we can prove that $\{X^n,\ n\geq 1\}$ is tight in $C([0,\infty),L_2^{\uparrow})$, using Corollary~\ref{corollary_estim_of_norm_of_X}, the equality $X^n_t=g_n+M_t^n+t\xi_n- \widehat{A}^n_t$, $t\geq 0$,  and the estimate
\begin{align*}
  &\p\left\{ \left|(X^n_{\tau_n+r_n},h )_{L_2}-(X^n_{\tau_n},h)_{L_2}\right|>\eps \right\}\leq \frac{1}{ \eps^2 }\E\left( (X^n_{\tau_n+r_n},h )_{L_2}-(X^n_{\tau_n},h)_{L_2} \right)^2\\
  &\qquad\qquad\leq \frac{3}{ \eps^2 }\E\left( (M^n_{\tau_n+r_n},h)_{L_2}-(M^n_{\tau_n},h)_{L_2} \right)^2+ \frac{3}{ \eps^2 }\E\big( r_n (\xi_n,h)_{L_2} \big)^2\\
  &\qquad\qquad+ \frac{3}{ \eps^2 }\E\left( (\widehat{A}^n_{\tau_n+r_n},h)_{L_2}-(\widehat{A}^n_{\tau_n},h)_{L_2} \right)^2\\
  &\qquad\qquad\leq \frac{3}{ \eps^2 }\E \int_{ \tau_n }^{ \tau_n+r_n } \big\|\pr_{X_s^n}h\big\|_{L_2}^2 ds  + \frac{6 r_n^2}{ \eps^2 }\|\xi_n\|_{L_2}^2 \|h\|_{L_2}^2\\
  &\qquad\qquad\leq \frac{3 r_n}{ \eps^2 }\|h\|_{L_2}^2\left( 1+2r_n \|\xi_n\|_{L_2}^2 \right).
\end{align*}
Hence, by Proposition~3.2.4~\cite{Ethier:1986}, $\{(X^n,A^n),\ n\geq 1\}$ is tight in $C([0,\infty),L_2)^2$ and, consequently, the sequence of processes $\{M^n,\ n\geq 1\}$, which are defined by $M^n_t=X^n_t-g_n-A^n_t$, $t\geq 0$, is tight in $C([0,\infty),L_2)$.

For every $i,j \in \N$ the tightness of the family of processes
$$
[M^n(e_i),M^n(e_j)]_t=\int_{ 0 }^{ t } (\pr_{X^n_s}e_i,e_j)_{L_2}ds, \quad t\geq 0,\ \ n\geq 1, 
$$ 
in $C[0,\infty)$ can be proved in the same way as the tightness of $\{ (\widehat{A}^n_{\cdot },h)_{L_2},\ n\geq 1 \}$.

Next, we prove the tightness of $\{\langle M^n\rangle,\ n\geq 1\}$ in $C[0,\infty)$. We are going to use the Aldous tightness criterion again. By Lemma~\ref{lemma_connection_HS_with_int} and Proposition~\ref{proposition_estim_of_dif_rate},
\begin{align*}
\E\langle M^n\rangle_t&=\E\int_0^t\|\pr_{X^n_s}\|_{HS}^2ds=\E\int_0^t\left(\int_0^1\frac{du}{m_n(u,s)}\right)ds \leq Ce^{C(\xi_n(1)-\xi_n(0))^2}(1+\|g_n\|_{L_{2+\delta}}^3+\|\xi_n\|_{L_{2+\delta}}),
\end{align*}
where the constant $C$ depends on $t$ and $\delta$. Thus, \ the \ boundedness \ of \ %??
$\{g_n,\ \xi_n,\ n\geq 1\}$ in $L_{2+\delta}$ and $\{\xi_n(1)-\xi_n(0),\ n\geq 1\}$ in $\R$ yields the tightness of $\{\langle M^n\rangle_t,\ n\geq 1\}$ in $\R$ for all $t\geq 0$. Next, let $T>0$, $\{r_n,\ n\geq 1\}\subset[0,T]$ be any sequence decreasing to 0 and $\{\tau_n,\ n\geq 1\}$ be any sequence of $(\F^{X^n}_t)$-stopping times on $[0,T]$. Then for each $\eps>0$ and $\beta\in\left(1,\frac{3}{2}-\frac{1}{2+\delta}\right)$
\begin{align*}
\p\big\{|\langle M^n\rangle_{\tau_n+r_n}-\langle M^n\rangle_{\tau_n}|>\eps\big\}&\leq\frac{1}{\eps}\E\left(\langle M^n\rangle_{\tau_n+r_n}-\langle M^n\rangle_{\tau_n}\right) =\frac{1}{\eps}\E\int_{\tau_n}^{\tau_n+r_n}\|\pr_{X^n_s}\|_{HS}^2ds\\
[\mbox{Lemma~\ref{lemma_connection_HS_with_int}}]&=\frac{1}{\eps}\E\int_{\tau_n}^{\tau_n+r_n}\left(\int_0^1\frac{du}{m_n(u,s)}\right)ds=\frac{1}{ \eps }\E \int_{ 0 }^{ 1 } \int_{ 0 }^{ 2T } \frac{\I_{[\tau_n,\tau_n+r_n]} }{ m_n(u,s) } duds  \\
[\mbox{H\"older in.}]&\leq \frac{1}{ \eps }\left(\E \int_{ 0 }^{ 1 } \int_{ 0 }^{ 2T } \I_{[\tau_n,\tau_n+r_n]}  duds\right)^{\frac{ \beta-1 }{ \beta }} \left(\E \int_{ 0 }^{ 1 } \int_{ 0 }^{ 2T } \frac{duds }{ m_n^{\beta}(u,s) }\right)^{ \frac{1}{ \beta }} \\
&=\frac{r_n^{\frac{\beta-1}{\beta}}}{\eps} \left(\E\int_0^1\int_0^{2T}\frac{duds}{m_n^{\beta}(u,s)}\right)^{ \frac{1}{ \beta }}.
\end{align*}
Consequently, $\p\{|\langle M^n\rangle_{\tau_n+r_n}-\langle M^n\rangle_{\tau_n}|>\eps\}\to 0$ as $n\to\infty$, by Proposition~\ref{proposition_estim_of_dif_rate}. Thus, the Aldous tightness criterion implies the compactness of $\{\langle M^n\rangle,\ n\geq 1\}$ in $C[0,\infty)$. This completes the proof of the proposition.
\end{proof}

\begin{corollary}\label{corollary_relatively_compactness_of_X}
Let $\delta>0$, $\{g_n,\ \xi_n,\ n\geq 1\}\subset \St$ and $\{\overline{X}^n,\ n\geq 1\}$ be defined by~\eqref{f_tilde_X}. If $g_n\to g$, $\xi_n\to\xi$ in $L_{2+\delta}$ and $\{\xi_n(1)-\xi_n(0),\ n\geq 1\}$ is bounded, then there exists a subsequence $N\subseteq\N$ and a random element $\overline{X}$ in $\W$ such that $\overline{X}^n\to\overline{X}$ in $\W$ in distribution along $N$.   
\end{corollary}

\begin{proof}
The statement of the corollary follows from Prochorov's theorem and Proposition~\ref{proposition_tightness_in_W}.
\end{proof}

\subsection{Tightness in the Skorohod space}\label{subsection_tightness_in_Skorohod_space}

In this section, we will consider the processes $M^n$, $A^n$, $X^n$, which were defined at the beginning of Section~\ref{section_tightness_results}, as random elements in the Skorohod space $D([0,1],C[0,\infty))$. In order to prove the tightness of $X^n$, $n\geq 1$, in $D([0,1],C[0,\infty))$, we are going to apply the same approach as in the proof of Proposition~2.2~\cite{Konarovskyi:2014:arx}, based on estimates from Section~\ref{sec:a_priory_estimates} (see corollaries~\ref{corollary_estim_of_diff_rate_at_u},~\ref{corollary_estim_of_diff_rate_at_0_and_1} and Lemma~\ref{lemma_three_points}) which require a type of regularity of $g_n$ and $\xi_n$ on the whole interval $[0,1]$. However, restricting $X^n(u,\cdot )$, $u \in [0,1]$, to smaller intervals $[a,b] \subseteq [0,1]$ and proving the tightness in $D([a,b],C[0,\infty))$, the piecewise regularity of $g_n$ and $\xi_n$ will be enough. This will allow to get the existence of the CFWD as a random element in the Skorohod space for every {\it piecewise} H\"older continuous initial condition $g$ and interaction potential $\xi$ in the next section. Therefore, we will restrict $X^n(u,\cdot )$, $u \in [0,1]$, to an interval $[a,b] \subseteq [0,1]$. For $\pi=[a,b]\subseteq [0,1]$ we set
$$
X^{\pi,n}(u,\cdot)=\begin{cases}
                    X^n(u,\cdot),& u\in[a,b),\\
                    \lim_{u\uparrow b}X^n(u,\cdot),& u=b,
                   \end{cases}
$$
where $\lim_{u\uparrow b}X^n(u,\cdot)$ exists in $C[0,\infty)$ due to the fact that the map $u\mapsto X^n(u,\cdot )$ is c\`{a}dl\'{a}g. Let $G^n$ be defined by~\eqref{f_function_G} with $g$ and $\xi$ replaced by $g_n$ and $\xi_n$, respectively, and let
\begin{align*}
G^n_a(r,t)&:=(\xi_n(a+r)-\xi_n(a))t+g_n(a+r)-g_n(a),\quad r\in(0,1-a],\\
G^n_b(r,t)&:=(\xi_n(b-)-\xi_n(b-r))t+g_n(b-)-g_n(b-r),\quad r\in(0,b].
\end{align*}

\begin{proposition}\label{proposition_tightness_in_D_pi}
  Let $T>0$, $\pi:=[a,b]\subseteq[0,1]$, and $\{g_n,\ \xi_n,\ n\geq 1\} \subset D^{\uparrow}$ be uniformly bounded, i.e.,
  \[
    \sup\limits_{ n\geq 1 }\sup\limits_{ u \in [0,1] }(|g_n(u)|\vee |\xi_n(u)|)<\infty.
  \]
  If there exist $\beta>0$ and $C>0$ such that for each $n\geq 1$
\begin{enumerate}
\item[(c1)] $G^n(r\wedge(u-a),r\wedge(b-u),u,T)\leq Cr^{1+\beta}$ for all $u\in(a,b),\ r>0$;

\item[(c2)] $G^n_v(r,T)\leq Cr^{\frac{1}{2}+\beta}$ for all $v\in\{a,b\}$ and $r\in(0,b\wedge 1-a]$,
\end{enumerate}
then the family $\{X^{\pi,n}(u,t),\ u\in[a,b],\ t\in[0,T]\}_{n\geq 1}$ is tight in $D([a,b],$ $C[0,T])$.%??
\end{proposition}

\begin{remark}\label{rem_on_holder_continuity_of_initial_condition} %
  The assumption on the piecewise H\"older continuity of the initial condition $g$ and the interaction potential $\xi$ in Theorem~\ref{theorem_existence_in_general_case} is required for the construction of sequences $\{g_n,\ n\geq 1\}$ and $\{ \xi_n,\ n\geq 1 \}$ which converge to $g$ and $\xi$, respectively, and satisfy the uniform regularity assumptions $(c1)$, $(c2)$ of Proposition~\ref{proposition_tightness_in_D_pi} (see Section~\ref{ssub:tightness_of_a_finite_particle_system} for more details).
\end{remark}

\begin{proof}[Proof of Proposition~\ref{proposition_tightness_in_D_pi}]
The proof is similar to the proof of Proposition~2.2~\cite{Konarovskyi_LDP:2015}. Here, we indicate the main steps only.

The statement will follow from theorems~3.8.6 and~3.8.8~\cite{Ethier:1986} and Remark~3.8.9~ibid. We only have to check the following properties of $\{X^{\pi,n},\ n\geq 1\}$.
\begin{enumerate}
\item[(a)] There exists $C_1>0$ such that 
\begin{align*}
\p\left\{\|X^{\pi,n}((u+r)\wedge b,\cdot)-X^{\pi,n}(u,\cdot)\|_{C[0,T]}>\lambda,\ \|X^{\pi,n}(u,\cdot)-X^{\pi,n}((u-r)\vee a,\cdot)\|_{C[0,T]}>\lambda\right\}\leq\frac{C_1r^{1+\beta}}{\lambda^2}
\end{align*}
for all $n\in\N$, $u\in(a,b)$, $r>0$ and $\lambda>0$.

\item[(b)] For some $\alpha>0$
\begin{equation}\label{f_lim_at_a_for_tightness}
\lim_{\delta\to 0+}\sup\limits_{n\geq 1}\E\left[\|X^{\pi,n}(a+\delta,\cdot)-X^{\pi,n}(a,\cdot)\|^{\alpha}_{C[0,T]}\wedge
1\right]=0
\end{equation}
and
\begin{equation}\label{f_lim_at_b_for_tightness}
\lim_{\delta\to 0+}\sup\limits_{n\geq 1}\E\left[\|X^{\pi,n}(b,\cdot)-X^{\pi,n}(b-\delta,\cdot)\|^{\alpha}_{C[0,T]}\wedge
1\right]=0.
\end{equation}

\item[(c)] For all $u\in[a,b]$ the sequence $\{X^{\pi,n}(u,t),\ t\in[0,T]\}_{n\geq 1}$ is tight in $C[0,T]$.
\end{enumerate}

Properties (a), (b) and (c) are needed for the verification of conditions (8.39), (8.30)\footnote{Here, we used the statement for the tightness in $D([a,\infty),C[0,T])$, which can be applied to $\{X^{\pi,n}(u\wedge b,t),\ u\in[0,\infty),\ t\in[0,T]\}_{n\geq 1}$. Since $D([a,b],C[0,T])$ contains functions which are continuous in $b$, additional property~\eqref{f_lim_at_b_for_tightness} is needed for the tightness there.} of~\cite[Chapter~3]{Ethier:1986} and (a) of Theorem~3.7.2~ibid., respectively.

Property~(a) immediately follows from~$(c1)$ and Lemma~\ref{lemma_three_points}. 

Next, let us prove (b). We will check~\eqref{f_lim_at_b_for_tightness} for every $\alpha>1$. The proof of~\eqref{f_lim_at_a_for_tightness} is similar. Using the monotonicity of $X^{\pi,n}(u,\cdot)$, $u\in[a,b]$, and the monotone convergence theorem, we have for each $\delta\in(0,b-a)$
\begin{equation}\label{f_lim_at_b_for_tightness_mon_conv}
\begin{split}
\E\left[\|X^{\pi,n}(b,\cdot)-X^{\pi,n}(b-\delta,\cdot)\|^{\alpha}_{C[0,T]}\wedge
1\right] &=\sup_{\gamma\in(0,\delta)}\E\left[\|X^{\pi,n}(b-\gamma,\cdot)-X^{\pi,n}(b-\delta,\cdot)\|^{\alpha}_{C[0,T]}\wedge
1\right]\\
&=\sup_{\gamma\in(0,\delta)}\E\left[\|X^n(b-\gamma,\cdot)-X^n(b-\delta,\cdot)\|^{\alpha}_{C[0,T]}\wedge
1\right].
\end{split}
\end{equation}
In order to estimate the last expression, we are going to use Doob's martingale inequality (see, e.g., Proposition~2.2.16~\cite{Ethier:1986}) for submartingales. Since $X^n(b-\gamma,t)-X^{n}(b-\delta,t)$, $t\geq 0$, is not a submartingale, we introduce a new process
$$
X_{\delta,\gamma}^n(t):=g_n(b-\gamma)-g_n(b-\delta)+M^n(b-\gamma,t)-M^n(b-\delta,t)+A_{\delta,\gamma}^n(t),\quad t\in[0,T],
$$
where
\begin{align*}
A_{\delta,\gamma}^n(t):&=\int_0^t(b_{\delta,\gamma}^n(s)\vee 0)ds
\end{align*}
and
$$
b_{\delta,\gamma}^n(s):=\xi_n(b-\gamma)-\xi_n(b-\delta)-\left[\left(\pr_{X_s^n}\xi_n\right)(b-\gamma)-\left(\pr_{X_s^n}\xi_n\right)(b-\delta)\right],
$$
and note that it is a continuous submartingale because $A_{\delta,\gamma}^n$ is an increasing continuous process. Moreover, 
$$
0\leq X^n(b-\gamma,t)-X^n(b-\delta,t)\leq X_{\delta,\gamma}^n(t)
$$ 
for all $t\in[0,T]$.  We also introduce the stopping time
$$
\sigma_{\delta,\gamma}^n:=\inf\left\{t:\ X_{\delta,\gamma}^n(t)=1\right\}\wedge T.
$$
By Doob's martingale inequality  and the estimate
$$
b_{\delta,\gamma}^n(t)\vee 0\leq\xi_n(b-\gamma)-\xi_n(b-\delta),\quad t\in[0,T], 
$$
we get
\begin{align*}
&\E\Bigg[\sup_{t\in[0,T]}(X^n(b-\gamma,t)-X^n(b-\delta,t))^{\alpha}\wedge 1\Bigg] \leq\E\Bigg[\sup_{t\in[0,T]}(X_{\delta,\gamma}^n(t\wedge \sigma_{\delta,\gamma}^n))^{\alpha}\Bigg]\\
&\qquad\qquad\leq C_{\alpha}\E\Bigg[(X_{\delta,\gamma}^n(T\wedge \sigma_{\delta,\gamma}^n))^{\alpha}\Bigg]
\leq C_{\alpha}\E\Bigg[X_{\delta,\gamma}^n(T\wedge \sigma_{\delta,\gamma}^n)\Bigg]\\
&\qquad\qquad\leq C_{\alpha}[g_n(b-\gamma)-g_n(b-\delta)+T(\xi_n(b-\gamma)-\xi_n(b-\delta))],
\end{align*}
where $C_{\alpha}=\left(\frac{\alpha}{\alpha-1}\right)^{\alpha}$. Thus, by~$(c2)$ and~\eqref{f_lim_at_b_for_tightness_mon_conv},
$$
\E\Big[\|X^{\pi,n}(b,\cdot)-X^{\pi,n}(b-\delta,\cdot)\|^{\alpha}_{C[0,T]}\wedge 1\Big]\leq C_{\alpha}G^n_b(\delta,T)\leq C_{\alpha}C\delta^{ \frac{1}{ 2 }+\beta}.
$$
This implies~\eqref{f_lim_at_b_for_tightness}.

Property (c) can be proved in the same way as Lemma~2.5~\cite{Konarovskyi_LDP:2015}, using the Aldous tightness criterion and the estimates
\begin{equation}\label{f_estim_of_m_n}
\E\int_0^T\frac{dt}{m^{1+\frac{\beta}{2}}_n(u,t)}\leq \tilde{C}<\infty,
\end{equation}
\begin{align*}
\E |X^{\pi,n}(u,t)|&\leq\E\left|X^{\pi,n}(u,t)-\int_0^1X^n(v,t)dv\right|+\E\left|\int_0^1X^n(v,t)dv\right|\\
&\leq\E(X^n(1,t)-X^n(0,t))+\E\left|\int_0^1X^n(v,t)dv\right|\\
&\leq g_n(1)-g_n(0)+T(\xi_n(1)-\xi_n(0))+\E\left|\int_0^1X^n(v,t)dv\right|
\end{align*}
for all $n\geq 1$, where $\int_0^1X^n(v,t)dv=(X^n_t,1)_{L_2}$, $t\in[0,T]$, is a Brownian motion, since it is a continuous martingale with quadratic variation $\int_0^t\pr_{X_s^n}1ds=t$ (we note that $(A^n_t,1)_{L_2}=0$). Here,~\eqref{f_estim_of_m_n} follows from $(c1)$, $(c2)$ and corollaries~\ref{corollary_estim_of_diff_rate_at_u},~\ref{corollary_estim_of_diff_rate_at_0_and_1}. The fact that $X^n(v,t)$, $t\geq 0$, is a (square-integrable) martingale for $v \in \{ 0,1 \}$ follows from~\eqref{f_estim_of_m_n}. The proposition is proved.
\end{proof}

\begin{remark}\label{remark_tightness_on_0_infty}
If $(c1)$, $(c2)$ hold for some $T>0$, then one can easily check that they hold for any $T\geq 0$, with $C$ depending on $T$. Thus, under the assumptions of Proposition~\ref{proposition_tightness_in_D_pi} (for some $T>0$), the family $\{X^{\pi,n}(u,t),\ u\in[a,b],\ t\in[0,\infty)\}_{n\geq 1}$ is tight in $D([a,b],C[0,\infty))$. 
\end{remark}

\section{Construction of CFWD}\label{construction_of_cfwd}

This section is devoted to the proof of theorems~\ref{theorem_existence_in_general_case1} and~\ref{theorem_existence_in_general_case}.

\subsection{The reversible CFWD and a finite sticky-reflected particle system}\label{subsection_Dirichlet_form_approach}

In this section, we will recall the construction of a weak solution to SDE~\eqref{f_the_main_equation} for some class of functions $g$ and $\xi$, using the Dirichlet form approach. Namely, we are going to construct a reversible CFWD for ``almost all'' $g\in\Li(\xi)$ and bounded $\xi \in \Dr$, as in~\cite{Konarovskyi_SR:2017}. In the case $\xi\in\St$, we also show that the constructed process has a modification from the Skorohod space satisfying $(R1)-(R4)$. This is the first step of the construction of CFWD in the general case, where we obtain a finite sticky-reflected particle system (see Proposition~\ref{proposition_existence_via_DF} below and Remark~\ref{rem_finite_particle_system}) needed for the approximation. We will assume that $\xi\in\Dr$ is a fixed bounded function.

We first introduce a measure $\Xi^{\xi}$ on $\Li$ which plays a role of an invariant measure for the reversible CFWD $X_t$, $t\geq 0$. We set for each $n\in\N$
\begin{equation} %definition of En
  \label{equ_definition_of_en}
  E^n:=\{x=(x_k)_{k\in[n]}\in\R^n:\ x_1<\ldots< x_n\}
\end{equation}
and
$$
Q^n:=\{q=(q_k)_{k\in[n-1]}\in[0,1]^{n-1}:\ q_1<\ldots<q_{n-1}\},\quad \mbox{if}\ \ n\geq 2,
$$
where $[n]:=\{1,\ldots,n\}$. Considering $q\in Q^n$, we will always take $q_0=0$ and $q_n=1$ for convenience. Let $\chi_1:\R\to\Li$ and $\chi_n:E^n\times Q^n\to\Li$, $n\geq 2$, be given by 
\begin{align*}
\chi_1(x):=x\I_{[0,1]}\quad\mbox{and}\quad
\chi_n(x,q):=\sum_{k=1}^nx_k\I_{[q_{k-1},q_k)}+x_n\I_{\{1\}}.
\end{align*}
The functions $\chi_n(x,q)$ can be interpreted as the description of $n$ clusters occupying positions $x_1<x_2<\dots<x_n $ whose masses equal $q_1-q_0,q_2-q_1,\dots,q_n-q_{n-1}$, respectively. Similarly, $\chi_1(x)$ describes a unique cluster of mass one occupying position $x$. Setting
$$
c_n(q):=\prod_{k=1}^n(q_k-q_{k-1}),\quad n\geq 2,
$$
we define the measure on $\Li$ as follows
$$
\Xi^{\xi}:=\sum_{ n=1 }^{ \infty } \Xi_n^{\xi},
$$ 
where
$$
\Xi^{\xi}_1(B):=\int_{\R}\I_B(\chi_1(x))dx
$$ 
and
$$
\Xi^{\xi}_n(B):=\int_{Q^n}\left[c_n(q)\int_{E^n}\I_B(\chi_n(x,q))dx\right]d\xi^{\otimes(n-1)}(q)
$$ 
for all $B\in \B(\Li)$. Here, $\int_{Q^n}\ldots d\xi^{\otimes(n-1)}(q)$ is the $(n-1)$-dim Lebesgue-Stieltjes integral with respect to $\xi^{\otimes(n-1)}(q)=\xi(q_1)\cdot\ldots\cdot\xi(q_{n-1})$. Roughly speaking, the measure $\Xi_n$ is the distribution of $n$ ordered clusters whose positions are uniformly distributed on the real line with masses $q_1-q_0,q_2-q_1,\dots,q_n-q_{n-1}$, where $(q_1,\dots,q_{n-1})$ have the distribution defined by
\[
  \int_{ \cdot  }  \prod_{k=1}^n(q_k-q_{k-1})d\xi^{\otimes(n-1)}(q).
\]
on $Q^{n}$.   
The measure $\Xi^{\xi}$ was first proposed in Section~4~\cite{Konarovskyi_SR:2017}. 

\begin{proposition}\label{proposition_measure_Xi}
  For each bounded $\xi \in D^{\uparrow}$ the measure $\Xi^{\xi}$ is a $\sigma$-finite measure on $\Li$ with $\supp\Xi^{\xi}=\Li(\xi)$.
\end{proposition}

\begin{proof}
The proof of the proposition was given in~\cite{Konarovskyi_SR:2017}. See Lemma~4.2~(ii), Remark~4.4 and Proposition~4.7 there.
\end{proof}

Next, we denote the linear space generated by functions on $\Li$ of the form
\begin{equation}
  \label{f_form_of_U}
U=\vartheta\left((\cdot,h_1)_{L_2},\ldots,(\cdot,h_m)_{L_2}\right)\varphi\left(\|\cdot\|_{L_2}^2\right)=\vartheta\left((\cdot,\vect{h})_{L_2}\right)\varphi\left(\|\cdot\|_{L_2}^2\right)
\end{equation}
by $\FC$, where $\vartheta\in\Cfb(\R^m),$ $\varphi\in\Cfo(\R)$ and $h_j\in L_2,$ $j\in[m]$.

For each $U\in\FC$ we introduce its derivative as follows
$$
\D U(g):=\pr_g\left[\nabla^{L_2}U(g)\right],\quad g\in \Li,
$$
where $\nabla^{L_2}$ denotes the Fr\'{e}chet derivative on $L_2$. If $U$ is given by~\eqref{f_form_of_U}, then a simple calculation shows that
\begin{equation}\label{f_derivativ_D}
\D U(g)=\varphi\left(\|g\|_{L_2}^2\right)\sum_{j=1}^m\partial_j\vartheta\left((g,\vect{h})_{L_2}\right)\pr_gh_j
+2\vartheta\left((g,\vect{h})_{L_2}\right)\varphi'\left(\|g\|_{L_2}^2\right)g
\end{equation}
for all $g\in\Li$, where $\partial_j\vartheta(x):=\frac{\partial}{\partial x_j}\vartheta(x)$, $x\in\R^m$.

The following integration by parts formula was proved in~\cite{Konarovskyi_SR:2017} (see Theorem~5.6 there).

\begin{theorem}\label{theorem_integration_by_parts}
  For each $\xi \in D^{\uparrow}$ and $U,V\in\FC$
 \begin{equation}\label{f_int_by_parts_for_Xi}
 \int_{\Li}(\D U(g),\D V(g))_{L_2}\Xi^{\xi}(dg)=-\int_{\Li}L_0U(g)V(g)\Xi^{\xi}(dg) -\int_{\Li}V(g)(\nabla^{L_2}U(g),\xi-\pr_g\xi)_{L_2}\Xi^{\xi}(dg),
 \end{equation}
 where
\begin{align*}
L_0U(g)&=\varphi\left(\|g\|_{L_2}^2\right)\sum_{i,j=1}^m\partial_i\partial_j\vartheta\left((g,\vect{h})_{L_2}\right)(\pr_gh_i,\pr_gh_j)_{L_2}\\
&+\vartheta\left((g,\vect{h})_{L_2}\right)\left[4\varphi''\left(\|g\|_{L_2}^2\right)\|g\|_{L_2}^2+2\varphi'\left(\|g\|_{L_2}^2\right)\cdot\# g\right]\\
&+2\sum_{j=1}^m\partial_j\vartheta\left((g,\vect{h})_{L_2}\right)\varphi'\left(\|g\|_{L_2}^2\right)(\pr_gh_j,g)_{L_2},
\end{align*}
if $U$ is defined by~\eqref{f_form_of_U}.
\end{theorem}

\begin{remark}
We note that $\# g$ is finite only for $g\in\St$. Since $\Xi^{\xi}(\Li\setminus\St)=0$, the function $L_0U$ is well-defined $\Xi^{\xi}$-a.e. for all $U\in\FC$. Moreover, it belongs to $L_2(\Li,\Xi^{\xi})$, by Lemma~4.2~\cite{Konarovskyi_SR:2017}.
\end{remark}

Since $\supp\Xi^{\xi}=\Li(\xi)$, we will define a bilinear form on $\LL$. We set
$$
\e^{\xi}(U,V)=\frac{1}{2}\int_{\Li(\xi)}(\D U(g),\D V(g))_{L_2}\Xi^{\xi}(dg),\quad U,V\in\FC.
$$
Then $(\e^{\xi},\FC)$ is a densely defined positive definite symmetric bilinear form on $\LL$. Moreover, Theorem~\ref{theorem_integration_by_parts} and Proposition~I.3.3~\cite{Ma:1992} imply that $(\e^{\xi},\FC)$ is closable on $\LL$. Its closure will be denoted by $(\e^{\xi},\Dom^{\xi})$. 

\begin{theorem}\label{theorem_dirichlet_form}
For each bounded $\xi\in D^{\uparrow}$ the bilinear form $(\e^{\xi},\Dom^{\xi})$ is a quasi-regular local\footnote{For the definition of quasi-regularity, strict quasi-regularity and local property see def.~IV.3.1,~V.2.11 and~V.1.1~\cite{Ma:1992}, respectively.} symmetric Dirichlet form on $\LL$. Moreover, if $\xi$ is constant on some neighbourhoods of~$0$ and~$1$, then $(\e^{\xi},\Dom^{\xi})$ is strictly quasi-regular and conservative.
\end{theorem}

\begin{proof}
The proof of the theorem can be found in~\cite{Konarovskyi_SR:2017}. The fact that $(\e^{\xi},\Dom^{\xi})$ is a Dirichlet form, the quasi-regularity and the local property were proved in propositions~5.14, 6.5 and~6.6, respectively. The strict quasi-regularity and conservativeness were proved in Proposition~6.9.
\end{proof}

By theorems~IV.6.4,~V.1.11~\cite{Ma:1992} and Theorem~\ref{theorem_dirichlet_form}, there exists a diffusion process\footnote{see Definition~V.1.10~\cite{Ma:1992}}
$$
\tilde{X}=\left(\tilde{\Omega},\tilde{\F},(\tilde{\F}_t)_{t\geq 0},\{\tilde{X}_t\}_{t\geq 0}, \{\tilde{\p}_g\}_{g\in\Li(\xi)}\right)
$$
with state space $\Li(2\xi)=\Li(\xi)$ and life time $\zeta$ that is properly associated with $(\e^{2\xi},\Dom^{2\xi})$\footnote{We consider the interaction potential $2\xi$ instead of $\xi$ in order to obtain solutions to SDE with the drift term $(\xi-\pr_{X_t}\xi)dt$ instead of $\frac{1}{2}(\xi-\pr_{X_t}\xi)dt$ (see Section~8~\cite{Konarovskyi_SR:2017}).}. Furthermore, if $\xi$ is constant on some neighbourhoods of~$0$ and~$1$, then $\tilde{X}$ is a Hunt process with infinite life time. 

We set 
$$
\tilde{M}_t:=\tilde{X}_t-\tilde{X}_0-\int_0^t(\xi-\pr_{\tilde{X}_s}\xi)ds,\quad t\geq 0,
$$
and denote the expectation with respect to $\tilde{\p}_g$ by $\tilde{\E}_g$.

\begin{proposition}\label{proposition_mart_prop_of_X_t}
  Let $\xi \in D^{\uparrow}$ be constant on some neighbourhoods of~$0$ and~$1$. Then there exists a set $\Theta_\xi\subseteq\Li(\xi)$ with $\e^{2\xi}$-exceptional complement (in $\Li(\xi)$) such that for every $g\in\Theta_{\xi}$ $\tilde{\E}_g\|\tilde{X}_t\|_2^2<\infty$, $t\geq 0$, and for each $h\in L_2$ the process 
$$
(\tilde{M}_t,h)_{L_2}=(\tilde{X}_t,h)_{L_2}-(\tilde{X}_0,h)_{L_2}-\int_0^t(\xi-\pr_{\tilde{X}_s}\xi,h)_{L_2}ds,\quad t\geq 0,
$$
is a continuous square-integrable $(\tilde{\F}_t)$-martingale under $\tilde{\p}_g$ with quadratic variation
$$
[(\tilde{M}_{\cdot},h)_{L_2}]_t=\int_0^t\|\pr_{\tilde{X}_s}h\|_{L_2}^2ds,\quad t\geq 0.
$$ 
In particular, $\tilde{X}$ is a weak solution to SDE~\eqref{f_the_main_equation} on the probability space $(\tilde{\Omega},\tilde{\F},\tilde{\p}_g)$.
\end{proposition}

\begin{proof}
See Corollary~8.2~\cite{Konarovskyi_SR:2017} for the proof of the proposition.
\end{proof}

We remark that from the technical point of view the assumption that $\xi$ is constant in some neighborhoods of 0 and 1 is needed for the local compactness of the state space $\Li(\xi)$ for the reversible CFWD $\tilde{X}$ (see Lemma~6.8~\cite{Konarovskyi_SR:2017}). In that case, the proof of the strict quasi-regularity and conservativity of the Dirichlet form $(\e^{\xi},\Dom^{\xi})$ can be easily proved using methods from~\cite{Fukushima:2011}. From the intuitive point of view, the clusters in maximal and minimal spatial positions have bounded diffusion rates since they cannot fragment according to Lemma~\ref{lemma_coalescing}. Therefore, particles between these two clusters cannot run to infinity during a finite time, which explains the infiniteness of the life time.

In the rest of this section we suppose that $\xi=\sum_{k=1}^n\varsigma_k\I_{\pi_k}\in\St$, where $\varsigma_k<\varsigma_{k+1}$, $k\in[n-1]$, and $\{\pi_k,\ k\in[n]\}$ is a partition of $[0,1]$. 

Let $\tilde{X}(\cdot,t,\omega)$ denote the modification of $\tilde{X}_t(\omega)$ from $\Dr$ for each $\omega\in\tilde{\Omega}$ and $t\geq 0$. Since $\tilde{X}$ takes values in the space $\Li(\xi)$, it is easy to see that 
$$
\tilde{X}(u,t)=\sum_{k=1}^n\tilde{x}_k(t)\I_{\pi_k}(u),\quad u\in[0,1],\ \  t\geq 0,
$$
where $\tilde{x}_k(t)=\frac{1}{\leb(\pi_k)}(\tilde{X_t},\I_{\pi_k})_{L_2}$, by Proposition~A.2~\cite{Konarovskyi_SR:2017}. This yields that the process $\tilde{X}(u,t)$, $t\geq 0$, is continuous for every $u\in[0,1]$ . 

\begin{proposition}\label{proposition_existence_via_DF}
  The process $\{\tilde{X}(u,t),\ u\in[0,1],\ t\geq 0\}$ belongs to the Skorohod space $D([0,1],C[0,\infty))$ and for each $g\in\Theta_{\xi}$ it satisfies properties $(R1)-(R4)$ on the probability space $(\tilde{\Omega},\tilde{\F},\tilde{\p}_g)$.
\end{proposition}

\begin{proof}
The statement follows from Proposition~\ref{proposition_mart_prop_of_X_t} and the following property of the projection operator:
$$
(\pr_fh^u,\pr_fh^v)_{L_2}=\frac{\I_{\{f(u)=f(v)\}}}{m_f(u)}
$$
for all $u,v\in[0,1]$ and $f=\sum_{k=1}^ny_i\I_{\pi_k}\in\St$, where $h^u:=\frac{1}{\leb(\pi_k)}\I_{\pi_k}$ with $k$ satisfying $u\in\pi_k$ and $m_f(u):=\leb\{v:\ f(u)=f(v)\}$.

The detaled proof is omited here since we will prove Theorem~\ref{theorem_existence_in_general_case1}~(ii) in a more general setting later. 
\end{proof}

Note that according to Remark~\ref{rem_finite_particle_system}, the family of semimartingales $\tilde{X}(u,\cdot )$, $u \in [0,1]$, defines a finite sticky-reflected particle system whose evolution is described by the processes $\tilde{x}_k$, $k \in [n]$.

\subsection{Existence of solutions to SDE~\texorpdfstring{\eqref{f_the_main_equation}}{(10)} (proof of Theorem~\texorpdfstring{\ref{theorem_existence_in_general_case1}}{(10)}~(i))}%
\label{sub:existence_of_solutions_to_sde_f_the_main_equation_proof_of_theorem_theorem_existence_in_general_case_i_}

The goal of this section is to show that equation~\eqref{f_the_main_equation} has a weak solution for any initial condition $g \in L^{\uparrow}_{2+\delta}$ and $\xi \in L^{\uparrow}_{\infty}$, where $\delta$ as a positive number. This will prove Theorem~\ref{theorem_existence_in_general_case1}~(i).

\subsubsection{Tightness of finite particle systems}%
\label{ssub:tightness_of_a_finite_particle_system}

We recall that, by propositions~\ref{proposition_mart_prop_of_X_t} and~\ref{proposition_existence_via_DF}, for every $\xi\in\St$ and $g\in\Theta_{\xi}$ there exists a weak solution to SDE~\eqref{f_the_main_equation} satisfying $(R1)-(R4)$, where $\Theta_{\xi}$ is defined in Proposition~\ref{proposition_mart_prop_of_X_t}. Moreover, $\Theta_{\xi}$ is dense in $\Li(\xi)$, since $\Xi^{\xi}(\Li(\xi)\setminus\Theta_{\xi})=0$ and $\supp\Xi^{\xi}=\Li(\xi)$ (see Proposition~\ref{proposition_measure_Xi}). 

We fix $g \in L^{\uparrow}_{2+\delta}$ and $\xi \in L^{\uparrow}_{\infty}$ for some $\delta>0$.  In order to prove the existence of solutions to SDE~\eqref{f_the_main_equation}, we first construct sequences $\{g_n,\ n\geq 1\}\in\St$ and $\{\xi_n,\ n\geq 1\}\in\St$ such that $g_n\in\Theta_{\xi_n}$ for all $n\geq 1$, $g_n\to g$, $\xi_n\to\xi$ in $L_{2+\delta}$ and $\{\xi_n(1)-\xi_n(0),\ n\geq 1\}$ bounded. Set
$$
\xi_n:=\sum_{k=1}^{2^n}\left(\frac{k}{2^{2n}}+\xi\left(\frac{k-1}{2^n}\right)\right)\I_{\left[\frac{k-1}{2^n},\frac{k}{2^n}\right)}+\left(\frac{1}{2^n}+\xi(1)\right)\I_{\{1\}},\quad n\geq 1.
$$ 
Since $\xi$ is discontinuous at most in a countable number of points, $\xi_n\to\xi$ a.e. and, thus, it convergences in $L_{2+\delta}$, by the dominated convergence theorem. Moreover, $\xi_n(1)-\xi_n(0)=\xi(1)-\xi(0)+ \frac{1}{ 2^n }- \frac{1}{ 2^{2n} }$ for all $n\geq 1$, that implies the boundedness of $\{ \xi_n(1)-\xi_n(0),\ n\geq 1 \}$. We also note that 
$$
\Li(\xi_n)=\left\{f\in\Li:\ f\ \mbox{is}\ \sigma^{*}\left(\left\{\left[\frac{k-1}{2^n},\frac{k}{2^n}\right),\ k\in[2^n]\right\}\right)\mbox{-measurable}\right\},
$$ 
due to the term $\frac{k}{2^{2n}}$ in the definition of $\xi_n$ and the monotonicity of $\xi$. To construct $g_n$, $n\geq 1$, we first set
$$
\tilde{g}_n:=\pr_{\xi_n}g=\sum_{k=1}^ny_k^n\I_{\left[\frac{k-1}{2^n},\frac{k}{2^n}\right)},\quad n\geq 1,
$$
where 
$$
y_k^n=2^n\int_{\frac{k-1}{2^n}}^{\frac{k}{2^n}}g(v)dv.
$$
Since $\tilde{g}_n\in\Li(\xi_n)$, the set $\Theta_{\xi_n}$ is dense in $\Li(\xi_n)$ and $\|\cdot\|_{L_{2+\delta}}$ is equivalent to $\|\cdot\|_{L_2}$ in $L_2(\xi_n)$ because $L_2(\xi_n)$ is finite dimensional, for each $n\geq 1$ we can find $g_n\in\Theta_{\xi_n}$ satisfying $\|g_n-\tilde{g}_n\|_{L_{2+\delta}}<\frac{1}{n}$. 
%Again, $g_n\to g$ a.e. 
%Using the dominated convergence theorem and the inequality
%$$
%|g_n(u)-g(u)|\leq 2|g(u)|+1,\quad u\in(0,1),\ \  n\geq 1,
%$$
Using, e.g., Theorem~1~\cite{MR1619139} and Remark~\ref{remark_notes_obout_proj}~(ii), we obtain $g_n\to g$ in $L_{2+\delta}$.

Let $X^n$ be a weak solution to SDE~\eqref{f_the_main_equation} with $g$ and $\xi$ replaced by $g_n$ and $\xi_n$, which exists according to Proposition~\ref{proposition_mart_prop_of_X_t} and describes the evolution of a finite particle system (see also Remark~\ref{rem_finite_particle_system}). We recall that 
$$
X^n_t=g_n+M^n_t+A^n_t,\quad t\geq 0,
$$
where
$$
A^n_t=\int_0^t(\xi_n-\pr_{X_s^n}\xi_n)ds
$$
and $M^n$ is a continuous square-integrable $(\F^{X^n}_t)$-martingale in $L_2$ with quadratic variation process
$$
\langle\langle M^n \rangle\rangle_t=\int_0^t\pr_{X^n_s}ds, \quad t\geq 0,
$$
and the increasing process
$$
\langle M^n \rangle_t=\int_0^t\|\pr_{X^n_s}\|_{HS}^2ds, \quad t\geq 0.
$$
Let $\overline{X}^n$, $n\geq 1$, be defined by~\eqref{f_tilde_X}, that is,
\[
  \overline{X}^n=\left(M^n,A^n,([M^n(e_i),M^n(e_j)])_{(i,j)\in\N^2}, \langle M^n \rangle \right),
\]
for a fixed orthonormal basis $\{e_i,\ i\in\N\}$ of $L_2$.  Then, by Corollary~\ref{corollary_relatively_compactness_of_X}, there exists a subsequence $N\subseteq\N$ such that $\overline{X}^n\to\overline{X}$ in $\W$ in distribution along $N$, where 
$$
\W=C([0,\infty),L_2)\times C([0,\infty),L_2)\times C[0,\infty)^{\N^2}\times C[0,\infty).
$$

Next, by Skorohod representation Theorem~3.1.8~\cite{Ethier:1986}, we can find a probability space and define there random elements $\overline{Z},\ \overline{Z}^n$, $n\in N $, taking values in $\W$ such that $\law(\overline{X})=\law(\overline{Z})$, $\law(\overline{X}^n)=\law(\overline{Z}^n)$ and $\overline{Z}^n\to\overline{Z}$ in $\W$ a.s. along $N$. We also note that $\{\E\|M^{Z^n}_t\|_{L_2}^2=\E\|M^{X^n}_t\|_{L_2}^2,\ n\geq 1\}$ is bounded for all $t\geq 0$ , by Corollary~\ref{corollary_estim_of_norm_of_X} and Remark~\ref{remark_def_of_sol}~(iii). In the next section, we will construct a solution to~\eqref{f_the_main_equation} using the process $\overline{Z}$.

\subsubsection{Identification of the limit}
\label{sub:identification_of_the_limit}

To obtain a solution to equation~\eqref{f_the_main_equation} from the process $\overline{Z}$, constructed in the previous section, we will prove a kind of stability of solutions to equation~\eqref{f_the_main_equation} under passing to the limit. We assume that $\{g_n,\ n\geq 1\},\ \{\xi_n,\ n\geq 1\}$ are arbitrary sequences of functions from $L_2^{\uparrow}$ (not necessarily from $\St$) and processes $X^n,\ n\geq 1$, defined on the same probability space, are weak solutions to SDE~\eqref{f_the_main_equation} with initial conditions $g_n$ and interacting potentials $\xi_n$. As in the previous section we express $X^n$ in the form
$$
X^n_t=g_n+M^n_t+A^n_t,\quad t\geq 0,
$$
and define 
\[
  \overline{X}^n=\left(M^n,A^n,([M^n(e_i),M^n(e_j)])_{(i,j)\in\N^2}, \langle M^n \rangle \right),
\]
where $\{e_i,\ i\in\N\}$ is the fixed orthonormal basis of $L_2$.

\begin{theorem}\label{theorem_identification_of_limit}
  Let $\{g_n,\ n\geq 1\},\ \{\xi_n,\ n\geq 1\}$ converge to $g$ and $\xi$ in $L_2$, respectively. Let also the sequence of stochastic processes $\{\overline{X}^n,\ n\geq 1\}$ converge to $\overline{X}=(M,A,(x_{i,j}),a)$ in $\W$ a.s. and $\{\E\|M^n_t\|_{L_2}^2, n\geq 1\}$ is bounded for all $t\geq 0$. Then 
\begin{enumerate}
\item[(a)] the process $X_t:=g+M_t+A_t$, $t\geq 0$, takes values in $L_2^{\uparrow}$;

\item[(b)] $M$ is a continuous square-integrable martingales in $L_2$ with quadratic variation
\begin{equation}\label{f_var_M}
  \langle\langle M \rangle\rangle_t=\int_0^t\pr_{X_s}ds, \quad t\geq 0,
\end{equation}
in particular, 
\begin{equation}\label{f_x_i_j}
x_{i,j}(t)=\int_0^t(\pr_{X_s}e_i,e_j)_{L_2}ds,\quad i,j\in\N,
\end{equation}
and
\begin{equation}\label{f_a}
a(t)=\int_0^t\|\pr_{X_s}\|_{HS}^2ds
\end{equation}
for all $t\geq 0$;
\item[(c)] $A_t=\int_0^t(\xi-\pr_{X_s}\xi)ds$, $t\geq 0$;

\item[(d)] for each $T>0$, \ $\int_0^T\|\pr_{X^n_s}-\pr_{X_s}\|_{HS}^2ds\to 0$ a.s. as $n\to\infty$.
\end{enumerate}
In particular, $X$ is a weak solution to SDE~\eqref{f_the_main_equation}.
\end{theorem}

Theorem~\ref{theorem_identification_of_limit} will be proved using the deterministic result from Section~\ref{subsection_limit_prop_of_proj_process}. 
The following lemmas are needed to check that $X_t$, $t\in[0,T]$, satisfies the assumptions of Proposition~\ref{proposition_conv_of_proj} almost surely.

\begin{lemma}\label{lemma_condition_a}
Under the assumptions of Theorem~\ref{theorem_identification_of_limit}, for each $T>0$ there exists a random element $P^{\infty}$ in $\LHS$ such that
$$
\p\left\{P^n\to P^{\infty}\ \mbox{weakly in}\ \LHS\ \mbox{as}\ n\to\infty\right\}=1,
$$
where $P^n:=\pr_{X^n_{\cdot}}$, $n\geq 1$.
\end{lemma}

\begin{proof}
We set
\begin{align*}
\Omega'&:=\{\omega:\ \langle M^n\rangle(\omega)\to a(\omega)\ \mbox{in}\ C[0,T]\} \cap\{\omega:\ [M^n(e_i),M^n(e_j)](\omega)\to x_{i,j}(\omega)\ \mbox{in}\ C[0,T]\ \mbox{for all}\ i,j\in\N\}.
\end{align*}
It is obvious that $\p\{\Omega'\}=1$. We take $\omega\in\Omega'$ and show that there exists $P^{\infty}(\omega)\in\LHS$ \ such \ that \ $P^n(\omega)\to P^{\infty}(\omega)$ \ weakly \ in \ %??
$\LHS$ as $n\to\infty$. 
Since the sequence $$\langle M^n\rangle_T(\omega)=\int_0^T\|P^n_t(\omega)\|_{HS}^2dt=\|P^n(\omega)\|_{T,HS}^2,\quad n\geq 1,$$ converges to $a(T,\omega)$, it is bounded. Thus, by the Banach-Alaoglu theorem, $\{P^n(\omega),\ n\geq 1\}$ is weakly compact in $\LHS$. Moreover, it has a unique weak limit point denoted by $P^{\infty}(\omega)$. Indeed, if $\{P^n(\omega),\ n\geq 1\}$ weakly converges to $P'$ along $N'$ and to $P''$ along $N''$, then for each $i,j\in\N$ and $t\in[0,T]$
$$
\int_0^t(P'_s(\omega)e_i,e_j)_{L_2}ds=\int_0^t(P''_s(\omega)e_i,e_j)_{L_2}ds=x_{i,j}(\omega,t)
$$ 
because 
\begin{equation}\label{f_B_P}
(B^t,P^n(\omega))_{T,HS}=\int_0^t(P^n_s(\omega)e_i,e_j)_{L_2}ds\to x_{i,j}(\omega,t)
\end{equation}
for $B_s^t:=\I_{[0,t]}(s)e_i\otimes e_j$, $s\in[0,T]$.
Hence, Corollary~\ref{corollary_equality_of_HS_valued_maps} implies $P'=P''$.
Thus, $\{P^n(\omega),\ n\geq 1\}$ weakly converges in $\LHS$ to $P^{\infty}(\omega)$.

We note that the measurability of the map $P^{\infty}:\Omega\to\LHS$ (here, $P^{\infty}(\omega)=0$, if $\omega\not\in\Omega'$) will easily follow from the facts that $P^{\infty}$ is a weak limit of random elements in $\LHS$ and Theorem~II.1.1~\cite{Vakhania:1987}. The lemma is proved.
\end{proof}

\begin{lemma}\label{lemma_condition_b}
Under the assumptions of Theorem~\ref{theorem_identification_of_limit}, for each $T>0$
\begin{align*}
\p\Big\{\|P_th\|_{L_2}\leq\varliminf_{n\to\infty}\|P_t^nh\|_{L_2},\ \forall t\in [0,T],\ \forall h\in L_2\Big\}=1,
\end{align*}
%\p\Big\{\|P_th\|_{L_2}\leq\varliminf_{n\to\infty}\|P_t^nh\|_{L_2},\ \forall t\in [0,T],\ \forall h\in L_2\Big\}=1,
where $P:=\pr_{X_{\cdot}}$ and $P^n=\pr_{X^n_{\cdot}}$, $n\geq 1$.
\end{lemma}

\begin{proof}
  The lemma immediately follows from the convergence of $\{X_n,\ n\geq 1\}$ a.s. in $C([0,T],\Li)$ and Lemma~\ref{lemma_lower_semicontinuity_of_pr}.
%Let $T>0$ be fixed. We set for each $\omega\in\Omega$
%\begin{align*}
%R(\omega)=\{t\in[0,T]:\ \forall n\geq 1\  \|P^n_t(\omega)\|_{HS}<\infty\}
%\end{align*}
%and
%$$
%\Omega'=\{\omega:\ \leb([0,T]\setminus R(\omega))=0\}\cap\{\omega:\ X^n(\omega)\to X(\omega)\ \mbox{in}\ C([0,T],L_2)\}.
%$$
%Since for all $n\geq 1$ 
%$$
%\E\int_0^T\|P_t^n\|_{HS}^2dt=\E\|M^n_T\|_{L_2}^2<\infty
%$$ 
%and $X^n\to X$ in $C([0,T],L_2)$, we have that $\p\{\Omega'\}=1$. 

%Let $\omega\in\Omega'$ and $t\in R(\omega)$. Then by Lemma~\ref{lemma_lower_semicontinuity_of_pr},
%$$
%\|P_t(\omega)h\|_{L_2}\leq\varliminf_{n\to\infty}\|P_t^n(\omega)h\|_{L_2},
%$$
%for all $h\in L_2$. The lemma is proved.
\end{proof}

\begin{proof}[Proof of Theorem~\ref{theorem_identification_of_limit}]
Property $(a)$ easily follows from the closability of $L_2^{\uparrow}$ in $L_2$. 

In order to show $(b)$, we only have to check equality~\eqref{f_x_i_j}. Indeed, the convergence of $\{\overline{X}^n,\ n\geq 1\}$ to $\overline{X}=(M,A,x_{i,j},a)$ in $\W$ a.s. yields that the process $M$ is a continuous square-integrable $(\F^{X}_t)$-martingale in $L_2$ with $[M(e_i),M(e_j)]=x_{i,j}$ and $\langle M\rangle=a$. This easily follows from the fact that the weak limit in $C[0,\infty)$ of local martingales is a local martingale (see, e.g., Corollary~9.1.19~\cite{Jacod:2003}), the boundedness of $\{\E\|M^n_t\|_{L_2}^2, n\geq 1\}$ and Fatou's lemma. Therefore, \eqref{f_var_M} will follow from \eqref{f_x_i_j}, and Lemma~2.1~\cite{Gawarecki:2011} and~\eqref{f_var_M} will imply equality \eqref{f_a}.

By lemmas~\ref{lemma_condition_a},~\ref{lemma_condition_b} and Proposition~\ref{proposition_prop_of_quad_var_of_mart}, trajectories of $X$ and $X^n,\ n\geq 1$, satisfy conditions $(a)-(c)$ of Proposition~\ref{proposition_conv_of_proj} almost surely. Thus, 
\begin{equation}\label{f_weak_convergence}
\p\{P^n\to P\ \mbox{weakly in}\ \LHS\ \mbox{as}\ n\to\infty\}=1
\end{equation}
for any fixed $T>0$. This immediately implies~\eqref{f_x_i_j}, by~\eqref{f_B_P}.

Next,~\eqref{f_weak_convergence} and the convergence
$$
\|P^n\|_{T,HS}^2=\langle M^n\rangle_T\to a(T)=\langle M\rangle_T=\|P\|_{T,HS}^2\quad\mbox{a.s. \ as}\ \ n\to\infty
$$
yield the strong convergence $(d)$. 

For fixed $t\in[0,T]$ and $h\in L_2$ we take 
$$
B_s:=\I_{[0,t]}(s)\xi\otimes h,\quad B^n_s:=\I_{[0,t]}(s)\xi_n\otimes h,\quad s\in[0,T],\ \ n\geq 1.
$$ 
Then, using the strong convergence of $\{P^n,\ n\geq 1\}$ to $P$ and $\{B^n,\ n\geq 1\}$ to $B$, we have
$$
\int_0^t(\pr_{X^n_s}\xi_n,h)_{L_2}ds=(P^n,B^n)_{T,HS}\to (P,B)_{T,HS}=\int_0^t(\pr_{X_s}\xi,h)_{L_2}ds.
$$
Hence, for each $h\in L_2$
$$
(A_t,h)_{L_2}=\int_0^t(\xi-\pr_{X_s}\xi,h)ds,\quad t\in[0,T],
$$
that implies $(c)$. The theorem is proved.
\end{proof}

We now are ready to conclude that SDE~\eqref{f_the_main_equation} has a weak solution for any $g \in L^{\uparrow}_{2+\delta}$ and $\xi \in L^{\uparrow}_{\infty}$.  Let $\overline{Z}^n$, $n\geq 1$, and $\overline{Z}$ be the random elements in $\W$ defined in Section~\ref{ssub:tightness_of_a_finite_particle_system}. Then they satisfy the assumptions of Theorem~\ref{theorem_identification_of_limit}. Therefore, the process $X=g+A+M$ solves equation~\eqref{f_the_main_equation}, where $M$ and $A$ are the first and the second components of $\overline Z$, respectively. The proof of Theorem~\ref{theorem_existence_in_general_case1}~(i) is completed.

\subsection{Equivalence between two definitions of CFWD (Proof of Theorem~\ref{theorem_existence_in_general_case1}~(ii))}%
\label{sub:equivalence_between_two_formulations_of_solutions}

Let $Y=\{Y(u,t),\ u\in[0,1],\ t\in[0,\infty)\}$ be a random element in $D([0,1],C[0,\infty))$ and $X_t$, $t \geq 0$, be a continuous process in $\Li$ such that $\E \|X_t\|_{L_2}^2<\infty$ and 
$$
Y(\cdot,t)=X_t\quad\mbox{in}\ \ L_2\ \ \mbox{a.s.}
$$ 
for all $t\geq 0$. In this section, we will prove that $Y$ satisfies $(R1)-(R4)$ if and only if $X$ is a solution to SDE~\eqref{f_the_main_equation}.

\begin{remark}\label{remark_X_equals_Y}
Since the processes $X_t$, $t\geq 0$, and $Y(\cdot,t)$, $t\geq 0$, are continuous in $L_2$, we have 
$$
\p\{X_t=Y(\cdot,t)\ \ \mbox{in}\ \ L_2\ \ \mbox{for all}\ \ t\geq 0\}=1.
$$
In particular,
$$
\p\{\pr_{X_t}=\pr_{Y(\cdot,t)}\ \ \mbox{for all}\ \ t\geq 0\}=1.
$$
\end{remark}

\subsubsection{Auxiliary statements}%
\label{ssub:auxiliary_statements}

\begin{lemma}\label{lemma_codlag_prop_of_int}
Let $b:[0,1]\times[0,T]\to\R$ be a measurable bounded function such that the function $b(u,t)$, $u\in[0,1]$, is c\`{a}dl\'{a}g for all $t\in[0,T]$. Then the function
$$
B(u,t)=\int_0^tb(u,s)ds,\quad u\in[0,1],\ \ t\in[0,T],
$$
belongs to $D([0,1],C[0,T])$.
\end{lemma}

\begin{proof}
Let $u_n\downarrow u$. Then $B(u_n,t)\to B(u,t)$ for any $t\in[0,T]$, by the dominated convergence theorem and the right continuity of $b(\cdot,t)$ for all $t\in[0,T]$. Moreover, by the Arzela-Ascoli theorem, $\{B(u_n,\cdot),\ n\geq 1\}$ is compact in $C[0,T]$. Thus, $B(u_n,\cdot)\to B(u,\cdot)$ in $C[0,T]$. Similarly, $B(u_n,\cdot)\to B(u-,\cdot):=\int_0^{\cdot}b(u-,s)ds$ in $C[0,T]$ as $u_n\uparrow u$. The lemma is proved.
\end{proof}

We define for each $u\in[0,1)$ and $\eps>0$ the functions from $L_2$ as follows
$$
h^u_{\eps}(v)=\frac{1}{\eps\wedge(1-u)}\I_{[u,(u+\eps)\wedge 1)}(v),\quad v\in[0,1],
$$
and 
$$
h^1_{\eps}(v)=\frac{1}{\eps\wedge 1}\I_{[(1-\eps)\vee 0, 1]}(v),\quad v\in[0,1].
$$

\begin{lemma}\label{lemma_molification}
  Let a function $f(u,t)$, $u\in[0,1]$, $t\in[0,T]$, belong to $D([0,1],C[0,T])$. Then for each $u\in[0,1]$ the \ sequence \ of \  functions \ \ ${\{(f(\cdot,t),h^u_{\eps})_{L_2},\ t\in[0,T]\}_{\eps>0}}$ converges to $f(u,t)$, $t\in[0,T]$, in $C[0,T]$ as $\eps\to 0+$. 
\end{lemma}

\begin{proof}
We first note that for every $u\in[0,1)$ and $\tilde{\eps}>0$ there exists $\delta>0$ such that 
$$
|f(u,t)-f(v,t)|<\tilde{\eps},\quad t\in[0,T],\ \  v\in[u,u+\delta).
$$  
In particular, $f(v,t)$, $t\in[0,T]$, $v\in[u,u+\delta)$, is bounded. Hence, for each $\eps\in(0,\delta]$ the function $(f(\cdot,t),h^u_{\eps})_{L_2},\ t\in[0,T]$, belongs to $C[0,T]$, by the dominated convergence theorem, and
$$
|(f(\cdot,t),h^u_{\eps})_{L_2}-f(u,t)|\leq\int_0^1|f(v,t)-f(u,t)|h_{\eps}^u(v)dv<\tilde{\eps}.
$$
For $u=1$ the convergence follows from the same argument and the continuity of $f(v,\cdot)$, $v\in[0,1]$, at $v=1$.
This proves the lemma.
\end{proof}

\begin{lemma}\label{lemma_prog_estim}
Let $f\in \St$ and $0\leq u<v\leq 1$. Then 
\begin{equation}\label{f_conv_of_proj}
\left(\pr_fh_{\eps}^u,\pr_fh_{\eps}^v\right)_{L_2}\to\frac{1}{m_f(u)}\I_{\{f(u)=f(v)\}}\quad\mbox{as}\ \ \eps\to 0+
\end{equation}
and
\begin{equation}\label{boundedness_of_proj}
0\leq\left(\pr_fh_{\eps}^u,\pr_fh_{\eps}^v\right)_{L_2}\leq \begin{cases}
\frac{1}{v-u-\eps},& \eps\in\left(0,v-u\right),\ \ v<1,\\
\frac{1}{v-u-2\eps},& \eps\in\left(0,\frac{v-u}{2}\right),\ \ v=1,
\end{cases}
\end{equation}
where $m_f(u)=\leb\{v:\ f(u)=f(v)\}$.
\end{lemma}

\begin{proof}
Convergence~\eqref{f_conv_of_proj} follows from Lemma~\ref{lemma_view_of_pr} and a simple calculation.

We show~\eqref{boundedness_of_proj} only for $v<1$. For this, we fix $\eps\in(0,v-u)$ and consider the following two cases.

a) $f(u+\eps)<f(v)$. Then, by Lemma~\ref{lemma_view_of_pr}, $\supp(\pr_fh_{\eps}^u)\cap\supp(\pr_fh_{\eps}^v)=\emptyset$. This implies that 
$$
\left(\pr_fh_{\eps}^u,\pr_fh_{\eps}^v\right)_{L_2}=0.
$$

b) $f(u+\eps)=f(v)$. Let $\tilde{u}$, $\tilde{v}$ be the ends of the interval $\{r:\ f(v)=f(r)\}$ and $\tilde{u}<\tilde{v}$. Then, $\tilde{u}\leq u+\eps<v<\tilde{v}$. Moreover, 
$$
\left(\pr_fh_{\eps}^u\right)(r)\left(\pr_fh_{\eps}^v\right)(r)=0,\quad r\not\in[\tilde{u},\tilde{v}),
$$
and
\begin{align*}
\left(\pr_fh_{\eps}^u\right)(r)\left(\pr_fh_{\eps}^v\right)(r)&=\frac{1}{\eps^2(\tilde{v}-\tilde{u})^2}\int_{[\tilde{u},\tilde{v})}\I_{[u,u+\eps)}(r)dr\int_{[\tilde{u},\tilde{v})}\I_{[v,v+\eps)}(r)dr \leq\frac{1}{(\tilde{v}-\tilde{u})^2},\quad r\in[\tilde{u},\tilde{v}).
\end{align*}
Hence,
\begin{align*}
\left(\pr_fh_{\eps}^u,\pr_fh_{\eps}^v\right)_{L_2}\leq \frac{1}{(\tilde{v}-\tilde{u})^2}\int_0^1\I_{[\tilde{u},\tilde{v})}(r)dr\leq\frac{1}{v-u-\eps}.
\end{align*}
The lemma is proved.
\end{proof}

\subsubsection{Proof of Theorem~\ref{theorem_existence_in_general_case1}~(ii)}%
\label{ssub:proof_of_theorem_theorem_existence_in_general_case1_ii_}

%  {\colb Let $Y=\{ Y(u,t),\ u \in [0,1],\ t\geq 0 \}$ be a random element in the Skorohod space $D([0,1],C[0,\infty))$ and $X_t$, $t\geq 0$, be a continuous $\Li$-valued process, both defined on a probability space $(\Omega,\F,\p)$, such that for every $t\geq 0$ almost surely 
%  \[
%    Y(\cdot ,t)=X_t \quad \mbox{in}\ \ L_2.
%  \]
%  Since $Y(\omega)$ belongs to the Skorohod space for every $\omega \in \Omega$, it is easily seen that $Y(u,t,\omega)$, $u \in [0,1]$, $t \in [0,T]$, is bounded for each $T>0$. Thus, the dominated convergence theorem and the continuity of $Y$ in $t$ imply that $Y(\cdot,t,\omega)$, $t\geq 0$, is continuous in $L_2$ for every $\omega \in \Omega$. Therefore, we can conclude that 
%  \[
%    \p\left\{ Y(\cdot ,t)=X_t \quad \mbox{in}\ \ L_2,\ \ t\geq 0 \right\}=1.
%  \]
%}
  We first assume that the process $X_t$, $t \geq 0$, is a weak solution to SDE~\eqref{f_the_main_equation} with the martingale part $M:=M^X$ and the part of bounded variation $A:=A^X$ and check that $Y$ satisfies $(R1)-(R4)$. The idea of proof is similar to the proof of Theorem~6.4~\cite{Konarovskyi:2017:EJP}. Namely, we are going to approximate $Y(u,\cdot)$ by $\left\{\left(X,h^u_{\eps}\right)_{L_2}\right\}_{\eps>0}$. 

We note that property $(R1)$ is trivial.

Let
\begin{align*}
  A^Y(u,t):&=\int_0^t\left(\xi(u)-(\pr_{Y(\cdot,s)}\xi)(u)\right)ds=\int_{ 0 }^{ t } \left( \xi(u) - \frac{1}{ m_Y(u,s) }\int_{ \pi_Y(u,s)}\xi(v)dv \right)ds ,\quad u\in[0,1],\ \ t\geq 0,
\end{align*}
where the equality follows from Lemma~\ref{lemma_view_of_pr}.
By Lemma~\ref{lemma_codlag_prop_of_int}, $A^Y(u,t)$, $u\in[0,1]$, $t\in[0,T]$, belongs to $D([0,1],C[0,T])$ for any $T>0$ and, thus, $A^Y$ belongs to $D([0,1],C[0,\infty) )$. Hence, $M^Y:=Y-g-A^Y$ also belongs to $D([0,1],C[0,\infty))$.

Let $h^u_{\eps}$ be defined as before for each $u\in[0,1]$ and $\eps>0$. Then, by Lemma~\ref{lemma_molification} and Remark~\ref{remark_X_equals_Y},
\begin{equation}\label{f_conv_of_M_Y}
\left(M_{\cdot},h^u_{\eps}\right)_{L_2}\to M^Y(u,\cdot) \quad\mbox{in}\ \ C[0,T]\ \ \mbox{a.s. \ as}\ \ \eps\to 0+
\end{equation}
and
\begin{equation}\label{f_def_of_Y_via_lim}
\left(X_{\cdot},h^u_{\eps}\right)_{L_2}\to Y(u,\cdot) \quad\mbox{in}\ \ C[0,T]\ \ \mbox{a.s. \ as}\ \ \eps\to 0+
\end{equation}
for all $T>0$. Thus, $Y$ satisfies $(R2)$, by Proposition~A.1~\cite{Konarovskyi:2017:EJP}\footnote{We remind the reader that the proposition claims that a function $f \in L_2$ belongs to $\Li$ if and only if for every $u,v \in (0,1)$ and $\eps,\delta>0$ satisfying $u+\eps\leq v$ one has $( f,h^u_{\eps} )_{L_2}\leq ( f,h^v_{\delta} )_{L_2}$. Moreover, the unique modification $\tilde{f}$ of $f$ is given by 
\[
  \tilde{f}(u)=\lim_{ \eps \to 0+ }( f,h^u_{\eps} )_{L_2}, \quad u \in (0,1).
\]
} and Remark~\ref{remark_X_equals_Y}. We also note that~\eqref{f_def_of_Y_via_lim} yields $\F^Y_t=\F^X_t$ for all $t\geq 0$.

Taking arbitrary $u\in(0,1)$, $\eps_0\in(0,u\wedge(1-u))$ and using Proposition~A.1~\cite{Konarovskyi:2017:EJP}, we have for every $t\geq 0$ and $\eps \in (0,\eps_0]$
$$
|(X_t,h^u_{\eps})_{L_2}|\leq |(X_t,h^0_{\eps_0})_{L_2}|+|(X_t,h^1_{\eps_0})_{L_2}|
$$
and
$$
|(g,h^u_{\eps})_{L_2}|\leq |(g,h^0_{\eps_0})_{L_2}|+|(g,h^1_{\eps_0})_{L_2}|.
$$
Hence, 
\begin{align*}
  |(M_t,h^u_{\eps})_{L_2}|&\leq |( g,h^u_{\eps} )_{L_2}|+|(X_t,h^u_{\eps})_{L_2}|+|( A_t,h^u_{\eps} )_{L_2}|\\
&\leq|(g,h^0_{\eps_0})_{L_2}|+|(g,h^1_{\eps_0})_{L_2}| +|(X_t,h^0_{\eps_0})_{L_2}|+|(X_t,h^1_{\eps_0})_{L_2}|+2t\|\xi\|_{L_{\infty}}
\end{align*}
for all $t\geq 0$ and $\eps\in(0,\eps_0]$. To estimate $( A_t,h^u_{\eps} )_{L_2}$, we used the inequality $\|\pr_{X_t}\xi\|_{L_{\infty}}\leq\|\xi\|_{L_{\infty}}$ which follows from Lemma~\ref{lemma_view_of_pr}. By~\eqref{f_conv_of_M_Y}, the fact that $\left(M_t,h^u_{\eps}\right)_{L_2}$, $t\geq 0$, is a square-integrable $(\F^Y_t)$-martingale for every $\eps>0$, the dominated convergence theorem and the finiteness of $\E \|X_t\|_{L_2}^2$, we have that $M^Y(u,t)$, $t\geq 0$, is a square-integrable $(\F^Y_t)$-martingale. In the case $u\in\{0,1\}$, using the convergence of $\{(M_{\cdot},h^u_{\eps})\}_{\eps>0}$ to $M^Y(u,\cdot)$ and $\{(X_{\cdot},h^v_{\eps})\}_{\eps>0}$ to $Y(v,\cdot)$ in $C[0,T]$ a.s. for every $v \in [0,1]$, one can show that $M^Y(u,t)$, $t\geq 0$, is a local $(\F^Y_t)$-martingale similarly to the proof of Proposition~9.1.17~\cite{Jacod:2003}.
This proves $(R3)$.

Next,~\eqref{f_conv_of_M_Y}, Lemma~B.11~\cite{Engelbert:2005} and the polarization formula for joint quadratic variation of martingales yield 
$$
\left[\left(M_{\cdot},h^u_{\eps}\right)_{L_2},\left(M_{\cdot},h^v_{\eps}\right)_{L_2}\right]\to\left[M^Y(u,\cdot),M^Y(v,\cdot)\right]
$$
in $C[0,\infty)$ in probability as $\eps\to 0+$ for all $u,v\in(0,1)$.

By the finiteness of the expectation $\E\|M_t\|_{L_2}^2<\infty$, the equality
$$
\E\int_0^t\|\pr_{Y(\cdot,s)}\|_{HS}^2ds=\E\int_0^t\|\pr_{X_s}\|_{HS}^2ds=\E\|M_t\|_{L_2}^2<\infty,\quad t\geq 0,
$$
and Lemma~\ref{lemma_connection_HS_with_int}, we have
\begin{equation}\label{f_Y_in_St}
\p\left\{\exists R\subseteq[0,\infty)\ \ \mbox{s.t.}\ \ \leb([0,\infty)\setminus R)=0\ \ \mbox{and}\ \ Y(\cdot,t)\in\St\ \ \forall t\in R\right\}=1.
\end{equation}
Thus, applying Lemma~\ref{lemma_prog_estim} to $f=Y(\cdot,t,\omega)$ and using the dominated convergence theorem, we obtain
\begin{align*}
\left[\left(M_{\cdot},h^u_{\eps}\right)_{L_2},\left(M_{\cdot},h^v_{\eps}\right)_{L_2}\right]_t&=\int_0^t\left(\pr_{Y(\cdot,s)}h_{\eps}^u,\pr_{Y(\cdot,s)}h_{\eps}^v\right)_{L_2}ds\ \  \to\ \ \int_0^t\frac{\I_{\{Y(u,s)=Y(v,s)\}}}{m_Y(u,s)}ds\quad\mbox{a.s.\ as}\ \ \eps\to 0+
\end{align*}
for any $t\geq 0$.
This implies $(R4)$ for all $u,v\in[0,1]$, $u\not=v$.

To finish the proof of the theorem, we have to check $(R4)$ for $u=v\in[0,1]$. 
Since $M^Y\in D([0,1],C[0,\infty))$, we have
$$
M^Y(v,\cdot)\to M^Y(u,\cdot) \quad\mbox{in}\ \ C[0,\infty)\ \ \mbox{a.s.}
$$
as $v\downarrow u$, if $u<1$, and $v\uparrow u$, if $u=1$. Thus, by Lemma~B.11~\cite{Engelbert:2005} and the polarization formula for joint quadratic variation of martingales, 
$$
\left[M^Y(v,\cdot),M^Y(u,\cdot)\right]\to \left[M^Y(u,\cdot)\right] \quad\mbox{in}\ \ C[0,\infty)\ \ \mbox{in probability}.
$$
Using~\eqref{f_Y_in_St}, Lemma~\ref{lemma_view_of_pr} and the monotone convergence theorem, it is easily seen that for each $t\geq 0$ 
$$
\left[M^Y(v,\cdot),M^Y(u,\cdot)\right]_t=\int_0^t\frac{\I_{\{Y(u,s)=Y(v,s)\}}}{m_Y(u,s)}ds\to\int_0^t\frac{ds}{m_Y(u,s)}
$$
as $v\downarrow u$, if $u<1$, and $v\uparrow u$, if $u=1$. The firs part of the statement is proved.

If $Y$ satisfies $(R1)-(R4)$, then a direct computation shows that $X$ satisfies property $(E'4)$ of Remark~\ref{remark_def_of_sol}. This immediately implies that $X$ solves equation~\eqref{f_the_main_equation}, that completes the proof of Theorem~\ref{theorem_existence_in_general_case1}~(ii).

\subsection{CFWD as a family of semimartingales (proof of Theorem~\ref{theorem_existence_in_general_case})}%
\label{sub:cfwd_as_a_family_of_semimartingales_proof_of_theorem_theorem_existence_in_general_case_}

  Let $g,\xi \in D([0,1],\R )$ be non-decreasing piecewise $ \frac{1}{ 2 }+$-H\"older continuous functions on $[0,1]$. Then there exist $\gamma> \frac{1}{ 2 }$, an ordered partition $0=\tilde{u}_0<\tilde{u}_1<\dots<\tilde{u}_l=1$ and a constant $C>0$ such that
$$
|g(u)-g(v)|\vee|\xi(u)-\xi(v)|\leq C|u-v|^{\gamma},\ \ u,v\in(\tilde{u}_{i-1},\tilde{u}_i),\ \ i\in[l].
$$ 
To construct a random element in the Skorohod space $D([0,1],C[0,\infty))$ satisfying conditions $(R1)-(R4)$, we will show the existence of a solution to equation~\eqref{f_the_main_equation} which has a modification from $D([0,1],C[0,\infty))$. The approach will be similar to the proof of Theorem~\ref{theorem_existence_in_general_case1}~(ii), but now we have to build sequences $\{g_n,\ n\geq 1\}\in\St$ and $\{\xi_n,\ n\geq 1\}\in\St$ which also satisfy conditions $(c1)$, $(c2)$ of Proposition~\ref{proposition_tightness_in_D_pi} on every interval $(\tilde{u}_{i-1},\tilde{u}_i)$, $i \in [l]$.

Letting $u_{i,j}^n$, $j=0,\ldots,2^n$, be the uniform partition of $[\tilde{u}_{i-1},\tilde{u}_i]$ for each $i\in[l]$ and $n\geq 1$, we define the functions $\iota_n:[0,1]\to[0,1]$ as follows
$$
\iota_n(u)=\sum_{i=1}^l\sum_{j=1}^{2^n}u_{i,j-1}^n\I_{[u_{i,j-1}^n,u_{i,j}^n)}(u)+u_{l,2^n-1}^n\I_{\{u_{l,2^n}^n\}}(u),\quad u\in[0,1],\ \ n\geq 1.
$$  
Then the functions $\iota_n$ belong to $D^{\uparrow}$ and map $[\tilde{u}_{i-1},\tilde{u}_i)$ into $[\tilde{u}_{i-1},\tilde{u}_i)$ for any $i\in[l]$. We put for every $n\geq 1$
$$
\xi_n:=\left(\xi+\frac{1}{n}\id\right)\circ\iota_n\footnote{The function $\frac{1}{n}\id$ is needed here in order to have $\sigma^{*}(\xi_n)=\sigma^{*}(\left[u_{i,j-1},u_{i,j}\right),\ j\in[2^n],\ i\in[l])$. Note that  it can be replaced by, e.g., any strictly increasing $\gamma$-H\"older continuous function.}\quad\mbox{and}\quad \tilde{g}_n:=g\circ\iota_n,
$$
where $\id$ denotes the identity function on $[0,1]$. As before, for each $n\geq 1$ we can take $g_n\in\Theta_{\xi_n}$ such that $\|g_n-\tilde{g}_n\|_{L_{\infty}}<\frac{1}{2^{\gamma n}}$. By the boundedness of $g$ and $\xi$ and the dominated convergence theorem, $g_n\to g$ and $\xi_n\to\xi$ in $L_{2+\delta}$. Moreover, $\{\xi_n(1)-\xi_n(0),\ n\geq 1\}$ is bounded. 

Next, we check $(c1)$ and $(c2)$ for every $(a,b):=(\tilde{u}_{i-1},\tilde{u}_i)$, $i \in [l]$. We take $u\in(a,b)$ and $r>0$ such that $u+r,u-r\in(a,b)$ and estimate $[g_n(u+r)-g_n(u)][g_n(u)-g_n(u-r)]$ for each $n\geq 1$. First, we note that for $r<\frac{b-a}{2^{n+1}}$
$$
[g_n(u+r)-g_n(u)][g_n(u)-g_n(u-r)]=0,
$$
since $g_n(u+r)-g_n(u)=0$ or $g_n(u)-g_n(u-r)=0$. Next, let $r\geq\frac{b-a}{2^{n+1}}$. Then 
\begin{align*}
[g_n(u+r)-g_n(u)][g_n(u)-g_n(u-r)] &\leq\left[\tilde{g}_n(u+r)-\tilde{g}_n(u)+\frac{2}{2^{\gamma n}}\right]\left[\tilde{g}_n(u)-\tilde{g}_n(u-r)+\frac{2}{2^{\gamma n}}\right]\\
&=\left[g(\iota_n(u+r))-g(\iota_n(u))+\frac{2}{2^{\gamma n}}\right]\left[g(\iota_n(u))-g(\iota_n(u-r))+\frac{2}{2^{\gamma n}}\right]\\
&\leq\left[C(\iota_n(u+r)-\iota_n(u))^{\gamma}+\frac{2}{2^{\gamma n}}\right]\left[C(\iota_n(u)-\iota_n(u-r))^{\gamma}+\frac{2}{2^{\gamma n}}\right]\\
&\leq \left(3^{\gamma}C+\frac{2^{\gamma+1}}{(b-a)^{\gamma}}\right)^2r^{2\gamma},
\end{align*}
since
$$
[\iota_n(u+r)-\iota_n(u)]\vee[\iota_n(u)-\iota_n(u-r)]\leq r+\frac{b-a}{2^n}\leq 3r\quad\mbox{for}\ r\geq\frac{b-a}{2^{n+1}}.
$$
Similarly, 
\begin{align*}
[g_n(u+r)-g_n(u)][\xi_n(u)-\xi_n(u-r)]&\leq \tilde{C}r^{2\gamma},\\
[\xi_n(u+r)-\xi_n(u)][g_n(u)-g_n(u-r)]&\leq \tilde{C}r^{2\gamma},\\
[\xi_n(u+r)-\xi_n(u)][\xi_n(u)-\xi_n(u-r)]&\leq \tilde{C}r^{2\gamma}.
\end{align*}
Thus, $\{g_n,\ n\geq 1\}$ and $\{\xi_n,\ n\geq 1\}$ satisfy~$(c1)$ of Proposition~\ref{proposition_tightness_in_D_pi} with $\beta=2\gamma>1$ and $\pi=[\tilde{u}_{i-1},\tilde{u}_i]$.

Estimate~$(c2)$ can be proved similarly, using the H\"older continuity of $g$ and $\xi$ on $(\tilde{u}_{i-1},\tilde{u}_i)$ and the form of the maps $\iota_n$, $n\geq 1$.

As before, let $X^n_t$, $t\geq 0$, be a weak solution to SDE~\eqref{f_the_main_equation} with $g$ and $\xi$ replaced by $g_n$ and $\xi_n$. Let $\overline{X}^n$, $n\geq 1$, be defined by~\eqref{f_tilde_X}. Let also $\{Y^n(u,t),\ u\in[0,1],\ t\geq 0\}$, $n\geq 1$, be random elements in $D([0,1],C[0,\infty))$ satisfying $(R1)-(R4)$ and  $X_t^n=Y^n(\cdot,t)$ in $L_2$ for all $t\geq 0$ almost surely. Such random elements exists by Proposition~\ref{proposition_existence_via_DF}. 

According to propositions~\ref{proposition_tightness_in_W},~\ref{proposition_tightness_in_D_pi},~\ref{proposition_tightnes_in_D} and Remark~\ref{remark_tightness_on_0_infty}, the sequence $\{(\overline{X}^n,Y^n),\ n\geq 1\}$ is tight in $\W\times D([0,1],C[0,\infty))$. As before, we can find a subsequence $N\subseteq\N$, a probability space and a \ family \ of \ random \ elements $\{(\overline{Z},\overline{V}),\ (\overline{Z}^n,\overline{V}^n),\ n\in N\}$ on this probability space such that $\law(\overline{X}^n,Y^n)=\law(\overline{Z}^n,V^n)$, $\law(\overline{X},Y)=\law(\overline{Z},V)$ and $(\overline{Z}^n,V^n)\to (\overline{Z},V)$ in $\W\times D([0,1],C[0,\infty))$ a.s. along $N$. Then, by Proposition~\ref{theorem_identification_of_limit} and Corollary~\ref{corollary_estim_of_norm_of_X}, the process $Z:=g+A^Z+M^Z$ is a weak solution to SDE~\eqref{f_the_main_equation}, where $\overline{Z}=(M^Z,A^Z,(x_{i,j}^Z),a^Z)$. Moreover, $Z_t=V(\cdot,t)$ in $L_2$ for all $t\geq 0$ almost surely. Hence, $V$ satisfies $(R1)$-$(R4)$, by Theorem~\ref{theorem_existence_in_general_case1}~(ii). The theorem is proved.

\appendix

\section{Appendix}

\subsection{The sitting time at zero of non-negative semimartingales} 
\label{sub:the_sitting_time_at_zero}

\begin{proposition}\label{proposition_estom_of_sitting_time}
Let $(\F_t)_{t\geq 0}$ be a complete right-continuous filtration and $y(t)$, $t\geq 0$, be a continuous non-negative $(\F_t)$-semimartingale such that
\begin{equation}\label{f_process_Y}
  y(t)=y_0+\int_0^t\rho(s)\I_{\{y(s)>0\}}dB(s)+\xi_0\int_0^t\I_{\{y(s)=0\}}ds, \quad t\geq 0,
\end{equation}
where $y_0$, $\xi_0$ are non-negative constants, $\rho(t)$, $t\geq 0$, is an $(\F_t)$-predictable process taking values in $[1,\infty)$ and satisfying 
\[
  \int_{ 0 }^{ t } \rho(s)^2 \I_{\left\{ y(s)>0 \right\}}ds<\infty  \quad \mbox{a.s.}
\]
for all $t>0$, and $B(t)$, $t\geq 0$, is an $(\F_t)$-Brownian motion. Then for every $t\geq 0$
$$
\E\int_0^t\I_{\{y(s)>0\}}ds\leq\sqrt{\frac{2t}{\pi}}(\xi_0t+y_0).
$$
\end{proposition}

\begin{proof}
We set
 $$
 R_t:=\int_0^t\I_{\{y(s)>0\}}ds,\quad t\geq 0,
 $$
 and use the idea from~\cite[P.~998-999]{Engelbert:2014} in order to estimate $\E R_t$. We will consider two cases.
 
 {\it Case I: $\xi_0>0$.}
  
 To estimate $\E R_t$, we first show that $R_t$, $t\geq 0$, is strictly increasing. Let us assume that it is not true, i.e,
 \[
   \p\left\{ \exists t_1<t_2\ \ R_{t_1}=R_{t_2} \right\}>0.
 \]
 By the monotonicity of $R_t$, $t\geq 0$, we can conclude that there exist (non-random) $t_1,t_2\geq 0$, $t_1<t_2$, such that 
 \[
   \p\left\{ R_{t_1}=R_{t_2} \right\}>0.
 \]
 Therefore, by the continuity of $y$, $y(r)=0$, $r \in [t_1,t_2]$, with positive probability. Hence 
 \begin{align*}
   \int_{ 0 }^{ t } \rho(s)\I_{\left\{ y(s)>0 \right\}}dB(s)&= y_0-\xi_0 \int_{ 0 }^{ t } \I_{\left\{ y(s)=0 \right\}}ds = y_0-\xi_0 \int_{ 0 }^{ t_1 }\I_{\left\{ y(s)=0 \right\}}ds -\xi_0 (t-t_1) , \quad t \in [t_1,t_2],
 \end{align*}
 with positive probability.  But this contradicts the fact that every continuous local martingale with bounded variation is constant (see, e.g., Theorem~17.2~\cite{Kallenberg:2002}). Consequently, the continuous process $R_t$, $t\geq 0$, is strictly increasing a.s. 
 
 We next define 
 \[
   R_{\infty}:=\lim_{ t\to\infty }R_t
 \]
 and 
 $$
 A_t:=\inf\{s\geq 0:\ R_s>t\}, \quad t\geq 0,
 $$
 which is a continuous strictly increasing process on $[0,R_{\infty})$ and $A_t \to \infty$ as $t$ increases to $R_{\infty}$. Note that $A_t$ is an $(\F_t)$-stopping time for each $t\geq 0$. Let $T>0$ be fixed. Then for the continuous local $(\F_t)$-martingale
 \[
   N^{\rho}_t:=\int_{ 0 }^{ t } \rho(s)\I_{\left\{ y(s)>0 \right\}}dB(s), \quad t\geq 0, 
 \]
 the process 
 \[
   N'_t:=N^{\rho}_{A_t\wedge T}=\int_0^{A_t\wedge T}\rho(s)\I_{\{y(s)>0\}}dB(s),\quad t\geq 0.
 \]
 is a continuous local $(\F'_t)$-martingale, where $\F'_t:=\F_{A_t}$. 
 
 Denoting 
 $$
 Q_t:=\int_0^{R_t}\rho(A_s)^2ds,\quad t\geq 0,
 $$
 and using the change of variables formula, we can see that
 $$
 Q_t=\int_0^t\rho(A_{R_s})^2dR_s=\int_0^t\rho(A_{R_s})^2\I_{\{y(s)>0\}}ds=\int_0^t\rho(s)^2\I_{\{y(s)>0\}}ds
 $$
 for all $t\geq 0$, since $A_t$, $t\geq 0$, is the inverse for $R_t$, $t\geq 0$. Thus,
 \begin{equation}\label{f_N_prime}
     [N']_t=[N^{\rho}]_{A_t\wedge T}=\int_0^{A_t\wedge T}\rho(s)^2\I_{\{y(s)>0\}}ds =Q_{A_t\wedge T}=\int_0^{R_{A_t\wedge T}}\rho(A_s)^2ds=\int_0^{t\wedge R_T}\rho(A_s)^2ds
 \end{equation}
 for any $t\geq 0$.
 
 Next, by Skorohod Lemma~22.2~\cite{Kallenberg:2002}, we directly get 
 \[
   \xi_0 \int_{ 0 }^{ t } \I_{\left\{ y(s)=0 \right\}}ds = \left\{-y_0-\inf_{s\leq t}N_s^{\rho} \right\}\vee 0, \quad t\geq 0.
 \]
 Hence, 
 \begin{align*}
   t&= \int_{ 0 }^{ t } \I_{\left\{ y(s)>0 \right\}}ds+\int_{ 0 }^{ t } \I_{\left\{ y(s)=0 \right\}}ds = R_t+ \frac{1}{ \xi_0 }\left\{-y_0-\inf_{s\leq t}N_s^{\rho} \right\}\vee 0, \quad t\geq 0.
 \end{align*}
 Using the equality $A_{R_t}=t$ for all $t\geq 0$, one has for every $t \in [0,T]$
 \begin{align*}
   R_t&= \max\left\{s:\ s+\frac{1}{\xi_0}\left[-y_0-\inf_{r\leq A_s}N^{\rho}_r\right]\vee 0\leq t\right\} = \max\left\{s:\ s+\frac{1}{\xi_0}\left[-y_0-\inf_{r\leq s}N'_r\right]\vee 0\leq t\right\}\\
   &=\max\left\{s:\ \left[-y_0-\inf_{r\leq s}N'_r\right]\vee 0\leq \xi_0(t-s)\right\} =\max\left\{s:\ -y_0-\inf_{r\leq s}N'_r\leq \xi_0(t-s),\ s\leq t\right\}\\
   &\leq\max\left\{s:\ \sup_{r\leq s}(-N'_r)\leq \xi_0t+y_0\right\}\wedge t.
 \end{align*}
 By Theorem~2.7.2'~\cite{Watanabe:1981:en}, there exists a Brownian motion $W(t)$, $t\geq 0$, defined probably on an extended probability space, such that 
 \[
   -N'_t=W([N']_t), \quad t\geq 0.
 \]
 We remark that $[N']_t=\int_{ 0 }^{ t\wedge R_T }\rho(A_s)^2ds\geq t\wedge R_T  $, $t\geq 0$. Hence,
 \begin{align*}
   R_t&\leq \max\left\{s:\ \sup_{r\leq s}W([N']_r)\leq \xi_0t+y_0\right\}\wedge t \leq \max\left\{s:\ \sup_{r\leq s}W(r)\leq \xi_0t+y_0\right\}\wedge t\\
   &\leq \sup\left\{s:\ \sup_{r\leq s}W(r)\leq \xi_0t+y_0\right\}\wedge t=\sigma_{\xi_0t+y_0}\wedge t,
 \end{align*}
 where  $\sigma_a:=\inf\{t:\ W(t)=a\}$. Therefore, 
 $$
 \E R_t\leq\E (\sigma_{\xi_0t+y_0}\wedge t)\leq\sqrt{\frac{2t}{\pi}}(\xi_0t+y_0).
 $$
 
 {\it Case II: $\xi_0=0$.}
 
 In this case, $y(t)=y_0+\int_0^t\rho(s)\I_{\{y(s)>0\}}dB(s)$, $t\geq 0$, is a continuous positive martingale. It implies that $y$ stays at zero for all $t\geq\tau^y_{y_0}:=\inf\left\{t:\ -\int_0^t\rho(s)\I_{\{y(s)>0\}}dB(s)=y_0\right\}$. Hence, using Lemma~2.4~\cite{Konarovskyi:2014:arx} and the fact that $\int_0^t\rho(s)\I_{\{y(s)>0\}}dB(s)=\int_0^t\rho(s)dB(s)$ for all $t\in\left[0,\tau^y_{y_0}\right]$, we have
 $$
 \E R_t=\E \tau^y_{y_0}\leq\E(\sigma_{y_0}\wedge t)\leq\sqrt{\frac{2t}{\pi}}y_0.
 $$
 
 Combining these two cases, we obtain the estimate
 $$
 \E R_t\leq\sqrt{\frac{2t}{\pi}}(\xi_0t+y_0),\quad t\geq 0.
 $$
 The proposition is proved.
\end{proof}

\subsection{The projection operator}\label{section_properties_of_pr}

We recall that for every $g \in L_2$ the completion of the $\sigma$-field generated by $g$ with respect to the Lebesgue measure $\leb$ is denoted by $\sigma^{*}(g)$.  We also denote the projection operator in $L_2$ on the closed linear subspace
$$
L_2(g)=\{f\in L_2:\ f\ \mbox{is}\ \sigma^{*}(g)\mbox{-measurable}\}
$$
by $\pr_g$.

\begin{remark}\label{remark_notes_obout_proj}
\begin{enumerate}
  \item[(i)] The operator $\pr_g$ is well-defined, since for two functions $g_1$ and $g_2$ coinciding a.e. the equality $\sigma^{*}(g_1)=\sigma^{*}(g_2)$ holds.

\item[(ii)] For each $h\in L_2$, $\pr_gh$ coincides a.e. with the conditional expectation $\E(h|\sigma(g))$ on the probability space $([0,1],\B([0,1]),\leb)$, where $\B([0,1])$ denotes the Borel $\sigma$-field on $[0,1]$.
\end{enumerate} 
\end{remark}

Recall that $D^{\uparrow}$ denotes the set of all non-decreasing c\`{a}dl\'{a}g functions from $(0,1)$ to $\R$. For fixed $g\in D^{\uparrow}$ we will denote the family of intervals
$I(c)=g^{-1}(\{c\})=\{u:\ g(u)=c\}$, $c\in\R$, satisfying $\leb(I(c))>0$ by $\K_g$. We note that either $I_1\cap I_2=\emptyset$ or $I_1=I_2$ for any $I_1,I_2\in\K_g$. This implies that $\K_g$ is countable. Let
$$
G_g:=\bigcup_{I\in\K_g}I\quad\mbox{and}\quad F_g:=(0,1)\setminus G_g.
$$ 
%where $I^{\circ}$ denotes the interior of $I$, that is, it coincide with the interval $I$ without its end points, and $\tilde{\partial} I$ contains only the end points of $I$ belonging to $I$, i.e., $\tilde{\partial} I=\partial I\cap I$. Since $\K_g$ is countable, $R_g$ is also countable and, trivially, $\leb(R_g)=0$. 
For any function $h\in L_2$ we define the function
\begin{equation}\label{f_explicit_formula_of_pr}
h_g(u):=\begin{cases}
           \frac{1}{\leb(I)}\int_Ih(v)dv,&u\in I\in\K_g,\\
           h(u),& u\in F_g,
          \end{cases}\quad u\in(0,1).
\end{equation}

\begin{lemma}\label{lemma_view_of_pr}
Let $g\in D^{\uparrow}$ and $h\in L_2$. Then $\pr_gh=h_g$ a.e. 
\end{lemma}

\begin{proof}
In order to prove the lemma, we first show that there exists a Borel function $\varphi:\R\to\R$ such that 
$$
h_g=\varphi(g).
$$
This will imply the measurability of $h_g$ with respect to $\sigma^{*}(g)$.

Since $g$ is a non-decreasing function, the restriction $g|_{F_g}$ of $g$ to the Borel set $F_g$ is an one-to-one map from $F_g$ to $g(F_g)=\{g(u):\ u\in F_g\}$. By Kuratowski's theorem (see Theorem~A.10.5~\cite{Ethier:1986}), $g(F_g)$ is a Borel subset of $\R$ and $\left(g|_{F_g}\right)^{-1}$ is a Borel measurable function from $g(F_g)$ to $F_g$. Thus, we define
$$
\varphi(x)=h\left(\left(g|_{F_g}\right)^{-1}(x)\right),\quad x\in g(F_g).
$$

If $x\in g((0,1))\setminus g(F_g)$, then there exists an unique interval $I_x\in\K_g$ such that $g(u)=x$ for all $u\in I_x$. 
%Indeed, if $x\in g((0,1))\setminus g(F_g)$, then for each $u\in(0,1)$ satisfying $g(u)=x$ we have $u\not\in F_g$. This immediately imply that $u$ belongs to an interval $I\in\K_g$, moreover, all $u$ should belong to the same interval. 
Hence, we can define
$$
\varphi(x)=\begin{cases}
            \frac{1}{\leb(I_x)}\int_{I_x}h(v)dv,& x\in g((0,1))\setminus g(F_g),\\
            0,& x\not\in g((0,1)).
           \end{cases}
$$
By the construction of $\varphi$, it is easy to see that $\varphi$ is a Borel function and for all $u\in(0,1)$ 
$$
\varphi(g(u))=h_g(u).
$$

Next, taking an arbitrary $\sigma^{*}(g)$-measurable function $f\in L_2$ and noting that there exists a Borel function $\psi:\R\to\R$ such that $f=\psi(g)$ a.e., we can estimate the norm $\|f-h\|_{L_2}^2$ from below as follows
\begin{align*}
\int_0^1(f(u)-h(u))^2du&=\int_0^1(\psi(g(u))-h(u))^2du\geq\sum_{I\in\K_g}\int_I(\psi(c_I)-h(u))^2du\\
&\geq \sum_{I\in\K_g}\int_I\left(\frac{1}{\leb(I)}\int_Ih(v)dv-h(u)\right)^2du =\int_0^1(h_g(u)-h(u))^2du,
\end{align*}
where $c_I=g(u)$, $u\in I$, and the last inequality is obtained by minimising the map
$$
\theta\mapsto\int_I(\theta-h(u))^2du.
$$
This completes the proof of the lemma.
\end{proof}

As before, we denote the (closed) subset of functions from $L_2$ which have a non-decreasing modification by $\Li$.

\begin{lemma}\label{lemma_monotonisity_of_pr}
For each $g\in D^{\uparrow}$ the projection operator $\pr_g$ maps $L_2^{\uparrow}$ into $L_2^{\uparrow}$.
\end{lemma}

The statement easily follows from the explicit formula~\eqref{f_explicit_formula_of_pr} for $\pr_gh$.

Let $g:(0,1)\to\R$ be a non-decreasing function. We define  
\begin{equation}\label{f_m_g}
m_g(u):=\leb\{v:\ g(u)=g(v)\},\quad u\in(0,1).
\end{equation}

\begin{remark}
If $g_1=g_2$ a.e., then $m_{g_1}=m_{g_2}$ a.e. Thus, $m_g$ is well-defined for any $g\in L_2^{\uparrow}$. 
\end{remark}

Let $\|A\|_{HS}$ be the Hilbert-Schmidt norm of a linear operator $A$ on $L_2$, that is, 
\[
  \|A\|^2_{HS}=\sum_{ i=1 }^{ \infty } \|Ae_i\|^2_{L_2},
\]
where $\{ e_i,\ i \in \N \}$ is an orthonormal basis of $L_2$. For $g \in \Li$ the number of distinct values of its unique modification from $D^{\uparrow}$ is denoted by $\#g$.

\begin{lemma}\label{lemma_connection_HS_with_int}
Let $g\in L_2^{\uparrow}$ and $m_g$ be defined by~\eqref{f_m_g}. Then
$$
\|\pr_g\|_{HS}^2=\int_0^1\frac{du}{m_g(u)}=\#g.
$$
In particular, $\|\pr_g\|_{HS}^2<\infty$ if and only if the c\`{a}dl\'{a}g modification of $g$ belongs to $\St$.
\end{lemma}

\begin{proof}
We take $\tilde{g}\in D^{\uparrow}$ such that $g=\tilde{g}$ a.e. and note that $\|\pr_g\|_{HS}^2=\#g$ follows from Lemma~6.1~\cite{Konarovskyi:2017:EJP}. Moreover, $\|\pr_g\|_{HS}^2<\infty$ if and only if $\tilde{g}\in\St$. Therefore, we only have to show that $\int_0^1\frac{du}{m_g(u)}=\#g$.

Let $\K_{\tilde{g}}$ and $F_{\tilde{g}}$ be defined as in the beginning of the present section. Then, obviously, $m_{\tilde{g}}(u)=0$ if and only if $u\in F_{\tilde{g}}$. 

If $\int_0^1\frac{du}{m_g(u)}<\infty$, then $\leb(F_{\tilde{g}})=0$ and
\begin{equation}\label{f_number_of_values}
\int_0^1\frac{du}{m_g(u)}=\sum_{I\in\K_{\tilde{g}}}\int_{I}\frac{du}{m_g(u)}=\#\K_{\tilde{g}},
\end{equation}
where $\#\K_{\tilde{g}}$ denotes the number of distinct intervals in $\K_{\tilde{g}}$. Since $\tilde{g}$ is c\`{a}dl\'{a}g, $\#\K_{\tilde{g}}<\infty$ and $\leb(F_{\tilde{g}})=0$, one can see that $F_{\tilde{g}}=\emptyset$. This implies that $\#\tilde{g}=\#\K_{\tilde{g}}$. Thus, $\#\tilde{g}\leq\int_0^1\frac{du}{m_g(u)}$ (including the trivial case $\int_0^1\frac{du}{m_g(u)}=+\infty$).

Next, if $\#\tilde{g}<\infty$, then $\tilde{g}\in\St$, and, consequently, $F_{\tilde{g}}=\emptyset$. This together with~\eqref{f_number_of_values} yield $\int_0^1\frac{du}{m_g(u)}\leq\#\tilde{g}$. The lemma is proved.
\end{proof}

%\begin{lemma}\label{lemma_lower_semicontinuity_of_pr}
%Let $\{g_n,\ n\geq 1\}\subset\St$ and $g_n\to g$ a.e. as $n\to\infty$. Then for each $h\in L_2$
%$$
%\|\pr_{g}h\|_{L_2}\leq\varliminf_{n\to 0}\|\pr_{g_n}h\|_{L_2}.
%$$
%\end{lemma}

%Alternative:

\begin{lemma}\label{lemma_lower_semicontinuity_of_pr}
For each $h\in L_2$ the map $g\mapsto\|\pr_gh\|_{L_2}$ from $L_2^{\uparrow}$ to $\R$ is lower semi-continuous, that is,
$$
\|\pr_{g}h\|_{L_2}\leq\varliminf_{n\to \infty}\|\pr_{g_n}h\|_{L_2},
$$
for each sequence $\{g_n,\ n\geq 1\}$ converging to $g$ in $L_2^{\uparrow}$.
\end{lemma}

\begin{proof}
We first note that it is enough to prove the lemma only for $g_n\to g=:g_0$ a.e., since every convergent sequence in $L_2$ contains a convergent a.e. subsequence. 

Let
$$
J:=\left\{x\in\R:\ \leb(g_n^{-1}(\{x\}))=0\ \mbox{for all}\ \ n\geq 0\right\}.
$$
Then $\leb(\R\setminus J)=0$, due to the countability of $\R\setminus J$. Thus, $J$ is dense in $\R$ and, consequently, we can choose an increasing sequence of finite subsets $J_k\subset J$, $k\geq 1$, such that $\bigcup_{x\in J_k}\left(x-\frac{1}{k},x+\frac{1}{k}\right)\supset [-k,k]$. Let $J_k=\{x_i^k,\ i\in[p_k]\}$ be ordered in an increasing way. For simplicity, we also set $x_0^k:=-\infty$ and $x_{p_k+1}^k:=+\infty$.
It is easily seen that for each $n\geq 0$ the sequence of $\sigma$-fields
$$
\s_n^k:=\sigma^{*}\left(\left\{g_n^{-1}([x_{i-1}^k,x_i^k)),\ i\in[p_k+1]\right\}\right),\quad k\geq 1,
$$ 
increases to $\sigma^{*}(g_n)$. Moreover, for all $n\geq 0$ and $k\geq 1$
$$
\E(h|\s_n^k)=\sum_{i=1}^{p_k+1}h_{I_{i,n}^k}\I_{I_{i,n}^k}\quad\mbox{a.e.},
$$
where $I_{i,n}^k:=g_n^{-1}([x_{i-1}^k,x_i^k))$, $h_{I_{i,n}^k}:=\frac{1}{\leb(I_{i,n}^k)}\int_{I_{i,n}^k}h(v)dv$ and $\E(\cdot|\cdot)$ denotes the conditional expectation on the probability space $([0,1],\B([0,1]),\leb)$. Thus, by Theorem~7.23~\cite{Kallenberg:2002} %?? 
and Remark~\ref{remark_notes_obout_proj}~(ii), for each $n\geq 0$
$$
\E(h|\s_n^k)\to\pr_{g_n}h\quad \mbox{in}\ \ L_2\ \ \mbox{as}\ \ k\to\infty. 
$$
In particular, for every $n\geq 0$ 
\begin{equation}\label{f_conv_for_k}
\sup_{k\geq 1}\|\E(h|\s_n^k)\|_{L_2}=\|\pr_{g_n}h\|_{L_2},
\end{equation}
since $\s_n^k$, $k\geq 1$, increases and $\E(h|\s_n^k)$ is the projection of $h$ in $L_2$ into the subspace of all $\s_n^k$-measurable functions.

Next, we fix $k\geq 1$ and $i\in[p_k+1]$ such that $\leb(I_{i,0}^k)>0$ and denote the ends of $I_{i,0}^k$ by $a$ and $b$, $a<b$. Then, using the monotonicity of the functions $g_n$, $n\geq 0$, the convergence of $\{g_n,\ n\geq 1\}$ to $g_0$ and the choice of $J_k$, we have that $a_n\to a$ and $b_n\to b$, where $a_n$ and $b_n$ are the ends of some intervals $I_{i_n,n}^k$. Consequently, for every $k\geq 1$
$$
\E(h|\s_n^k)\to\E(h|\s_0^k)\quad \mbox{a.e.\ \ as}\ \ n\to\infty. 
$$ 
By Fatou's lemma, for every $k\geq 1$
\begin{equation}\label{f_conv_for_n}
\|\E(h|\s_0^k)\|_{L_2}\leq\varliminf_{n\to \infty}\|\E(h|\s_n^k)\|_{L_2}. 
\end{equation} 
Hence,
\begin{align*}
\|\pr_{g_0}h\|_{L_2}&\stackrel{\eqref{f_conv_for_k}}{=}\sup_{k\geq 1}\|\E(h|\s_0^k)\|_{L_2}\stackrel{\eqref{f_conv_for_n}}{\leq}\sup_{k\geq 1}\varliminf_{n\to \infty}\|\E(h|\s_n^k)\|_{L_2} \stackrel{\eqref{f_conv_for_k}}{\leq}\sup_{k\geq 1}\varliminf_{n\to \infty}\|\pr_{g_n}h\|_{L_2}=\varliminf_{n\to \infty}\|\pr_{g_n}h\|_{L_2}.
\end{align*}
The lemma is proved.
\end{proof}

\subsection{Limit properties of some projection-valued functions}\label{subsection_limit_prop_of_proj_process}

We recall that $\HS(L_2)$ denotes the space of Hilbert-Schmidt operators on $L_2$ with the inner product defined by 
\[
  (A,B)_{HS}=\sum_{ i=1 }^{ \infty } (Ae_i,Be_i)_{L_2},
\]
where $\{ e_i,\ i \in \N \}$ is an orthonormal basis of $L_2$,  
and the space $\LHS$ is endowed with the inner product 
$$
(A,B)_{T,HS}=\int_0^T(A_t,B_t)_{HS}dt,\quad A,B\in\LHS.
$$
Since $\HS(L_2)$ is a Hilbert space, $\LHS$ also is a Hilbert space.

\begin{proposition}\label{proposition_conv_of_proj}
Let functions $f$ and $f^n$, $n\geq 1$, from $C([0,T],L_2^{\uparrow})$ satisfy the following conditions
\begin{enumerate}
\item[(a)] $\{P^n,\ n\geq 1\}$ converges weakly in $\LHS$ to $P^{\infty}$, that is, 
$$
(P^n,A)_{T,HS}\to (P^{\infty},A)_{T,HS}\quad\mbox{as}\ n\to\infty
$$
for any $A\in \LHS$, where $P^n_t=\pr_{f^n_t}$, $t\in[0,T]$;

\item[(b)] there exists $R\subseteq[0,T]$ such that $\leb([0,T]\setminus R)=0$ and $\|P_th\|_{L_2}\leq\varliminf_{n\to\infty}\|P^n_th\|_{L_2}$ for all $t\in R$ and $h\in L_2$, where $P_t=\pr_{f_t}$, $t\in[0,T]$;

\item[(c)] for every $h\in L_2$ and almost all $t\in[0,T]$ $P^{\infty}_t(P_th)=P^{\infty}_th$.
\end{enumerate}
Then $P^{\infty}=P$.
\end{proposition}

\begin{remark}\label{remark_prop_of_P}
\begin{enumerate}
\item[(i)] Condition $(a)$ together with the uniform boundedness principle imply the boundedness of the sequence $\{P^n,\ n\geq 1\}$ in $\LHS$.

\item[(ii)] The function $P$ belongs to $\LHS$ and $$\|P\|_{T,HS}\leq\varliminf_{n\to\infty}\|P^n\|_{T,HS},$$ by condition $(b)$, Fatou's lemma and the boundedness of $\{P^n,\ n\geq 1\}$. 

\item[(iii)] Since $P^n_t$ is an adjoint operator in $L_2$ for every $t\in[0,T]$, $P^{\infty}_t$ is also adjoint for almost all $t\in[0,T]$, by Corollary~\ref{corollary_equality_of_HS_valued_maps} below.
\end{enumerate}
\end{remark}

To prove the proposition, we need to prove some auxiliary statements.

\begin{lemma}\label{lemma_E_i_j_t_is_dense}
Let $\{e_i,\ i\in\N\}$ be an orthonormal basis of $L_2$ and $E^{i,j,r}_t=\I_{[0,r]}(t)e_i\otimes e_j$, $t\in[0,T]$, $i,j\in\N$, $r\in[0,T]$. Then $\spann\{E^{i,j,r},\ r\in[0,T],\ i,j\in\N\}$ is dense in $\LHS$.
\end{lemma}

\begin{proof}
The statement easily follows from the density of simple functions $\sum_{k=1}^n\I_{[t_{k-1},t_k)}A_k$ in $\LHS$, where $0=t_0<t_1<\ldots<t_n=T$ and $A_k\in\HS(L_2)$, $k\in[n]$, and the fact that $\{e_i\otimes e_j,\ i,j\in\N\}$ is an orthonormal basis of $\HS(L_2)$.
\end{proof}

\begin{corollary}\label{corollary_equality_of_HS_valued_maps}
Let $\{e_i,\ i\in\N\}$ be an orthonormal basis of $L_2$ and $A,B\in\LHS$. If for each $r\in[0,T]$ and $i,j\in\N$
$$
\int_0^r(A_te_i,e_j)_{L_2}dt=\int_0^r(B_te_i,e_j)_{L_2}dt,
$$
then $A=B$.
\end{corollary} 

\begin{proof}
The statement immediately follows from Lemma~\ref{lemma_E_i_j_t_is_dense} and the equality
$$
(A,E^{i,j,r})_{T,HS}=\int_0^r(A_te_i,e_j)_{L_2}dt.
$$
\end{proof}

\begin{proof}[Proof of Proposition~\ref{proposition_conv_of_proj}]
Let $e\in L_2([0,T],L_2)$ such that 
\begin{equation}\label{f_prop_for_e}
\|e_t\|_{L_2}=1\quad\mbox{and}\quad P_te_t=e_t\quad\mbox{for almost all}\  t\in[0,T].
\end{equation} 
We first prove that
\begin{equation}\label{f_prop_of_P_infty}
  (P^{\infty}_te_t,e_t)_{L_2}=1\quad\mbox{for almost all}\  t\in[0,T].
\end{equation}
To show this, we set for fixed $r\in[0,T]$ 
$$
A_t^r:=\I_{[0,r]}(t)e_t\otimes e_t,\quad t\in[0,T],
$$
and use the weak convergence of $P^n$ to $P^{\infty}$. We get
\begin{align*}
r=\int_0^r\|e_t\|_{L_2}^2dt&= \int_0^r\|P_te_t\|_{L_2}^2dt\stackrel{(b)}{\leq}\int_0^r\varliminf_{n\to\infty}\|P_t^ne_t\|_{L_2}^2dt\\
[\mbox{Fatou's lemma}]&\leq\varliminf_{n\to\infty}\int_0^r\|P_t^ne_t\|_{L_2}^2dt=\varliminf_{n\to\infty}\int_0^r(P^n_te_t,e_t)_{L_2}dt\\
&=\varliminf_{n\to\infty}(A^r,P^n)_{T,HS}\stackrel{(a)}{=}(A^r,P^{\infty})_{T,HS}=\int_0^r(P^{\infty}_te_t,e_t)_{L_2}dt.
\end{align*}
On the other hand, $(P_t^ne_t,e_t)_{L_2}=\|P_t^ne_t\|_{L_2}^2\leq \|e_t\|_{L_2}^2=1$ for all $t\in[0,T]$ and $n\geq 1$. Hence, 
$$
\int_0^r(P^{\infty}_te_t,e_t)_{L_2}dt=\lim_{n\to\infty}\int_0^r(P^n_te_t,e_t)_{L_2}dt\leq r.
$$
Consequently,
$$
\int_0^r(P^{\infty}_te_t,e_t)_{L_2}dt=r
$$
for all $r\in[0,T]$. This immediately implies~\eqref{f_prop_of_P_infty}.

Next, without loss of generality, we may suppose that $f_t\in D^{\uparrow}$ for all $t\in[0,T]$. We set for each $v\in(0,1)$
$$
e_t^v(u)=\frac{1}{\sqrt{m_{f_t}(v)}}\I_{\{f_t(v)=f_t(u)\}},\quad u\in(0,1),\ \ t\in[0,T],
$$
where $m_{f_t}$ is defined by~\eqref{f_m_g}. By Remark~\ref{remark_prop_of_P}~(ii), $\int_0^T\|P_t\|_{HS}^2dt<\infty$. Thus, $f_t\in\St$ for almost all $t\in[0,T]$, by Lemma~\ref{lemma_connection_HS_with_int}. This together with the right continuity of $f_t(u),\ u\in(0,1)$, imply that for every $v\in(0,1)$ the function $e^v_t$ is well-defined for almost all $t$ and $e^v\in L_2([0,T],L_2)$ . Let
$$
e_t^{v_1,v_2}:=\begin{cases}
               1,& f_t(v_1)=f_t(v_2),\\
               \frac{e_t^{v_1}+e_t^{v_2}}{\sqrt{2}},& f_t(v_1)\not=f_t(v_2),
              \end{cases}\quad t\in[0,T].
$$
It is easy to see that $e^{v_1,v_2}$ belong to $L_2([0,T],L_2)$ for all $v_1,v_2\in(0,1)$.
  
Since $e^{v_1,v_2}$ and $e^{v_1}$ satisfy~\eqref{f_prop_for_e} for all $v_1,v_2\in(0,1)$,
\begin{equation}\label{f_e^v_and_P}
  (P^{\infty}_te_t^{v_1,v_2},e_t^{v_1,v_2})_{L_2}=1\quad\mbox{and}\quad (P^{\infty}_te_t^{v_1},e_t^{v_1})_{L_2}=1\quad\mbox{for almost all}\  t\in[0,T].
\end{equation}
We set
\begin{align*}
  R&=\left\{t\in[0,T]:\ (P^{\infty}_te_t^{v_1,v_2},e_t^{v_1,v_2})_{L_2}=1 \ \mbox{and}\ (P^{\infty}_te_t^{v_1},e_t^{v_1})_{L_2}=1,\ v_1,v_2\in(0,1)\cap\Q\right\}\\
&\cap\left\{t\in[0,T]:\ P^{\infty}_t(P_t)=P^{\infty}_t \ \mbox{and}\ \|P_t\|_{HS}<\infty\}\cap\{t\in[0,T]:\ P^{\infty}_t\ \mbox{is adjoint}\right\}.
\end{align*}
Then $\leb([0,T]\setminus R)=0$, by~\eqref{f_e^v_and_P}, Condition~$(c)$ and Remark~\ref{remark_prop_of_P}~(ii),~(iii).

Next, we fix $t\in R$ and note that $f_t$ is a step function with a finite number of values, by Lemma~\ref{lemma_connection_HS_with_int}. Thus, there exists $v_1,\ldots,v_l$ from $(0,1)\cap\Q$, which depends on $t$, such that $l=\#f_t$ and $\{e_i:=e^{v_i}_t,\ i=1,\ldots,l\}$ is an orthonormal basis of the image of $P_t$. We extend $\{e_i,\ i=1,\ldots,l\}$ to an orthonormal basis of $L_2$ denoted by $\{e_i,\ i\in\N\}$ and note that $f_t(v_i)\not =f_t(v_j)$ for $i\not=j$, according to  the definition of $e^v$. By the choice of $t$, $(P_t^{\infty}e_i,e_i)_{L_2}=1$, $i=1,\ldots,l$. Moreover, $(P_t^{\infty}e_i,e_j)_{L_2}=0$ for all $i,j\in[l]$ and $i\not=j$. Indeed,
\begin{align*}
  1&=(P_t^{\infty}e_t^{v_i,v_j},e_t^{v_i,v_j})_{L_2}=\frac{1}{2}(P_t^{\infty}(e_i+e_j),e_i+e_j)_{L_2}\\
  &=\frac{1}{2}\big[(P_t^{\infty}e_i,e_i)_{L_2}+(P_t^{\infty}e_j,e_j)_{L_2}+2(P_t^{\infty}e_i,e_j)_{L_2}\big]=1+(P_t^{\infty}e_i,e_j)_{L_2}.
\end{align*}
If $i>l$, then 
$$
P_t^{\infty}e_i=P_t^{\infty}(P_te_i)_{L_2}=P_t^{\infty}0=0,
$$
by Condition~$(c)$.
This implies that $(P_t^{\infty}e_i,e_j)_{L_2}=(P_te_i,e_j)_{L_2}$ for all $i,j\in\N$. Thus, $P_t=P_t^{\infty}$. The proposition is proved.
\end{proof}

\subsection{Quadratic variations of L2-valued continuous semimartingales}\label{subsection_prop_of_semimartingales}
Let $(\Omega,\F,\p)$ be a complete probability space and $(\F_t)_{t\in[0,T]}$ be a complete right continuous filtration.  
\begin{proposition}\label{proposition_prop_of_quad_var_of_mart}
Let $g\in L_2^{\uparrow}$, $M_t$, $t\in[0,T]$, be a continuous $L_2$-valued square-integrable $(\F_t)$-martingale with quadratic variation
$$
\langle M\rangle_t=\int_0^tL_sL^{*}_sds, \quad t \in [0,T],
$$
where $L_t$, $t\in[0,T]$, is an $(\F_t)$-adapted $\HS(L_2)$-valued process belonging to $\LHS$ a.s. and $L^{*}_s$ denotes the adjoint operator of $L_s$. Let $b_t$, $t\in[0,T]$, be an $(\F_t)$-adapted $L_2$-valued continuous process such that for each $h\in L_2$ the process $(b_t,h)_{L_2}$, $t\in[0,T]$, has a locally finite variation. Also assume that the process
$$
X_t:=g+M_t+b_t,\quad t\in[0,T],
$$
takes values in $L_2^{\uparrow}$. Then 
$$
\p\left\{\exists R\subseteq[0,T]\ \mbox{s.t.}\ \ 
\leb\{[0,T]\setminus R\}=0\ \ \mbox{and}\  L_t(\pr_{X_t}h)=L_th,\ \forall t\in R,\ \forall h\in L_2
\right\}=1.
$$
\end{proposition}

To prove the proposition, we need the following lemma.

\begin{lemma}\label{lemma_prop_of_positiv_semimart}
Let $x(t)$, $t\in[0,T]$, be a continuous real valued semimartingale. Then 
$$
\int_0^T\I_{\{0\}}(x(t))d[x]_t=0\ \ \mbox{a.s.}
$$
\end{lemma} 

\begin{proof}
The statement immediately follows from the equality
$$
\int_0^T\I_{\{0\}}(x(t))d[x]_t=\int_{-\infty}^{+\infty}\I_{\{0\}}(y)l_T^{y}dy=0,
$$
where $l_t^{y}$, $t\in[0,T]$, $y\in\R$, is the local time of $x$ (see, e.g., Theorem~22.5~\cite{Kallenberg:2002}).
\end{proof}

\begin{proof}[Proof of Proposition~\ref{proposition_prop_of_quad_var_of_mart}]
We set  
\begin{equation}\label{f_h_a_b}
f_{a,b}:=\frac{1}{b-a}\I_{[a,b)}
\end{equation}
for each $a,b\in[0,1]$, $a<b$, and
$$
\RR:=\{f_{a,b}:\ a,b\in[0,1]\cap\Q,\ a<b\}.
$$
If $b_1\leq a_2$, then we will write $f_{a_1,b_1}\preccurlyeq f_{a_2,b_2}$.  

Taking $f',f''\in\RR$,  $f'\preccurlyeq f''$ and applying Lemma~\ref{lemma_prop_of_positiv_semimart} to the semimartingale
$$
x(t):=X_t(f'')-X_t(f')=X_t(f''-f'),\quad t\in[0,T],
$$
where $X_t(f):=(X_t,f)_{L_2}$, we obtain
\begin{align*}
0&=\int_0^T\I_{\{0\}}(X_t(f'')-X_t(f'))d[X(f''-f'')]_t =\int_0^T\I_{\{0\}}(X_t(f'')-X_t(f'))\|L_t(f''-f')\|_{L_2}^2dt\quad\mbox{a.s.}
\end{align*}
For each $\omega\in\Omega$, we set
\begin{align*}
R(\omega):=\left\{t\in[0,T]:\ \I_{\{0\}}(X_t(f'')(\omega)-X_t(f')(\omega))\|L_t(\omega)(f''-f')\|_{L_2}=0, \ \forall f',f''\in\RR,\ f'\preccurlyeq f''\right\}
\end{align*}
and
$$
\Omega'=\left\{\omega:\ \leb([0,T]\setminus R(\omega))=0\right\}.
$$
Since $\RR$ is countable, we have that $\p\{\Omega'\}=1$. Next, let $\omega\in\Omega'$ and $t\in R(\omega)$ be fixed. To finish the proof of the theorem, it is needed to show that 
\begin{equation}\label{f_equality_for_L_pr}
L_t(\omega)(\pr_{X_t(\omega)}h)=L_t(\omega)h
\end{equation}
for all $h\in L_2$. But since $C[0,1]$ is dense in $L_2$, the equality is enough to check only for $h\in C[0,1]$.  Therefore, we fix $h\in C[0,1]$ and denote the modification of $X_t(\omega)$ from $D^{\uparrow}$ also by $X_t(\omega)$. 

First, we take arbitrary $a<b$ from $(0,1)\cap\Q$ such that $X_t(a,\omega)=X_t(b,\omega)$ and show that 
\begin{equation}\label{f_equality_for_h_Q}
L_t(\omega)h=L_t(\omega)(h\I_{\pi^c}+h_{\pi}\I_{\pi}),
\end{equation}
where $h_{\pi}:=\frac{1}{b-a}\int_a^bh(u)du$, $\pi:=[a,b]$ and $\pi^c:=[0,1]\setminus\pi$. Let $a=u_0<u_1<\ldots<u_k=b$ be an arbitrary partition of $[a,b]$ with $u_i\in\Q$, $i\in[k]$. The monotonicity of $X_t(u,\omega)$, $u\in(0,1)$, yields that $X_t(f_{u_{i-1},u_i})(\omega)=X_t(f_{u_{j-1},u_j})(\omega)$ for all $i,j\in[k]$. Thus, we have that
$$
L_t(\omega)f_{u_{i-1},u_i}=L_t(\omega)f_{u_{j-1},u_j},\ i,j\in[k],
$$ 
due to the choice of $t$ and $\omega$, where $f_{u_{i-1},u_i}$, $f_{u_{j-1},u_j}$ are defined by~\eqref{f_h_a_b}. Using the equality $f_{a,b}=\sum_{i=1}^k\frac{u_i-u_{i-1}}{b-a}f_{u_{i-1},u_i}$, one can easily seen that
\begin{equation}\label{f_Lf_and_Lf_a_b}
L_t(\omega)f_{u_{i-1},u_i}=L_t(\omega)f_{a,b}
\end{equation}
for all $i\in[k]$. Taking 
$$
h^k:=h\I_{\pi^c}+\sum_{i=1}^kh(u_{i-1})\I_{[u_{i-1},u_i)}
$$
and using~\eqref{f_Lf_and_Lf_a_b}, we obtain
\begin{align*}
L_t(\omega)h^k&=L_t(\omega)(h\I_{\pi^c})+\sum_{i=1}^kh(u_{i-1})L_t(\omega)\I_{[u_{i-1},u_i)} =L_t(\omega)(h\I_{\pi^c})+\sum_{i=1}^kh(u_{i-1})(u_i-u_{i-1})L_t(\omega)f_{u_{i-1},u_i}\\
&=L_t(\omega)(h\I_{\pi^c})+\sum_{i=1}^kh(u_{i-1})(u_i-u_{i-1})L_t(\omega)f_{a,b}=L_t(\omega)\tilde{h}^k,
\end{align*}
where
$$
\tilde{h}^k=h\I_{\pi^c}+\frac{1}{b-a}\sum_{i=1}^kh(u_{i-1})(u_i-u_{i-1})\I_\pi.
$$
Since $h^k\to h$ and $\tilde{h}^k\to h\I_{\pi^c}+h_{\pi}\I_{\pi}$ in $L_2$ as $\max_{i\in[k]}(u_i-u_{i-1})\to 0$, the equality~\eqref{f_equality_for_h_Q} holds.

Next, by the approximation argument, it is easy to prove that~\eqref{f_equality_for_h_Q} folds for each $\pi\in\K_{X_t(\omega)}$, where $\K_g$ was defined in Section~\ref{section_properties_of_pr} for any $g\in D^{\uparrow}$. Let $\K_{X_t(\omega)}=\{\pi_i,\ i\in\N\}$, that is countable, be ordered %?? 
in decreasing of the length of $\pi_i$. If $\K_g$ is finite then~\eqref{f_equality_for_L_pr} immediately follows from~\eqref{f_equality_for_h_Q} and Lemma~\ref{lemma_view_of_pr}. Otherwise, using~\eqref{f_equality_for_h_Q}, the continuity of $L_t(\omega)$ and Lemma~\ref{lemma_view_of_pr}, we have
\begin{align*}
  L_t(\omega)h&=L_t(\omega)\left(h\I_{\tilde{\pi}_l}+\sum_{i=1}^{l}h_{\pi_i}\I_{\pi_i}\right)\ \ \to \ \ L_t(\omega)\left(h_{X_t(\omega)}\right)=L_t(\omega)\left(\pr_{X_t(\omega)}h\right)\quad\mbox{as}\ \ l\to\infty,
\end{align*}
where $h_g$ is defined by~\eqref{f_explicit_formula_of_pr} and $\tilde{\pi}_l:=[0,1]\setminus\left(\bigcup_{i=1}^l\pi_i\right)$. This completes the proof of the proposition.
\end{proof}

\subsection{Some compact sets in Skorohod space}\label{section_compact_sers_in_D}

Let $(E,r)$ be a Polish space and let $D([a,b],E)$ denote the space of c\`{a}dl\'{a}g functions from $[a,b]$ to $E$ which are continuous at $b$. We endow $D([a,b],E)$ with the metric
$$
d_{[a,b]}(f,g)=\inf_{\lambda\in\Lambda_{[a,b]}}\left\{\gamma(\lambda)\vee \sup_{u\in[a,b]}r(f(\lambda(u)),g(u))\right\}, \quad f,g\in D([a,b],E),
$$
where $\Lambda_{[a,b]}$ is the set of all strictly increasing functions $\lambda:[a,b]\to[a,b]$ such that $\lambda(a)=a$, $\lambda(b)=b$ and
$$
\gamma(\lambda):=\sup_{v<u}\left|\log\frac{\lambda(u)-\lambda(v)}{u-v}\right|<\infty.
$$

For each $[c,d]\subset[a,b]$ and $f\in D([a,b],E)$ it is clear that the function 
$$
f^{[c,d]}(u):=\begin{cases}
               f(u),& u\in[c,d),\\
               f(d-),& u=d.
              \end{cases}
$$
belongs to $D([c,d],E)$.

\begin{proposition}\label{proposition_tightnes_in_D}
Let $U=\{u_i,\ i=0,\ldots,l\}$ be an ordered partition of $[a,b]$ and let $\{X_n,\ n\geq 1\}$ \ be an \ arbitrary sequence of random elements in $D([a,b],E)$. If \ $\left\{X_n^{[u_{i-1},u_i]},\ n\geq 1\right\}$ is tight in $D([u_{i-1},u_i],E)$ for any $i\in[l]$, then $\{X_n,\ n\geq 1\}$ is tight in $D([a,b],E)$.
\end{proposition}

\begin{proof}
Let $K_i$ be compact in $D([u_{i-1},u_i],E)$, $i\in[l]$, and let
\begin{equation}\label{f_set_K}
K:=\left\{f\in D([a,b],E):\ f^{[u_{i-1},u_i]}\in K_i,\ i\in[l]\right\}.
\end{equation}
In order to prove the proposition, it is enough to show that $K$ is compact in $D([a,b],E)$. Indeed, by the definition of the tightness (see, e.g., Section~3.2~\cite{Ethier:1986}), for each $\eps>0$ there exist compact sets $K_i$, $i\in[l]$, such that
$$
\p\left\{X_n^{[u_{i-1},u_i]}\not\in K_i\right\}\leq \frac{\eps}{l}
$$
for all $i\in[l]$ and $n\geq 1$. Thus,
$$
\p\{X_n\not\in K\}=\p\left\{\bigcup_{i=1}^l \left\{X_n^{[u_{i-1},u_i]}\not\in K_i\right\}\right\}\leq \eps,
$$
where $K$ is defined by~\eqref{f_set_K}. This implies the tightness of $\{X_n,\ n\geq 1\}$ in $D([a,b],E)$.

Let $\{f_n,\ n\geq 1\}\subset K$. Then there exists a subsequence $N\subset\N$ such that $f_n^{[u_{i-1},u_i]}$ converges to $f^i$ in $D([u_{i-1},u_i],E)$ along $N$ for any $i\in[l]$. Thus, for every $i\in[l]$ there exists a sequence $\{\lambda_n^i,\ n\in N\}\subset\Lambda_{[u_{i-1},u_i]}$ such that 
$$
\gamma(\lambda_n^i)\to 0\quad\mbox{and}\quad \sup_{u\in[u_{i-1},u_i]}r\left(f_n^{[u_{i-1},u_i]}(\lambda_n^i(u)),f^i(u)\right)\to 0 \quad\mbox{along}\ \  N.
$$
Taking
$$
f(u):=\sum_{i=1}^lf^i(u)\I_{[u_{i-1},u_i)}(u)+f^l(u)\I_{\{u_l\}}(u)
$$
and
$$
\lambda_n:=\sum_{i=1}^l\lambda_n^i(u)\I_{[u_{i-1},u_i)}(u)+\lambda_n^l(u)\I_{\{u_l\}}(u),\quad n\geq 1,
$$
it is easily seen that $f\in D([a,b],E)$ and $\lambda_n$ is a continuous strictly increasing function from $[a,b]$ onto $[a,b]$ for all $n\geq 1$. Moreover,
$$
\sup_{u\in[a,b]}|\lambda_n(u)-u|\to 0\quad\mbox{and}\quad \sup_{u\in[a,b]}r(f_n(\lambda_n(u)),f(u))\to 0 \quad\mbox{along}\ \  N.
$$
By Theorem~12.1~\cite{Billingsley:1999}, $f_n$ converges to $f$ in $D([a,b],E)$ along $N$. The proposition is proved.
\end{proof}

%%%%%%%%%%%%%%%%%%%%%%%%%%%%%%%%%%%%%%%%%%%%%%
%% Single Appendix:                         %%
%%%%%%%%%%%%%%%%%%%%%%%%%%%%%%%%%%%%%%%%%%%%%%
%\begin{appendix}
%\section*{???}%% if no title is needed, leave empty \section*{}.
%\end{appendix}
%%%%%%%%%%%%%%%%%%%%%%%%%%%%%%%%%%%%%%%%%%%%%%
%% Multiple Appendixes:                     %%
%%%%%%%%%%%%%%%%%%%%%%%%%%%%%%%%%%%%%%%%%%%%%%
%\begin{appendix}
%\section{???}
%
%\section{???}
%
%\end{appendix}

%%%%%%%%%%%%%%%%%%%%%%%%%%%%%%%%%%%%%%%%%%%%%%
%% Support information, if any,             %%
%% should be provided in the                %%
%% Acknowledgements section.                %%
%%%%%%%%%%%%%%%%%%%%%%%%%%%%%%%%%%%%%%%%%%%%%%
\begin{acks}[Acknowledgments]
The author is grateful to Max von Renesse for useful discussions and suggestions. The author also thanks the anonymous referees for their careful reading of the manuscript and many valuable comments which improved the presentation of the results.
\end{acks}
%%%%%%%%%%%%%%%%%%%%%%%%%%%%%%%%%%%%%%%%%%%%%%
%% Funding information, if any,             %%
%% should be provided in the                %%
%% funding section.                         %%
%%%%%%%%%%%%%%%%%%%%%%%%%%%%%%%%%%%%%%%%%%%%%%
\begin{funding}
The research was partly supported by Alexander von Humboldt Foundation and partly supported by the Deutsche Forschungsgemeinschaft (DFG, German Research Foundation) – SFB 1283/2 2021 – 317210226.
% The first author was supported by ...
%
% The second author was supported in part by ...
\end{funding}

%%  imsart-number.bst  will be used to                     %%
%%  MR numbers will be added by VTeX.                      %%
%%                                                         %%
%%  Use \cite{...} to cite references in text.             %%
%%                                                         %%
%%%%%%%%%%%%%%%%%%%%%%%%%%%%%%%%%%%%%%%%%%%%%%%%%%%%%%%%%%%%%

%% if your bibliography is in bibtex format, uncomment commands:
%\bibliographystyle{imsart-number} % Style BST file
%\bibliography{ref}       % Bibliography file (usually '*.bib')

%% or include bibliography directly:
% \begin{thebibliography}{}
% \bibitem{b1}
% \end{thebibliography}

\end{document}